\documentclass{amsart}
\usepackage{amssymb}
\usepackage{latexsym}
\usepackage{euscript}
\usepackage{graphics}
\usepackage[usenames,dvipsnames,svgnames,table]{xcolor}

\def\sfR{{\sf R}}

\def\sfS{{\text{\rm{\textsf S}}}}

\def\bK{{\mathbf K}}

\def\sD{{\mathfrak D}}      
   \def\sH{{\mathfrak H}}   
      \def\sL{{\mathfrak L}}
\def\sM{{\mathfrak M}}   \def\sN{{\mathfrak N}}   
      
   \def\sT{{\mathfrak T}}   
\def\sV{{\mathfrak V}}      %\def\sX{{\mathfrak X}}
   
\def\gh{{\mathfrak h}}

      \def\dC{{\mathbb C}}
\def\dD{{\mathbb D}}

   \def\dT{{\mathbb T}}

   \def\cB{{\mathcal B}}   
\def\cD{{\mathcal D}}      
   \def\cH{{\mathcal H}}   
\def\cJ{{\mathcal J}}      
   \def\cN{{\mathcal N}}   \def\cO{{\mathcal O}}
      
\def\cS{{\mathcal S}}      \def\cU{{\mathcal U}}

\def\gh{{\mathfrak h}}

\def\wt#1{{{\widetilde #1} }}
\def\h#1{{{\hat #1} }}
\def\wh#1{{{\widehat #1} }}

\def\bm\chi{\mbox{\boldmath$\chi$}}
\def\gr{{\rm gr\,}}

\def\ker{{\rm ker\,}}
\def\ran{{\rm ran\,}}
\def\cran{{\rm \overline{ran}\,}}
\def\dom{{\rm dom\,}}
\def\mul{{\rm mul\,}}

\def\cdom{{\rm \overline{dom}\,}}

\def\dim{{\rm dim\,}}

\def\card{{\rm card\,}}
\let\xker=\ker \def\ker{{\xker\,}}
\def\spn{{\rm span\,}}

\def\ind{{\rm ind\,}}

\DeclareMathOperator{\hplus}{\, \widehat + \,}

\newcommand {\sk}[3]{\left#1#2\right#3}  %скобки
\renewcommand {\l}{\lambda}
\renewcommand {\k}{\kappa}
\newcommand {\s}{\sigma}

\newcommand {\e}{\varepsilon}
\newcommand {\T}{\Theta}
\newcommand {\g}{\gamma}

\newcommand {\ov}{\overline}
\newcommand {\m}{\mu}
\newcommand {\G}{\Gamma}
\renewcommand {\r}{\rho}

\newcommand {\zx}{{[*]}}
\newcommand {\sx}{{-[*]}}

\newtheorem{theorem}{Theorem}[section]
\newtheorem{proposition}[theorem]{Proposition}
\newtheorem{corollary}[theorem]{Corollary}
\newtheorem{lemma}[theorem]{Lemma}

\theoremstyle{definition}
\newtheorem{example}[theorem]{Example}
\newtheorem{remark}[theorem]{Remark}
\newtheorem{definition}[theorem]{Definition}

\numberwithin{equation}{section}

\usepackage[usenames,dvipsnames,svgnames,table]{xcolor}

\begin{document}

\title[Unitary boundary pairs for  isometric operators ]
{Unitary boundary pairs for  isometric operators in  Pontryagin spaces and generalized coresolvents}
\author{D.~Baidiuk}
\author{V.~Derkach}
\author{S.~Hassi}

\address{Department of Mathematics \\
Tampere University \\
%P.O. Box 553, 33101
Tampere \\
Finland} \email{baydyuk@gmail.com}

\address{Department of Mathematics and Statistics \\
University of Vaasa \\
%P.O. Box 700,
%65101 Vaasa \\
Finland} \email{sha@uwasa.fi}

\address{
Department of Mathematics and Natural Sciences, TU Ilmenau, %98693 Ilmenau,
Germany }

\address{ Department of Mathematics\\
Vasyl Stus Donetsk National University\\
 %600-Richchya Str 21, 21021,
 Vinnytsya,  Ukraine} \email{derkach.v@gmail.com}

\dedicatory{Dedicated to our friend and colleague Henk de Snoo on the occasion of his
75th birthday}

\thanks{
The research of V.D. was supported by Ministry of Education and Science of Ukraine (projects \# 0118U003138, 0118U002060)
and by the German Research Foundation (DFG), grant TR 903/22-1. %TR 903/16-1.
V.D. also thanks for hospitality the University of Vaasa where part of this study was conducted.
The research of D.B. was funded by the Academy of Finland grant number 310489.\\
\textbf{Data Availability Statement.}
Data sharing not applicable to this article as no datasets were generated or analysed during the current study.
}



\begin{abstract}
An isometric operator $V$ in a Pontryagin space ${\sH}$ is called
standard, if its domain and the range are nondegenerate subspaces in
${\sH}$. A description of coresolvents for standard isometric operators is known
and basic underlying concepts that appear in the literature are unitary colligations
and characteristic functions.
In the present paper generalized coresolvents of non-standard Pontryagin space
isometric operators are described.
The methods used in this paper rely on a new general notion of boundary pairs
introduced for isometric operators in a Pontryagin space setting.
Even in the Hilbert space case this notion generalizes the earlier concept of
boundary triples for isometric operators and offers an alternative approach to
study operator valued Schur functions without any additional invertibility requirements
appearing in the ordinary boundary triple approach.
\end{abstract}

\maketitle

%\tableofcontents

\section{Introduction}
Extension theory for standard symmetric and isometric operators in  Pontryagin spaces was first developed by I.S. Iokhvidov and M.G. Kre\u{\i}n in \cite{IK65}, generalized
resolvents of such operators were described by  M.G. Kre\u{\i}n and H. Langer in \cite{KL71,KL72,L71}. Following~\cite{DLS90} we will use the notion {\it standard} for an isometric operator $V$ in a Pontryagin space ${\sH}$, if its domain $\dom V$ and the
range $\ran V$ are nondegenerate subspaces in ${\sH}$. In this case every unitary extension of $V$ can be obtained in pretty much the same way as in the case of Hilbert space isometric operator.
Similarly, the extension theory and the theory of generalized coresolvents of standard isometric operators in Kre\u{\i}n spaces was built by A. Dijksma, H. Langer and H. de Snoo in~\cite{DLS90}. For a nonstandard isometric operator in a Pontryagin space, description of its regular (resp. nonregular) generalized coresolvents in Pontryagin spaces without growth (resp. with growth) of negative index  was given by P. Sorjonen~\cite{Sor85} (resp. by O. Nitz \cite{N001}, \cite{N002}). However, the proof in~\cite{N002}  is not so convincing, as it becomes quite complicated and contains some gaps.

Another approach to the extension theory of symmetric operators in Hilbert spaces is based on the concept of \textit{abstract boundary value} introduced by J. Calkin~\cite{Cal39} and later formalized in the notion of  boundary value space  in~\cite{Koc75}, \cite{GG84} (or ordinary {\it boundary triple} in~\cite{DM95}). In~\cite{DM87} with each boundary triple there was associated an analytic object -- {\it abstract Weyl function} which allows to carry out spectral analysis of  extensions of symmetric operators. In the case of a Hilbert space isometric operator (and more generally for a dual pair of operators) the notions of a boundary triple and a corresponding Weyl function were introduced in~\cite{MM03,MM04}. These notions, when generalized to the indefinite case in~\cite{B13}, proved to be an adequate language in the extension theory of nonstandard isometric operator $V$ in a Pontryagin space, since they allowed to give full description of generalized coresolvents of $V$.
However, the method proposed in~\cite{B13} is restricted to the case of regular generalized coresolvents, which have minimal realizations in Pontryagin spaces $\wt \sH$ with the same negative index as $\sH$, and does not work for generalized coresolvents of $V$ which have minimal realizations in Pontryagin spaces $\wt \sH$ with bigger negative indices.

This difficulty can be prevented by using an appropriate notion of \textit{boundary pairs}, which extend the concept of ordinary boundary triples.
In the case of symmetric operators in Hilbert spaces an extension of ordinary boundary triples, a so-called \textit{generalized  boundary triple}, was introduced and studied in~\cite{DM95}. This notion was further generalized in~\cite{DHMS06} to the notion of a \textit{unitary boundary pair} (called therein as a \textit{boundary relation}), which can be applied to study generalized resolvents of symmetric operators \cite{DHMS09,DHMS12} and various general classes of boundary value problems for ordinary and partial differential operators, see~\cite{DHMS12,DHM17,DHM2020a,DHM2020b}.
In particular, in~\cite{DHMS06} it was shown that every Nevanlinna pair (or Nevanlinna family of holomorphic relations) can be realized as the Weyl family of some unitary boundary pair, and in~\cite{DHMS1,DHMS09} this notion was used to get a new proof of M.G. Kre\u{\i}n formula for generalized resolvents of symmetric operators via the coupling method developed therein.
In~\cite{BDHS11} the notion of unitary boundary pair was introduced for symmetric operators in Pontryagin spaces and it was shown that every generalized Nevanlinna pair, allowing a finite negative index for the associated Nevanlinna kernel, can be realized as the Weyl family of such
a unitary boundary pair.

In this paper a new notion of a unitary boundary pair with an associated Weyl function is introduced and studied in the setting of isometric operators in Pontryagin spaces. In particular, it is shown in Section \ref{sec3} how a certain subclass of unitary boundary pairs is connected to unitary colligations (see~\cite{ADRS97}) and, moreover, that the Weyl functions associated with that subclass of unitary boundary pairs actually coincide with characteristic functions of the corresponding unitary colligations; see Theorems \ref{T:UV},~\ref{T:UV2}. Furthermore, using some transformations results, being motivated by~\cite{ADRS97}, it is also shown that every operator valued generalized Schur function (not necessarily holomorphic at the origin) can still be realized as the Weyl function of some unitary boundary pair for an isometric operator $V$ in a Pontryagin space; see Theorems~\ref{thm:General},~\ref{thm:Inv}. These two theorems show that the present notion of a unitary boundary pair for isometric operators in a Pontryagin (as well as in the classical Hilbert) space setting is a natural object to realize and study generalized (or standard) Schur functions as their Weyl functions. In particular, these new notions complement and extend the approach, which relies on characteristic functions of unitary colligations being associated with the special case stated in Theorem \ref{T:UV2}.

After these characteristic results on unitary boundary pairs and their Weyl function for isometric operators we study in Section \ref{sec4} some spectral properties of proper extensions of $V$ and find a formula for their canonical coresolvents, see Theorem \ref{res1}, and then with these preparations prove an analog of M.G. Kre\u{\i}n formula for the generalized coresolvents of the isometric operator $V$. This latter problem is solved via the coupling method, where we consider a coupled unitary boundary pair as a direct sum of an ordinary boundary triple and a unitary boundary pair and then derive the formula for generalized coresolvents from the formula for canonical coresolvents associated with the coupled boundary pair.

\section{Preliminaries}
\label{sec1}
\subsection{Indefinite inner product spaces}
A linear space $\sH$ endowed with an inner product $[\cdot,\cdot]_{\sH}$ is called an inner product space, see~\cite{Bognar}, \cite{AI86}.
A vector $f\in\sH$ is called positive (resp. negative or neutral), if $[f,f]_{\sH}>0$  (resp. $[f,f]_{\sH}<0$ or  $[f,f]_{\sH}=0$).
A subspace ${\sL}\subset\sH$ is called positive  (resp. negative   or neutral), if every vector $f\in\sL\setminus\{0\}$ is positive (resp. negative or neutral). The orthogonal complement of a subspace ${\sL}$ is denoted by ${\sL}^{[\perp]}$.

An  inner product space $\sH$ is called a \textit{Kre\u{\i}n space}, if it admits a fundamental decomposition
\begin{equation}\label{eq:FundDecom}
    \sH=\sH_+[+]\sH_-
\end{equation}
as an orthogonal sum of a positive  subspace $\sH_+$ and a negative  subspace $\sH_-$. The operator $J=P_+-P_-$, where $P_{\pm}$ are orthogonal projections in $\sH$ onto $\sH_\pm$, is called {\it the fundamental symmetry} of $\sH$. We will use the notation $(\sH,J)$ for the Kre\u{\i}n space $\sH$ with the fundamental symmetry $J$. A Kre\u{\i}n space $(\sH,J)$ with a finite negative index $\kappa_-(\sH):=\dim\sH_-$ is called a {\it Pontryagin space}.

Every closed subspace ${\sD}$ of a Pontryagin space $(\sH,J)$ admits
the following decomposition
\begin{equation}\label{decoclosed}
 {\sD}={\sD}_0\dot{+} {\sD}_+ \dot{+} {\sD}_-,
\end{equation}
where ${\sD}_0={\sD}\cap {\sD}^{[\perp]}$ is a neutral subspace (the
isotropic part of ${\sD}$) and ${\sD}_+$ and ${\sD}_-$ are closed
(uniformly) positive and negative subspaces of $(\sH,J)$; see e.g.
\cite[Theorems~IX.2.5]{Bognar}. We will need the following slightly
modified version of this statement.

\begin{lemma}\label{lem:2.1}
Every linear subspace $\sT$ of a Pontryagin space $(\sH,J)$ admits
the following decomposition
\begin{equation}\label{Tdom0}
 \sT=\sT_+ \dot{+} \sT_1,
\end{equation}
where  $\sT_+$ is a positive  subspace of $(\sH,J)$, such
that $\overline{\sT_+}$ is a maximal positive  subspace of
$\overline{\sT}$, and ${\sT_1}$ is a $k$-dimensional subspace of
$(\sH,J)$, where $k=\dim  \overline{\sT}/ \overline{\sT_+}$.
\end{lemma}
\begin{proof}
Let ${\sD}$ be the closure of $\sT$ in $\sH$ and decompose $\sD$ as
in \eqref{decoclosed},
\[
 {\sD}={\sD}_0\dot{+} {\sD}_- \dot{+} {\sD}_+,
\]
where ${\sD}_0={\sD}\cap {\sD}^{[\perp]}$, ${\sD}_-$, and ${\sD}_+$
are closed neutral, negative, and positive
subspaces of $(\sH,J)$, respectively. Since $\sT$ is a dense subset
of ${\sD}$ and the subspaces ${\sD}_0$ and ${\sD}_-$ are finite
dimensional, $\sT$ has a dense intersection with ${\sD}_+$,
\begin{equation}\label{Tdom1}
 \overline{\sT \cap {\sD}_+}={\sD}_+;
\end{equation}
see e.g. \cite[Lemma~2.1]{GK}. Denote ${\sT}_+:=\sT \cap {\sD}_+$
and let $k=\dim({\sD}_0\dot{+} {\sD}_-)$. Since $\overline{\sT}=\sD$
one concludes that there exists a $k$-dimensional subspace
${\sT}_1\subset \sT \setminus \sD_+$. The closed subspace
${\sT}_1\subset \sT$ decomposes $\sT$ and \eqref{Tdom1} together
with a dimension argument leads to
\[
 \overline{{\sT}_+ \dot{+} {\sT}_1}=  \overline{{\sT}_+} \dot{+}{\sT}_1=
 \sD_+ \dot{+} {\sT}_1=\sD=\overline{\sT}.
\]
The equality $\overline{\sT}={\sT}_1 \dot{+} \overline{{\sT}_+}$
combined with \eqref{Tdom1} yields the decomposition \eqref{Tdom0}
for $\sT$.
\end{proof}

\subsection{Linear relations in Kre\u{\i}n spaces}

Let $(\sH_1,J_{\sH_1})$ and $(\sH_2,J_{\sH_2})$ be two  Kre\u{\i}n spaces.
A linear relation $T$ from $\sH_1$ to $\sH_2$ is a linear subspace
of $\sH_1 \times \sH_2$, see e.g. \cite{Arens}.
%in what follows we also use a brief alternative notation $T:\sH_1 \to \sH_2$.
Often a linear operator $T:\sH_1\to\sH_2$ will be identified with its graph
\[
\textup{gr }T:=\{\,\{f,Tf\}:\,f\in \dom T\,\}.
\]

For a linear relation $T$ from $\sH_1$ to $\sH_2$  the symbols $\dom T$, $\ker T$, $\ran T$, and $\mul T$ stand for the domain, kernel, range,
and multivalued part, respectively. The inverse $T^{-1}$ is a
relation from $\sH_2$ to $\sH_1$ defined by
$\{\,\{f',f\}:\,\{f,f'\}\in T\,\}$.
Denote by $\rho(T)$ the resolvent set of $T$,
%by $\wh\r(V)$ the set of regular type points,
 by $\sigma(T)$ the spectrum  of $T$  and by $\sigma_p(T)$  (resp. $\sigma_c(T)$, $\sigma_r(T)$) the point  (resp. continuous, residual)   spectrum of $T$.
The adjoint $T^{[*]}$ is the closed linear relation from $\sH_2$ to
$\sH_1$ defined by (see~\cite{Ben72})
\begin{equation}
\label{adjo}
 T^{[*]}=\{\,\{h,k\} \in \sH_2 \times \sH_1 :\,
     [k,f]_{\sH_1}=[h,g]_{\sH_2}, \, \{f,g\}\in T \,\}.
\end{equation}
The following equalities are obvious from~\eqref{adjo}
\begin{equation}
\label{adjequ}
 (\dom T)^{[\perp]}=\mul T^{[*]}, \quad
 (\ran T)^{[\perp]}=\ker T^{[*]}.
\end{equation}

A linear relation $T$ from $\sH_1$ to $\sH_2$ is called
\textit{isometric} (resp. \textit{contractive} or \textit{expanding}), if for every
$\{f,g\}\in T$ one has
 \begin{equation}\label{eq:isom}
    [g,g]_{\sH_2}=[f,f]_{\sH_1} \quad\left(\textup{ resp. }\, [g,g]_{\sH_2}\le[f,f]_{\sH_1}\textup{ or }\,
    [g,g]_{\sH_2}\ge[f,f]_{\sH_1}\right).
 \end{equation}
It follows from~\eqref{adjo} and~\eqref{eq:isom} that $T$ is
isometric, if $T^{-1}\subseteq T^{[*]}$. A linear relation $T$ from
$\sH_1$ to $\sH_2$ is called  \textit{unitary}, if $T^{-1}=
T^{[*]}$, \cite{Sh76}. Moreover, $T$ is said to be a \textit{standard
unitary operator} if $\dom T=\sH_1$ and $\ran T=\sH_2$. For an
isometric linear relation $T$ one obtains from $T^{-1}\subseteq
T^{[*]}$ and the identities \eqref{adjequ} that
\begin{equation}
\label{isome}
 \ker T \subseteq  (\dom T)^{[\perp]},
 \quad
 \mul T \subseteq  (\ran T)^{[\perp]}.
\end{equation}

For a unitary linear relation the following statements hold, see~\cite[Theorem 2]{Sh76}.
\begin{proposition}
\label{UNIT}
Let $T$ be a unitary relation from the Kre\u{\i}n space
$(\sH_1,J_{\sH_1})$ to the Kre\u{\i}n space $(\sH_2,J_{\sH_2})$.
Then:
\begin{enumerate}
\def\labelenumi{\rm (\roman{enumi})}
\item $\dom T$ is closed if and only if $\ran T$ is closed;
\item the following equalities hold:
\begin{equation}
\label{eee}
\ker T=(\dom T)^{[\perp]},\quad \mul T=(\ran T)^{[\perp]}.
\end{equation}
\end{enumerate}
\end{proposition}
Denote
\begin{equation}\label{DDe}
 \dD=\{\lambda\in\dC:\, |\lambda|<1 \}, \quad  \dD_e=\{\lambda\in\dC:\, |\lambda|>1 \},
 \quad  \dT=\{\lambda\in\dC:\, |\lambda|=1 \}.
\end{equation}
If $V$ is a single-valued closed isometric operator in a Pontryagin space $\sH$
then the subspaces $\ran(V-\lambda I)$ are closed for every $\lambda\in\dD\cup\dD_e$, see e.g.
\cite[Section 1.3]{LS74} or Lemma~\ref{PontryIsom} below,
and each of the sets $\sigma_p(V)\cap\dD$ and $\sigma_p(V)\cap\dD_e$ consist of at most $\kappa=\kappa_-(\sH)$ eigenvalues,
see e.g. \cite[p.49 Corollary 2]{IKL}.
Denote by  $\sN_\l$ the \textit{defect subspace} of $V$:
$$
\sN_\l:=\sH[-]\ran(I-\overline{\lambda} V), \quad \lambda\in\dD\cup\dD_e.
$$
Then
$$
\sN_\l=\ker(V^{-[*]}-\lambda I)=\{f_\l:(f_\l,\l f_\l)^T\in V^{-\zx}\}.
$$
As is known, see  \cite[Theorem 6.1]{IKL},  the numbers $\dim\sN_\l$ take a constant value $n_+(V)$ for all $\lambda\in\dD\setminus{\overline{\sigma_p(V)}}^{-1}$, and $n_-(V)$ for all  $\lambda\in\dD_e\setminus{\overline{\sigma_p(V)}}^{-1}$.
The numbers $n_\pm(V)$ are called the \textit{defect numbers} of $V$.

\begin{definition}\label{Def:simple}
The isometric operator $V$ in ${\sH}$ is called {\it simple}, if $\sigma_p(V)\setminus\dT=\emptyset$ and
$$\ov\spn\{\sN_\l:\l\in\dD\cup\dD_e\}={\sH}.$$
\end{definition}

In the case of Pontryagin spaces some further results on isometric
and unitary relations can be established. For any isometric relation
$T$ between two Kre\u{\i}n spaces it is clear that $\ker T$ and
$\mul T$ are neutral subspaces. Therefore, in a Pontryagin space
$\ker T$ and $\mul T$ are necessarily finite dimensional. If $T$ is
closed then $\ker T$ and $\mul T$ are also closed. In Pontryagin
spaces the following stronger result is true.

\begin{lemma}\label{PontryIsom}
Let $T$ be a closed isometric relation from the Pontryagin space
$(\sH_1,J_{\sH_1})$ to the Pontryagin space $(\sH_2,J_{\sH_2})$.
Then the domain and the range of $T$ are closed.
\end{lemma}

\begin{proof}
The isometry of $T$ means that $T^{-1}\subseteq  T^{[*]}$. Taking
inverses one gets $T=(T^{-1})^{-1}\subseteq  T^{-[*]}$, i.e., $T$ and
$T^{-1}$ are simultaneously isometric. Therefore, to prove the
statement it suffices to prove that the range of $T$ is a closed
subspace in $\sH_2$, since $T$ is closed precisely when its inverse
$T^{-1}$ is closed and clearly $\dom T=\ran T^{-1}$.

Now let $\dom T$ be decomposed as in Lemma~\ref{lem:2.1},
\begin{equation}\label{Tdom2}
 \dom T= {\sT}_+ \dot{+} {\sT}_1,
\end{equation}
so that $\sD_+:=\overline{{\sT}_+}$ is a maximal uniformly positive
subspace of $\sD:=\cdom T$.
Next introduce the restriction of (the graph of) $T$ by setting
\[
 T_+:=T \cap \left( \sD_+ \times \sH_2\right).
\]
Then $T_+$ is closed and as a restriction of $T$ it is also an
isometric relation from $(\sH_1,J_{\sH_1})$ to $(\sH_2,J_{\sH_2})$.
Moreover, $\dom T_+={\sT}_+\subseteq  \sD_+$ is a uniformly positive
subspace. This implies that for all $\{f,f'\}\in T_+$ and some
$\delta>0$,
\[
 [f',f']_2=[f,f]_1\geq \delta \|f\|_1^2,
\]
which shows that $\ker T_+=\{0\}$, so that $(T_+)^{-1}$ is an
isometric operator, and, moreover,
\[
 \|(T_+)^{-1}f'\|_1^2=\|f\|_1^2\leq \delta^{-1} [f,f]_1 \leq \delta^{-1}
 \|f\|_1^2.
\]
Therefore, the closed isometric operator $(T_+)^{-1}$ is also
bounded. Consequently, $\ran T_+=\dom (T_+)^{-1}$ is a closed
subspace in $\sH_2$. On the other hand, since $\dom T$ admits the
decomposition \eqref{Tdom2}, where ${\sT}_1$ is finite dimensional and
$\mul T=\mul T_+$ (also finite dimensional), one concludes that
\[
 \ran T=T({\sT}_1) + \ran T_+
\]
as a finite dimensional extension of the closed subspace $\ran T_+$
is a closed subspace of $\sH_2$. This completes the proof.
\end{proof}

Lemma \ref{PontryIsom} can be seen as an extension of
\cite[Theorem~IX.3.1]{Bognar}. It is known e.g. from
\cite[Theorem~IX.3.2]{Bognar} and \cite[Theorems~6.2,~6.3]{IKL})
that if $T$ is an isometric operator in a Pontryagin space such that
$\cran T$ (resp. $\cdom T$) is a nondegenerate subspace, then $T$
(resp. $T^{-1}$) is continuous. The next lemma contains main
properties of isometric relations acting between two Pontryagin
spaces.

\begin{lemma}\label{PontryIsom2}
For an isometric relation $T$ from the Pontryagin space
$(\sH_1,J_{\sH_1})$ to the Pontryagin space $(\sH_2,J_{\sH_2})$ the
following statements hold:
\begin{enumerate}
\def\labelenumi{\rm (\roman{enumi})}
\item If $\cran T$ (resp. $\cdom
T$) is a nondegenerate subspace of $(\sH_2,J_{\sH_2})$, then $T$
(resp. $T^{-1}$) is a continuous operator.

\item If\; $T$ is densely defined then $\kappa_-(\sH_1)\leq\kappa_-(\sH_2)$ and if, in addition,
$\kappa_-(\sH_1)=\kappa_-(\sH_2)$, then $\ran \overline{T}$ is a
closed nondegenerate subspace of $\sH_2$ and, moreover, $T$ and
$T^{-1}$ are continuous operators.

\item If $\ran T$ is dense in $\sH_2$ then $\kappa_-(\sH_1)\geq\kappa_-(\sH_2)$ and if, in addition,
$\kappa_-(\sH_1)=\kappa_-(\sH_2)$, then $\dom \overline{T}$ is a
closed nondegenerate subspace of $\sH_1$ and, moreover, $T$ and
$T^{-1}$ are continuous operators.

\item If the relation $T$ is unitary and $\ker T=\mul T=\{0\}$, then
$\kappa_-(\sH_1)=\kappa_-(\sH_2)$ and $T$ is a standard unitary
operator.

\item If the relation $T$ is unitary and $\kappa_-(\sH_1)=\kappa_-(\sH_2)$, then
\begin{equation}\label{kermul}
 \mul T=\{0\}\quad\Longleftrightarrow\quad \ker T=\{0\}.
\end{equation}
In particular, if $\kappa_-(\sH_1)=\kappa_-(\sH_2)$ then a unitary
relation $T$ is an operator if and only if it is a standard unitary
operator.
\end{enumerate}

\end{lemma}
\begin{proof}
(i) By assumption $T\subseteq  T^{-[*]}$ and $\overline{T}\subseteq
T^{-[*]}$, i.e., $\overline{T}$ is also isometric. Hence, if $\cran T$ is nondegenerate,
\eqref{isome} implies
\[
 \mul \overline{T}\subset \cran T\cap (\cran T)^{[\perp]}=\{0\}.
\]
Thus $T$ is a closable operator, which by Lemma~\ref{PontryIsom} and
the closed graph theorem implies that $\overline{T}$ and, therefore,
also $T$ is continuous. Similarly it is seen that
\[
 \ker \overline{T}\subseteq  \cdom T\cap (\cdom T)^{[\perp]}=\{0\},
\]
if $\cdom T$ is nondegenerate, and then $\overline{T}^{-1}$ and
$T^{-1}$ are continuous.

(ii) Since $\cdom T=\sH_1$ is nondegenerate, $T^{-1}$ is a
continuous operator by item (i). On the other hand $\dom T$, as a
dense subspace of $\sH_1$, contains a negative subspace
$\cD_-\subset \dom T$ of dimension $\kappa_-(\sH_1)$; see
\cite[Theorem~IX.1.4]{Bognar}. Then also $\ran T$ contains a
negative subspace of the same dimension and hence
$\kappa_-(\sH_2)\geq \kappa_-(\sH_1)$.

Now assume that $\kappa_-(\sH_1)=\kappa_-(\sH_2)$. Then $\cran
T=\cran \overline{T}$ is necessarily a nondegenerate subspace of the
Pontryagin space $(\sH_2,J_{\sH_2})$, see e.g.
\cite[Lemma~II.10.5]{Bognar}, and (i) shows that $T$ is a continuous
operator.

(iii) This follows by applying (ii) to $T^{-1}$, which is also an
isometric relation.

(iv) If $T$ is unitary then the conditions $\ker T=\mul T=\{0\}$ are
equivalent to $\cdom T=\sH_1$ and $\cran T=\sH_2$; see \eqref{eee}
in Proposition~\ref{UNIT}. Now the assertions follow from (ii) and
(iii).

(v) Let $\kappa_-(\sH_1)=\kappa_-(\sH_2)$ and let $T$ be unitary. By
symmetry it suffices to prove one implication in \eqref{kermul},
say, ``$\Rightarrow$''. The condition $\mul T=\{0\}$ is equivalent
to $\cran T=\sH_2$; see \eqref{eee}. Now item (iii) shows that $T$ and
$T^{-1}$ are continuous operators. Thus, in particular, $\ker
T=\{0\}$ and \eqref{eee} together with Lemma~\ref{PontryIsom} shows
that $\dom T=\sH_1$ and $\ran T=\sH_2$.
\end{proof}

\subsection{Operator colligations}
Let $\sH$ be a Pontryagin space, $\sL_1$ and $\sL_2$ be Hilbert spaces.
The set of bounded everywhere defined operators from  $\sL_1$ to $\sL_2$  is denoted by $\mathbf{B}(\sL_1,\sL_2)$, $\mathbf{B}(\sL_1):=\mathbf{B}(\sL_1,\sL_1)$.
Let $U$ be a bounded operator
from ${\sH}\oplus{\sL}_1 $ to ${\sH}\oplus{\sL}_2$  represented in the block form
\[
U=\left(%
\begin{array}{cc}
  T & F \\
  G & H \\
\end{array}%
\right):\begin{pmatrix} {\sH} \\{\sL}_1
\end{pmatrix}\rightarrow\begin{pmatrix} {\sH} \\{\sL}_2
\end{pmatrix}
\]
The quadruple $({\sH},{\sL}_1,{\sL}_2, U)$ is called a \textit{colligation}, $\sH$ is the state space, ${\sL}_1$ and ${\sL}_2$ are the incoming and the outgoing spaces, $T$ is the main operator and $U$ is called the connecting operator of the colligation. The colligation $({\sH},{\sL}_1,{\sL}_2, U)$  is called \textit{unitary}, if $U$ is a unitary operator from ${\sH}\oplus{\sL}_1 $ to ${\sH}\oplus{\sL}_2$. The colligation $({\sH},{\sL}_2,{\sL}_1, U^{[*]})$  is called adjoint to the colligation $({\sH},{\sL}_1,{\sL}_2, U)$; cf. \cite{Br,ADRS97}.

Components of a \textit{unitary colligation} satisfy the following identities
\begin{equation}\label{E:komp}
\begin{split}
&T^\zx T+G^\zx G=I_{{\sH}},\ F^\zx F+H^*H=I_{\sL_1},\ T^\zx F+G^\zx
H=0,\\& TT^\zx+FF^\zx=I_{{\sH}},\ GG^\zx+HH^*=I_{\sL_2},\
TG^\zx+FH^*=0,
\end{split}
\end{equation}
which are equivalent to the identities
\[
U^{[*]}U=I_{{\sH}\oplus{\sL}_1},\quad UU^{[*]}=I_{{\sH}\oplus{\sL}_2}.
\]
A colligation $({\sH},{\sL}_1,{\sL}_2,U)$ is said to be
\textit{closely connected}, if
\begin{equation}\label{eq:Closely_C}
    \sH=\overline{\textup{span}}\,\{\ran (T^mF),\ran ((T^{[*]})^nG^{[*]}):\,m,n\ge 0\}.
\end{equation}
A unitary colligation is closely connected if and only if the operator $U$ has no nontrivial reducing subspaces.
The operator valued function
$$
\Theta_\Delta(\l):=H+\l G(I-\l T)^{-1}F,\quad 1/\l\in\rho(T),
$$
is called the characteristic function of the unitary colligation $({\sH},{\sL}_1,{\sL}_2,U)$.

Recall, see e.g. \cite{ADRS97}, that a $\mathbf{B}(\sL_1,\sL_2)$-valued  function $\Theta(\lambda)$ is said to belong
to the \textit{generalized Schur class} $\cS_{\kappa}^0(\sL_1,\sL_2)$ if it is holomorphic in a neighborhood $\Omega$ of $0$ and the
kernel
\begin{equation}\label{kerLambda}
{\mathsf K}_\omega^\Theta(\lambda)=
\frac{I-\Theta(\lambda)\Theta(\omega)^*}{1-\lambda\overline{\omega}}
\end{equation}
has $\kappa$ negative squares in $\Omega\times\Omega$, i.e. for any finite set of points $\omega_1,\dots,\omega_n$ in $\Omega$ and vectors $f_1,\dots,f_n$ in $\sL_2$, the
Hermitian matrix
\begin{equation}\label{kerLambda1}
\left(({\mathsf K}_{\omega_i}^\Theta(\omega_j)f_j,f_i)_{\sL_2}\right)_{i,j=1}^n
\end{equation}
has at most $\kappa$ negative eigenvalues, and for some choice of  $\omega_1,\dots,\omega_n$ in $\Omega$ and $f_1,\dots,f_n$ in $\sL_2$ the matrix~\eqref{kerLambda1} has exactly $\kappa$
negative eigenvalues.

As is known, see~\cite{ADRS97}, the characteristic function of a closely connected unitary colligation belongs
to the generalized Schur class $\cS_{\kappa}^0(\sL_1,\sL_2)$, where $\kappa=\kappa_-(\sH)$. Moreover, the converse is also true; see e.g.~\cite[Theorem~2.3.1]{ADRS97}.

\begin{theorem}\label{thm:Realiz}
Let $\sL_1$ and $\sL_2$ be Hilbert spaces and let $S(z)$ belong to the generalized Schur class $\cS_{\kappa}^0(\sL_1,\sL_2)$. Then there exists a closely connected unitary colligation $({\sH},{\sL}_1,{\sL}_2,U)$, such that  the corresponding characteristic function $\Theta_\Delta(z)$ coincides with $S(z)$ in a neighborhood $\Omega$ of $0$.
\end{theorem}

In what follows a $\mathbf{B}(\sL_1,\sL_2)$-valued
function $\Theta(\cdot)$ holomorphic in some open subset $\Omega\subset\dD$ is said to belong to the generalized Schur class $\cS_{\kappa}(\sL_1,\sL_2)$,
if the kernel \eqref{kerLambda} has $\kappa$ negative squares in $\Omega\times\Omega$. In particular, we do not require that $0\in\Omega$, which implies
that characteristic functions of unitary colligations used in Theorem \ref{thm:Realiz} are not sufficient to give a realization for all functions $\Theta(\cdot)$ from the class $\cS_{\kappa}(\sL_1,\sL_2)$ .
In the next section we introduce the notions of a unitary boundary pair for an isometric operator and an associated Weyl function as a replacement for unitary colligations and their characteristic functions. These new notions allow to realize an arbitrary operator function from the class $\cS_{\kappa}(\sL_1,\sL_2)$ as the Weyl function of a Pontryagin space isometric operator, corresponding to some unitary boundary pair.

\section{Unitary boundary pairs for isometric operators}
\label{sec3}
\subsection{Unitary boundary pairs and the main transform}
\label{subsec3.1}

Let ${\sH}$ be a Pontryagin space with the negative index $\k$ and the fundamental symmetry $J_\sH$
and let ${\sL}_1$ and ${\sL}_2$ be Hilbert spaces. In this section we introduce the notion of
\textit{a unitary boundary pair for an isometric operator} $V$:
here, and in what follows, $V$ is assumed to be closed.

For this purpose equip the Hilbert spaces
${\sH}^2$ and ${\sL}={\sL}_1\times{\sL}_2$ with the indefinite inner
products by the formulas
$$
[\wh{f},\wh{g}]_{{\sH}^2}=( J_{{\sH^2}}\wh{f},\wh{g})_{\sH^2},\quad
[\wh{u},\wh{v}]_{{\sL}}=\left( J_{{\sL}}\wh{u},\wh{v}\right)_{\sL},
$$
where
$$
J_{{\sH}^2}=\left(%
\begin{array}{cc}
  J_{{\sH}} & 0 \\
  0 & -J_{{\sH}} \\
\end{array}%
\right),\qquad
J_{{\sL}}=\left(%
\begin{array}{cc}
  I_{{\sL}_1} & 0 \\
  0 & -I_{{\sL}_2} \\
\end{array}%
\right),
$$
and $\wh{f}=\{f,f'\}, \wh{g}=\{g,g'\}\in \sH^2$, $\wh{u}=\{u_1,u_2\}, \wh{v}=\{v_1,v_2\}\in\sL$.
Then $({\sH}^2, J_{{\sH}^2})$ and
$({\sL}, J_{{\sL}})$ are Kre\u{\i}n spaces.
In particular, for $\sV\subset {\sH}^2$ the linear set $\sV^{[\perp]}$
in the Kre\u{\i}n space $({\sH}^2, J_{{\sH}^2})$
can be characterized as follows
\begin{equation}\label{e:invadj}
 \wh{g} \in \sV^{[\perp]} \quad\Longleftrightarrow\quad
 [\wh{f},\wh{g}]_{{\sH}^2}=0 \text{ for all } \wh{f}\in \sV \quad\Longleftrightarrow\quad
 \wh{g} \in \sV^{-[*]}.
\end{equation}

\begin{definition}\label{D:1}
Let ${\sL}_1$ and ${\sL}_2$ be Hilbert spaces, let $V$ be a closed
isometric operator (or isometric relation) in ${\sH}$ and let $\G$ be a linear relation
${\sH}^2\rightarrow {\sL}$, where ${\sL}:={\sL}_1\times{\sL}_2$.

The pair $({\sL},\G)$ will be called a unitary boundary pair for $V$ if:
\begin{enumerate}
\item
$V=\ker \G$ and for all $\{\wh f, \wh
u\},\{\wh g, \wh v\}\in \G$
 the following identity holds
 \begin{equation}\label{E:1.1}
[f,g]_{{\sH}}-[f',g']_{{\sH}}=(u_1,v_1)_{{\sL}_1}-(u_2,v_2)_{{\sL}_2};
\end{equation}

\item
$\G$ is maximal in the sense that if  $\{\wh g, \wh v\}\in{\sH}^2\times{\sL}$ satisfies~\eqref{E:1.1} for all  $\{\wh f,\wh u\}\in\G$, then
  $\{\wh g, \wh v\}\in \G$.
Here
$$\{\wh f, \wh u\}= \left\{ \begin{pmatrix} f \\ f' \end{pmatrix},
         \begin{pmatrix} u_1 \\ u_2 \end{pmatrix}
 \right\},\quad
 \{\wh g, \wh v\}=\left\{
\begin{pmatrix} g \\ g' \end{pmatrix},
\begin{pmatrix} v_1 \\ v_2 \end{pmatrix}
 \right\}\in{\sH}^2\times{\sL}.
 $$
\end{enumerate}
\end{definition}

Item (1) of Definition~\ref{D:1} means that $\G$ is an isometric
linear relation from the Kre\u{\i}n space $({\sH}^2,
J_{{\sH}^2})$ to the Kre\u{\i}n space $({\sL},J_{{\sL}})$, while  items (1) and (2) together mean that $\G$ is unitary.

Application of Proposition~\ref{UNIT} to a unitary boundary pair
leads to the following statement.

\begin{proposition}\label{cor:ker_G}
  Let $({\sL},\G)$ be  a unitary boundary pair for $V$. Then:
\begin{enumerate}
\def\labelenumi{\rm (\roman{enumi})}
  \item $V_*:=\dom \G$ is dense in $V^\sx$;
  \item $\ran\G$ is dense in ${\sL}$ if and only if $\mul\G=\{0\}$;
  \item $\ran\G={\sL}$ if and only if $\dom\G=V^{-[*]}$ and $\mul\G=\{0\}$.
\end{enumerate}
\end{proposition}

Define the components $\G_1$ and $\G_2$ of $\G$ by
\begin{equation}\label{eq:G_1}
  \G_1:=\left\{ \{\widehat{f}, u_1 \}:  \left\{ \widehat{f},
         \begin{pmatrix} u_1 \\ u_2 \end{pmatrix}
 \right\} \in \G \quad\text{for some} \quad u_2\in {\sL}_2
  \right\};
\end{equation}
\begin{equation}\label{eq:G_2}
\G_2:=\left\{ \{\widehat{f},u_2 \}:  \left\{ \widehat{f},
         \begin{pmatrix} u_1 \\ u_2 \end{pmatrix}
 \right\} \in \G \quad\text{for some} \quad u_1 \in {\sL}_1
  \right\}.
\end{equation}
In the case that $\G_1$ and $\G_2$ are single-valued, i.e. $\mul \G_1=\mul
\G_2=0$ and $\dom \G=V^{-\zx}$ and $\ran\G={\sL}_1\times{\sL}_2$ the collection
$\{{\sL}_1\times{\sL}_2,\G_1,\G_2\}$ is called an \textit{ordinary boundary triple} for the isometric operator $V$.
For a Hilbert space isometric operator the corresponding notion was introduced and studied in~\cite{MM03,MM04}
as a boundary triple for the dual pair $(V,V^{-1})$.

For ordinary boundary triples an application of the closed graph theorem
shows that the component mappings $\G_1$ and $\G_2$ are bounded.
However, for a general unitary boundary pair $({\sL},\G)$ the mappings $\G_1$ and $\G_2$
need not be bounded or single-valued. With $\G_1$ and $\G_2$ one associates the extensions $V_1$ and $V_2$ of $V$
by the equalities
\begin{equation}\label{eq:A12}
  V_1:=\ker\G_1,\quad V_2:=\ker\G_2.
\end{equation}
It follows from the identity~\eqref{E:1.1} that $V_1$ is an \textit{expanding linear relation}
and $V_2$ is a \textit{contractive linear relation} in the Pontryagin space $\sH$.
Moreover, it is clear from \eqref{E:1.1} that the sets $\ker V_1\setminus \ker V$ and $\mul V_2\,(=\mul
V_2\setminus \mul V)$ consist of negative vectors in $\sH$ while the sets
$\ker V_2\setminus \ker V$ and $\mul V_1\,(=\mul V_1\setminus \mul V)$ consist
of positive vectors of $\sH$.

Let $\sH_j:={\sH}\times{\sL}_j$ $(j=1,2)$ be a Pontryagin space with the inner product
\[
 \left[\begin{pmatrix} f \\ u \end{pmatrix},\begin{pmatrix} f \\ u \end{pmatrix}\right]_{\sH_j}=[f,f]_\sH+\|u\|_{\sL_j}^2, \quad f\in\sH,\,u\in\sL_j,\,j=1,2.
\]
In establishing some properties of unitary boundary pairs it is useful to connect the unitary relation $\G$ which acts
between two Kre\u{\i}n spaces to another unitary relation that acts between two Pontryagin spaces, since
unitary relations between Pontryagin spaces have simpler structure. For this purpose
we introduce the following transform from the Kre\u{\i}n space $\left({\sH}^2\oplus{\sL},J_{{\sH}^2}\oplus(-J_\sL)\right)$ to
the Kre\u{\i}n space $\left({\sH_2}\oplus{\sH_1},(J_{{\sH_2}})\oplus(-J_{{\sH_1}})\right)$ by
\[
 \cJ:\left\{ \begin{pmatrix} f \\ f' \end{pmatrix},
         \begin{pmatrix} u_1 \\ u_2 \end{pmatrix}
     \right\}
 \mapsto
     \left\{ \begin{pmatrix} f \\ u_2 \end{pmatrix},
         \begin{pmatrix} f' \\ u_1 \end{pmatrix}
     \right\},
\quad f,f'\in{\sH} \text{ and }  u_1\in{\sL}_1,\,  u_2\in{\sL}_2
\]
It establishes a one-to-one correspondence between
(closed) linear relations $\G$ from $({\sH}^2, J_{{\sH_2}})$ to $(\sL, J_\sL)$
and (closed) linear relations $\cU$ from ${\sH_2}$ to ${\sH_1}$ via
\begin{equation}
\label{awig}
 \Gamma\mapsto\cU :=\cJ(\Gamma)=
 \left\{\,
 \left\{ \begin{pmatrix} f \\ u_2 \end{pmatrix},
         \begin{pmatrix} f' \\ u_1 \end{pmatrix}
 \right\} :\,
 \left\{ \begin{pmatrix} f \\ f' \end{pmatrix},
         \begin{pmatrix} u_1 \\ u_2 \end{pmatrix}
 \right\}
 \in \Gamma
 \,\right\}.
\end{equation}

The linear relation $\cU$ will be called the \textit{main transform} of $\G$; cf. \cite{DHMS06} for the case of symmetric operators.
In the following lemma, which is an analog of \cite[Proposition~2.10]{DHMS06}, some basic properties of the transform $\cJ$ are given.

\begin{lemma}\label{lem:Col_BP}
Let the linear relation $\G$ from $({\sH}^2, J_{{\sH}^2})$  to $({\sL}, J_{{\sL}})$ and the linear relation $\cU$ from
${\sH}\times{\sL}_2$ to ${\sH}\times{\sL}_1$ be connected by $\cU=\cJ(\G)$. Then
\begin{equation}\label{Lemma_col_bp}
    \cU^{-[*]}=\cJ(\G^{-[*]}).
\end{equation}
Moreover, the transform $\cJ$ establishes a one-to-one correspondence between isometric (unitary, contractive, expanding) relations $\G$ from $({\sH}^2, J_{{\sH}^2})$  to $({\sL}, J_{{\sL}})$ and isometric (unitary, contractive, expanding) relations $\cU$ from ${\sH}\times{\sL}_2$ to ${\sH}\times{\sL}_1$.
\end{lemma}

\begin{proof}
It is straightforward to check that for all elements of the form
$$
  \sk\{{\begin{pmatrix}f \\ f' \end{pmatrix},\begin{pmatrix}u_1 \\ u_2 \end{pmatrix}}\},
  \sk\{{\begin{pmatrix}g \\ g' \end{pmatrix},\begin{pmatrix}v_1 \\ v_2 \end{pmatrix}}\}\in
  \begin{pmatrix}{\sH} \\ {\sH} \end{pmatrix}\times\begin{pmatrix}{\sL}_1 \\ {\sL}_2 \end{pmatrix},
$$
the following identity is satisfied:
\begin{equation}\label{connectGreenEq}
\left[\begin{pmatrix}f \\ u_2 \end{pmatrix},\begin{pmatrix}g \\ v_2 \end{pmatrix}\right]_{\sH_2}
  -\left[\begin{pmatrix}f' \\ u_1 \end{pmatrix},\begin{pmatrix}g' \\ v_1 \end{pmatrix}\right]_{\sH_1}
  =
  \left[\begin{pmatrix}f \\ f' \end{pmatrix},\begin{pmatrix}g \\ g' \end{pmatrix}\right]_{\sH^2}
  -\left[\begin{pmatrix}u_1 \\ u_2 \end{pmatrix},\begin{pmatrix}v_1 \\ v_2 \end{pmatrix}\right]_\sL.
\end{equation}
In view of \eqref{e:invadj} this identity implies the equivalence
$$
\sk\{{\begin{pmatrix}f \\ u_2 \end{pmatrix},\begin{pmatrix}f' \\ u_1 \end{pmatrix}}\}\in\cU^{-[*]}\Longleftrightarrow
\sk\{{\begin{pmatrix}f \\ f' \end{pmatrix},\begin{pmatrix}u_1 \\ u_2 \end{pmatrix}}\}\in\G^{-[*]}
$$
which leads to identity \eqref{Lemma_col_bp}. It follows from \eqref{Lemma_col_bp} that
\[
\cU^{-1}\subseteq \cU^{[*]}\Longleftrightarrow\G^{-1}\subseteq \G^{[*]}, \quad
\cU^{-1}=\cU^{[*]}\Longleftrightarrow\G^{-1}=\G^{[*]},
 \]
i.e., $\cU$ is isometric (unitary) precisely when $\G$ is isometric (resp. unitary).
The con\-nection between contractive (expanding) relations $\G$ and $\cU$ is clear from \eqref{connectGreenEq}.
\end{proof}

The next proposition contains the basic properties of $\G_1$, $\G_2$ and $V_1$, $V_2$
for a unitary boundary pair $({\sL},\G)$.

\begin{proposition}\label{P:G1G2}
Let $({\sL},\G)$ be a unitary boundary pair for $V$, %let $\G_1$ and $\G_2$ be the components of $\G$ as above,
let  $\G_1$, $\G_2$ and $V_1$, $V_2$ be  defined by \eqref{eq:G_1}, \eqref{eq:G_1} and \eqref{eq:A12}, and let $\cU$ be the main transform of $\G$.
Then:
\begin{enumerate}
\def\labelenumi{\rm (\roman{enumi})}
\item $\G_1$ and $\G_2$ are closed;
\item $\ran\G_1=\sL_1$ and $\ran\G_2=\sL_2$;
\item $\mul\G_j=P_{\sL_j}(\mul\G)$, $j=1,2$, and the following equivalences hold:
\[
  \mul\G_1=\{0\}\quad \Longleftrightarrow\quad  \mul\G_2=\{0\} \quad \Longleftrightarrow\quad  \mul\G=\{0\};
\]
\item the extensions $V_1$ and $V_2$ of $V$ are closed, $V_1\subseteq  V_2^{-[*]}$ and $V_2\subseteq  V_1^{-[*]}$;

\item the following equivalences hold:
\[
 \mul V_2=\{0\}\Longleftrightarrow \mul \cU=\{0\};
\]
\[
 \ker V_1=\{0\}\Longleftrightarrow \ker \cU=\{0\};
\]
\[
 \mul V_2=\{0\} \Longleftrightarrow \ker V_1=\{0\}.
\]
If one of the sets appearing in the above equivalences is trivial, then the main transform $\cU$ from $\sH\times\sL_2$ to $\sH\times\sL_1$
of $\Gamma$ is a standard unitary operator.
\end{enumerate}
\end{proposition}
\begin{proof}
(i) To prove that $\G_1$ is closed assume that $\{\wh f_n,\wh u_n\}\in \Gamma$ such that
$\wh f_n\to \wh f\in\sH^2$ and $u_{1,n}\to u_1\in\sL_1$. Then \eqref{E:1.1} gives
\[
 [f_n-f_m,f_n-f_m]_{{\sH}}-[f_n'-f_m',f_n'-f_m']_{{\sH}}=\|u_{1,n}-u_{1,m}\|^2_{{\sL}_1}-\|u_{2,n}-u_{2,m}\|^2_{{\sL}_2}
\]
and letting $n,m\to\infty$ one concludes that $\|u_{2,n}-u_{2,m}\|_{{\sL}_2}\to 0$. As a Cauchy sequence $(u_{2,n})$
converges to some element $u_2$ in $\sL_2$. This means that $\{\wh f_n,\wh u_n\}\to \{\wh f,\wh u\}$ and, since $\G$ is closed
as a unitary relation, $\{\wh f,\wh u\}\in \G$ and thus $\{\wh f,u_1\}\in\G_1$.
This proves that $\G_1$ is closed. Similarly one proves that $\G_2$ is closed.

(ii) First it is shown that $\ran \G_1$ is a closed subspace of $\sL_1$. For this consider
the main transform $\cU=\cJ(\G)$. By Lemma \ref{lem:Col_BP} $\cU$ is a unitary relation between the Pontryagin spaces
$\sH\times\sL_2$ and $\sH\times\sL_1$. Moreover, by Lemma \ref{PontryIsom} $\ran \cU$ is closed and Proposition \ref{UNIT}
shows that $\mul \cU$ is the isotropic part of $\ran \cU$. Therefore $\mul \cU$ is a closed finite dimensional
subspace of $\sH\times\sL_1$ and thus also the co-dimension $k$ of $\ran \cU$ is finite ($k\leq \kappa_-(\sH))$.
Let $\sM$ be any $k$-dimensional subspace such that $\ran\cU\dot{+}\sM=\sH\times\sL_1$ and let $P_1$ be the orthogonal projection
from $\sH\times\sL_1$ onto $\sL_1$. Then
\[
 \sL_1=P_1(\ran\cU\dot{+}\sM)=P_1\ran\cU + P_1\sM,
\]
and here $\dim P_1\sM\leq k$, which implies that $P_1\ran\cU=\ran\G_1$ is closed.

To see that $\ran \G_1=\sL_1$ it suffices to prove that $\ran \G_1$ is dense in $\sL_1$. For this assume that $v_1\perp \ran\G_1$.
Let $\wh g=\{0,0\}\in\sH^2$ and $\wh v=\{v_1,0\}\in\sL$. Then $\{\wh g,\wh v\}$ satisfies the identity \eqref{E:1.1}
for all $\{\wh f,\wh u\}\in\Gamma$, and hence assumption (2) in Definition \ref{D:1} implies that $\{\wh g,\wh v\}\in\G$.
This means that $\wh v\in\mul\G$ and then, in particular, $v_1\in\mul \G_1\subseteq  \ran\G_1$. Thus, $v_1=0$ and this proves
that $\ran \G_1=\sL_1$. The equality $\ran\G_2=\sL_2$ is then clear by symmetry.

(iii) The identities $\mul\G_j=P_{\sL_j}(\mul\G)$, $j=1,2$, are clear from the definition of $\G$.
Hence, $\mul\G=\{0\}$ implies that $\mul\G_1=\mul\G_2=\{0\}$.
Conversely, assume that e.g. $\mul\G_1=\{0\}$ and that $\wh v\in\mul\G$. Then $\wh v=\{0,v_2\}$
and hence for all $\{\wh f,\wh u\}\in\Gamma$ the identity \eqref{E:1.1} implies that
\[
 0=(u_1,0)_{\sL_1}-(u_2,v_2)_{\sL_2}=-(u_2,v_2)_{\sL_2}.
\]
In item (ii) it was shows that $\ran\G_2=\sL_2$ and, thus, one concludes that $v_2=0$.
Hence, $\mul\G=\{0\}$. Similarly, $\mul \G_2=\{0\}$ implies $\mul\G=\{0\}$.

(iv) Since $V_j=\ker \G_j$ and $\G_j$ is closed by item (i) also $V_j$ is closed $j=1,2$.
If $\wh f\in V_1$ and $\wh g\in V_2$ then it follows from \eqref{E:1.1} that $[\wh f,\wh g ]_{\sH^2}=0$
and in view of \eqref{e:invadj} this means that the inclusions $V_1\subseteq  V_2^{-[*]}$ and
$V_2\subseteq  V_1^{-[*]}$ hold; these inclusions are clearly equivalent to each other.

(v) The definition in \eqref{awig} shows that $$\mul V_2=P_\sH (\mul \cU).$$
Hence, $\mul \cU=\{0\}$ implies $\mul V_2=\{0\}$. Conversely, if $\{f',u_1\}\in\mul \cU$ then
$f'\in\mul V_2$ and if $\mul V_2=\{0\}$ then $f'=0$. Now \eqref{E:1.1} implies that $(u_1,u_1)_{\sL_1}=0$
and thus also $u_1=0$, i.e., $\mul \cU=\{0\}$.
The equivalence of $\ker \cU=\{0\}$ and $\ker V_1=\{0\}$ can be seen in the same way.
As to the last equivalence notice that $\kappa_-(\sH\times\sL_1)=\kappa_-(\sH)=\kappa_-(\sH\times\sL_2)$.
Now, according to item (v) in Lemma~\ref{PontryIsom2} $\ker \cU=\{0\}$ is equivalent to $\mul \cU=\{0\}$
and in this case $\cU$ is a standard unitary operator.
\end{proof}

Later it is shown that the inclusions in (iv) of Proposition \ref{P:G1G2} actually hold as equalities; see Theorem \ref{thm:General}.
In the special case that $\sH$ is a Hilbert space and $V$ is an isometry in $\sH$
Proposition~\ref{P:G1G2} can be specialized as follow.

\begin{corollary}\label{Cor:G1G2}
Let $({\sL},\G)$ be a unitary boundary pair for an isometric operator $V$ in the Hilbert space $\sH$.
Then the properties (i)--(iv) in Proposition~\ref{P:G1G2} hold and, moreover, $V_2$ and $V_1^{-1}$ are contractive operators.
Furthermore, the main transform $\cU$ from $\sH\times\sL_2$ to $\sH\times\sL_1$ of $\Gamma$ is a standard Hilbert space
unitary operator.
\end{corollary}
\begin{proof}
The fact that in the Hilbert space case $V_2$ and $V_1^{-1}$ are contractive operators follows from \eqref{E:1.1}
which with the choice $\wh f=\wh g$ and $\wh u=\wh v$ can be rewritten as
\[
 \|f\|_\sH^2-\|f'\|_\sH^2=\|u_1\|_{\sL_1}^2-\|u_2\|_{\sL_2}^2.
\]
In particular, with $u_1=0$ the condition $f'=0$ implies $f=0$ and hence $\ker V_1=\{0\}$. Similarly
with $u_2=0$ the condition $f=0$ implies $f'=0$ and thus $\mul V_2=\{0\}$.
According to item (v) in Proposition~\ref{P:G1G2} one has $\mul \cU=\{0\}$ and thus $\cU\in \mathbf{B}(\sH)$ is a standard
unitary operator.
\end{proof}

\subsection{The Weyl function and the $\gamma$-fields of a unitary boundary pair}
\label{subsec3.2}
Define the set $\sN_\l(V_*)$ as the intersection of $\sN_\l$ and $V_*$:
\[
\sN_\l(V_*):=\left\{f_\l:\begin{pmatrix}f_\l\\
\l f_\l\end{pmatrix}\in V_*\right\}\,\left(\subseteq  \sN_\l
=\left\{f_\l:\begin{pmatrix}f_\l\\
\l f_\l\end{pmatrix}\in V^{-\zx}\right\}\right)
\]
and the corresponding subset of $V_*$:
\[
\wh{\sN}_\l(V_*)=\left\{\wh{f_\l}=\begin{pmatrix}f_\l \\
\l f_\l\end{pmatrix}:\wh{f_\l}\in V_*\right\}.
\]
If $\widehat{f}_{\l}\in \widehat{\sN}_{\l}(V_*)$ and $\widehat{f}_{\mu}\in \widehat{\sN}_{\mu}(V_*)$ then for some $\widehat u,\widehat v\in {\sL}$ one has
\[
\left\{\widehat{f}_\l,
 \begin{pmatrix}
  u_1 \\ u_2
 \end{pmatrix}\right\}\in\G, \quad
 \left\{\widehat{f}_\mu,
 \begin{pmatrix}
  v_1 \\ v_2
 \end{pmatrix}\right\}\in\G
\]
and an application of \eqref{E:1.1} shows that
\begin{equation}\label{Eqgamma1}
 (1-\l\overline{\mu} )[f_{\l},f_{\mu}]_{\sH}=(u_1,v_1)_{\sL_1}-(u_2,v_2)_{\sL_2}.
\end{equation}

\begin{lemma}\label{lem:matrix_G}{\rm (cf.~\cite[Lemma~3.2]{DD09})}
Let $u_j$ belong to a Hilbert space ${\cH}$, let $\l_j\in\mathbb{C}$, and define the vector valued function $f_j$ by
\[
 f_j(\l):=\frac{u_j}{1-\l_j\l}, \quad j=1,\dots,n.
\]
Then the Gram matrix of the functions $f_j(\l)$ $(j=1,\dots,n)$
in the space $L_2({\cH})$ over the unit circle
\begin{enumerate}
\def\labelenumi{\rm (\roman{enumi})}
  \item is equal to
$G=\left(\frac{(u_j,u_k)_{{\cH}}}{1-\l_j\overline{\l}_k}\right)_{j,k=1}^n$, if $\l_j\in \dD$ $(j,k=1,\dots,n)$;
  \item is equal to
$-G=\left(-\frac{(u_j,u_k)_{{\cH}}}{1-\l_j\overline{\l}_k}\right)_{j,k=1}^n$, if $\l_j\in \dD_e$ $(j,k=1,\dots,n)$.
\end{enumerate}
\end{lemma}
\begin{proof}
To determine the Gram matrix consider the inner product of the $\cH$-valued functions
$f_j(\lambda)$ and $f_k(\lambda)$ with $\l=e^{it}$, $t\in[0,2\pi]$, for $j,k=1,\dots,n$.

(i) If $\l_j\in \dD$ then in view of the equalities
\[
\begin{split}
\left(f_j,f_k\right)_{L_2({\cH})}
%\left(\frac{u_j}{1-\l_j\l},\frac{u_k}{1-\l_k\l}\right)_{L_2}
&=\frac{1}{2\pi}\int\limits_0^{2\pi}
\frac{dt}{(1-\l_je^{it})(1-\overline{\l}_k e^{-it})}(u_j,u_k)_{{\cH}}\\
&=\frac{1}{2\pi i}\oint
\frac{d\l}{(1-\l_j\l)(\l-\overline{\l}_k)}(u_j,u_k)_{{\cH}}=\frac{(u_j,u_k)_{{\cH}}}{{1-\l_j\overline{\l}_k}}
\end{split}
\]
the matrix $G$ coincides with the Gram matrix of the functions $f_j(\l)$ $(j=1,\dots,n)$
in the space $L_2({\cH})$ on the unit circle. %Hence, if $\l_j\in \dD$ $(j=1,\dots,n)$ then

(ii) Analogously, if $\l_j\in \dD_e$ $(j=1,\dots,n)$ then it follows from
\[
 \frac{1}{2\pi i}\oint
 \frac{d\l}{(1-\l_j\l)(\l-\overline{\l}_k)}=\frac{1}{2\pi i}\oint
 \frac{d\l}{-\l_j(\l-1/\l_j)(\l-\overline{\l}_k)}=-\frac{1}{{1-\l_j\overline{\l}_k}}
\]
that the matrix
G=$\left(\frac{(u_j,u_k)_{{\cH}}}{1-\l_j\overline{\l}_k}\right)_{j,k=1}^n$
differs in sign from the Gram matrix for functions $f_j(\l)$, $j=1,\dots,n$.
\end{proof}

\begin{proposition}\label{P:kappagamma}
Let $V$ be an isometric operator in a Pontryagin space ${\sH}$ and let $({\sL},\G)$ be a unitary boundary pair for $V$.
Then:
\begin{enumerate}
\def\labelenumi{\rm (\roman{enumi})}

\item
the set $(\sigma_p(V_2)\setminus \sigma_p(V))\cap\dD_e$ consists of at most $\k$ points and
the corresponding eigenspaces are negative;

\item
the set $(\sigma_p(V_1)\setminus \sigma_p(V))\cap\dD$ consists of at most $\k$ points and the corresponding eigenspaces are negative.
\end{enumerate}
\end{proposition}
\begin{proof}
(i) A point $\l\in\dD_e$ belongs to the set $\sigma_p(V_2)$ if and only if
there exists $\widehat{f}_{\l}\in \widehat{\sN}_{\l}(V_*)$, ${f}_{\l}\neq 0$, such that
$\left\{\widehat{f}_\l,
\begin{pmatrix}
  u_1 \\ 0
\end{pmatrix}\right\}\in\G$ for some $u_1\in {\sL}_1$.
The assumption $\l\not\in\sigma_p(V)$ means that $u_1\ne 0$.

Now assume that $\l_1,\dots,\l_{\k+1}\in \mathbb{D}_e$ and for some linearly independent vectors $\widehat{f}_{\l_j}\in \widehat{\sN}_{\l_j}(V_*)\setminus\{0\}$ and some $u_{1,j}\in{\sL}_1\setminus\{0\}$,
$$\left\{\widehat{f}_{\l_j},
\begin{pmatrix}
  u_{1,j} \\ 0
\end{pmatrix}
\right\}\in\G,\quad j=1,\dots,\k+1.$$
Then from \eqref{Eqgamma1} one gets
$$
[f_{\l_j},f_{\l_k}]_{\sH}=\frac{(u_{1,j},u_{1,k})_{\sL_1}}{{1-\l_j\overline{\l}_k}},
\quad j,k=1,\dots,\k+1.
$$
If $\lambda_j=\lambda_k$ for some $j\neq k$ then the vectors $u_j$ and $u_k$ are linearly independent
by the assumptions $\lambda_j\not\in\sigma_p(V)$. On the other hand, if $\lambda_j\neq \lambda_k$, then the vector functions
$f_j(\lambda)$ and $f_k(\lambda)$ defined in Lemma \ref{lem:matrix_G} are also linearly independent. Hence the matrix
$G$ in Lemma \ref{lem:matrix_G} is invertible. One concludes that the form
$$
\sum\limits_{j,k=1}^{\k+1}[f_{\l_j},f_{\l_k}]_{\sH}\xi_j\overline{\xi}_k=
\left[\sum\limits_{j=1}^{\k+1}\xi_jf_{\l_j},\sum\limits_{k=1}^{\k+1}\xi_kf_{\l_k}\right]_{\sH}
$$
is negative for linearly independent vectors $\widehat{f}_{\l_j}$ when $\l_j\in \mathbb{D}_e$.
This contradicts the assumption that the Pontryagin space
${\sH}$ has negative index $\kappa$.

(ii) The second statement is proved analogously.
\end{proof}

In the sequel the following two subsets of $\dD$ and $\dD_e$ (cf. \eqref{DDe}) will often appear:
\begin{equation}\label{eq:cD}
  \cD:=\dD\backslash\s_p(V_1),\quad
  \cD_e:=\dD_e\backslash\s_p(V_2).
\end{equation}
It should be noted that for various realization results and for the study of proper extensions of the isometry $V$ it is typically sufficient
to assume that $\sigma_p(V)=\emptyset$; this is the case in particular when $V$ is a simple isometric operator in $\sH$.
In this case Proposition~\ref{P:kappagamma} shows that both of the sets $\sigma_p(V_2)\cap\dD_e$ and $\sigma_p(V_1)\cap\dD$
contain at most $\k$ points.

Now consider the restrictions of (the graphs of) $\G_1$ and $\G_2$ to $\wh{\sN}_\l(V_*)$,
\[
 \G_1\upharpoonright\wh{\sN}_\l(V_*):=\left\{\left\{\wh{f},\begin{pmatrix}
      u_1 \\
      0
    \end{pmatrix}\right\}\in\G_1:
                                   \wh{f}\in\wh{\sN}_\l(V_*),\,
                                   u_1\in\sL_1
                                 \right\}
\]
and
\[
 \G_2\upharpoonright\wh{\sN}_\l(V_*):=\left\{\left\{\wh{f},\begin{pmatrix}
      0 \\
      u_2
    \end{pmatrix}\right\}\in\G_2:
                                   \wh{f}\in\wh{\sN}_\l(V_*),\,
                                   u_2\in\sL_2
                                 \right\}.
\]
It follows from \eqref{eq:cD} that
\[
 \ker\left(\G_1\upharpoonright\wh{\sN}_\l(V_*)\right)=\{0\} \quad  (\l\in\cD);\quad
  \ker\left(\G_2\upharpoonright\wh{\sN}_\l(V_*)\right)=\{0\} \quad  (\l\in\cD_e)
\]
and the assumption $\sigma_p(V)=\emptyset$ guarantees in particular that
the inverses
\begin{equation}\label{E:g12}
\wh{\g}_1(\l)=\sk({\G_1\upharpoonright\wh{\sN}_\l(V_*)})^{-1}\; (\l\in\cD);\quad
\wh{\g}_2(\l)=\sk({\G_2\upharpoonright\wh{\sN}_\l(V_*)})^{-1}\; (\l\in\cD_e)
\end{equation}
determine single-valued operator functions, which will be denoted by the same symbols
\[
\wh{\g}_1(\l):{\sL}_1\rightarrow\wh{\sN}_\l(V_*) \quad (\l\in\cD);\qquad
 \wh{\g}_2(\l):{\sL}_2\rightarrow\wh{\sN}_\l(V_*)\quad
(\l\in\cD_e).
\]

Let $\pi_1$ and $\pi_2$ be projections onto $\sL_1$ and $\sL_2$ in $\sL$, respectively.

\begin{definition}\label{def:gamma}
The operator functions
\begin{equation}\label{E:g122}
\g_1(\l)=\pi_1\wh{\g}_1(\l)\quad (\l\in\cD)
\quad\textup{and}\quad
\g_2(\l)=\pi_1\wh{\g}_2(\l)\quad (\l\in\cD_e)
\end{equation}
will be called the $\gamma$-fields of the  unitary boundary pair $({\sL},\G)$.
\end{definition}

The definition of the $\gamma$-fields of the unitary boundary pair $({\sL},\G)$ yields
the following explicit formulas:
  \begin{equation}\label{D:G1}
 \left\{\begin{pmatrix}
      f_\l \\
      \l f_\l
    \end{pmatrix},\begin{pmatrix}
      u_1 \\
      u_2
    \end{pmatrix}\right\}\in \Gamma, \quad \wh f_\l \in\wh\sN_\l(V_*),\,\; \l\in\cD\quad\Longrightarrow\quad
 \g_1(\l)u_1=    f_\l
  \end{equation}
  \begin{equation}\label{D:G2}
 \left\{\begin{pmatrix}
    f_\l \\
     \l f_\l
    \end{pmatrix},\begin{pmatrix}
      u_1 \\
      u_2
    \end{pmatrix}\right\}\in \Gamma,  \quad \wh f_\l \in\wh\sN_\l(V_*),\,\; \l\in\cD_e\quad\Longrightarrow\quad
    \g_2(\l)u_2=   f_\l.
  \end{equation}

Later it is shown that $\g_1(\l)$ and $\g_2(\l)$ are bounded everywhere defined operators and holomorphic in $\l$;
see Theorem \ref{thm:General}.

\begin{definition}\label{char}
  The family of linear relations defined by
  \begin{equation}\label{D:M1}
    \Theta(\l)=%\left(\G\upharpoonright\wh\sN_\l(V_*)\right)^{-1}=
    \left\{{\begin{pmatrix}
      u_1 \\
      u_2
    \end{pmatrix}}:\, \left\{\wh f_\l,\begin{pmatrix}
      u_1 \\
      u_2
    \end{pmatrix}\right\}\in \Gamma,\,\wh f_\l \in\wh\sN_\l(V_*)\right\},\quad\l\in \cD,
  \end{equation}
will be called the Weyl family of $V$ corresponding to the unitary boundary pair $({\sL},\G)$, or, briefly,
the  Weyl family of the unitary boundary pair $({\sL},\G)$.
\end{definition}

In Theorem~\ref{thm:General} it will be shown that the formula~\eqref{D:M1} determines a single-valued operator function, which is called the Weyl function of $V$ corresponding to the boundary triple $(\sL,\Gamma)$. If the mapping $\Gamma$  is single-valued, then the Weyl function  $\Theta(\l)$ can be defined by using the $\g$-fields
\begin{equation}\label{MG2}
                \Theta(\l)=\G_2\wh{\g}_1(\l), \quad \l\in\cD.
\end{equation}

\subsection{Unitary boundary pairs and unitary colligations}
\label{subsec3.3}

In the present section we consider a unitary boundary pair whose main transform is a unitary colligation
\begin{equation}\label{eq:U}
\Delta=({\sH},{\sL}_1,{\sL}_2, U),\quad
U=\left(%
\begin{array}{cc}
  T & F \\
  G & H \\
\end{array}%
\right):\begin{pmatrix} {\sH} \\{\sL}_1
\end{pmatrix}\rightarrow\begin{pmatrix} {\sH} \\{\sL}_2
\end{pmatrix}
\end{equation}
and write explicit formulas for all the objects connected with this unitary boundary pair in terms of the blocks of $U$.

\begin{theorem}\label{T:UV}
Let  $\Delta=({\sH},{\sL}_1,{\sL}_2, U)$  be a unitary colligation of the form~\eqref{eq:U}, let $\cU=\text{gr }U^{[*]}$, let $\Gamma=\cJ^{-1}(\cU)$ be a unitary relation from $({\sH}^2, J_{{\sH}^2})$ to $({\sL}, J_{{\sL}})$, i.e. $\cU=\cJ(\Gamma)$ as in Lemma~\ref{lem:Col_BP}, and let
\[
V:=\ker \Gamma,\quad V_*:=\dom \Gamma,\quad V_1:=\ker \Gamma_1,\quad V_2:=\ker \Gamma_2.
\]
Then the following statements hold:
\begin{enumerate}
\def\labelenumi{\rm (\roman{enumi})}

\item The pair $({\sL},\Gamma)$ is a unitary boundary pair for $V$.

\item The unitary relation $\Gamma$ admits the representations
\begin{equation}\label{E:Gstar}
\begin{split}
 \G & =\left\{\left\{\begin{pmatrix} Th+Fu_1\\ h \end{pmatrix},
         \begin{pmatrix} u_1\\Gh+Hu_1\end{pmatrix}\right\}:
        \begin{array}{c} h\in\sH \\ u_1\in {{\sL}_1} \end{array}\right\} \\[3mm]
    & =\left\{\left\{\begin{pmatrix} g\\T^{[*]}g+G^{[*]}u_2 \end{pmatrix},
    \begin{pmatrix} F^{[*]}g+H^*u_2\\u_2\end{pmatrix}\right\}:\,
    \begin{array}{c} g\in\sH \\ u_2\in {{\sL}_2} \end{array}\right\}.
\end{split}
\end{equation}
\item The isometry $V$ in $({\sH},J_\sH)$ admits the representations
\begin{equation}\label{eq:VV}
    V^{-1}=T\upharpoonright \ker G,\quad V=T^{[*]}\upharpoonright \ker F^{[*]}
\end{equation}
and the linear relation $V_*$ takes the form
\begin{equation}\label{E:V_star}
V_*=\left\{\begin{pmatrix} Th+Fu_1\\h\end{pmatrix}:\,
               \begin{array}{c}
                 h\in\sH \\
                 u_1\in {{\sL}_1}
               \end{array}
     \right\}
   = \left\{\begin{pmatrix} g\\T^{[*]}g+G^{[*]}u_2 \end{pmatrix}:\,
               \begin{array}{c}
                 g\in\sH \\
                 u_2\in {{\sL}_2}
               \end{array} \right\}.
\end{equation}

\item The multivalued part of $\Gamma$ has the representations
\begin{equation}\label{eq:mulV*}
    \mul \Gamma=\left\{\begin{pmatrix} u_1\\Hu_1 \end{pmatrix}: u_1\in \ker F \right\}
  =\left\{\begin{pmatrix}  H^*u_2\\u_2 \end{pmatrix}: v\in \ker G^{[*]} \right\}.
\end{equation}

\item The linear relations $V_1$ and $V_2$ are given by
\begin{equation}\label{E:V_12}
V_1=\left\{\begin{pmatrix} Th \\h
\end{pmatrix}: h\in {\sH}\right\}, \quad
V_2=\left\{\begin{pmatrix} g \\T^{[*]}g
\end{pmatrix}: g\in {\sH}\right\},
\end{equation}
and hence the sets $\cD:=\rho(V_1)\cap \dD$ and  $\cD_e:=\rho(V_2)\cap \dD_e$
are nonempty, they
coincide with the sets in \eqref{eq:cD} and are connected by
\begin{equation}\label{eq:DD_e}
\cD_e=\cD^\circ:=\{\lambda\in\dD_e:{1}/{\overline{\lambda}}\in\cD\},
\end{equation}
moreover, $\card(\dD\setminus\cD)=\card(\dD_e\setminus\cD_e)\le\kappa$.
\end{enumerate}
\end{theorem}
\begin{proof}
(i) By Lemma~\ref{lem:Col_BP} $\Gamma=\cJ^{-1}\cU$ is a unitary relation from $({\sH}^2, J_{{\sH}^2})$ to $({\sL}, J_{{\sL}})$.
This means that the assumptions (1), (2) of Definition \ref{D:1} are satisfied and, therefore,
$({\sL},\Gamma)$ is a unitary boundary pair for the isometry $V=\ker\G$ (which is an operator in view of (3) below).

(ii) Since the operator $U:\begin{pmatrix}
  \sH\\
  \sL_1
\end{pmatrix}\to\begin{pmatrix}
  \sH\\
  \sL_2
\end{pmatrix}$ is unitary then $\mathcal{U}=\gr U^{[*]}$ is also the graph of the operator $U^{-1}$ in \eqref{eq:U}, and hence $\mathcal{U}$ has the following representations
$$
\mathcal{U}=\gr U^{-1}=\sk\{{\sk\{{
\begin{pmatrix}
  Th+Fu_1\\
  Gh+Hu_1
\end{pmatrix},
\begin{pmatrix}
  h\\
  u_1
\end{pmatrix}}\}:\,\begin{array}{c}
  h\in\sH\\
  u_1\in\sL_1
\end{array}}\}.
$$
In view of \eqref{awig} this yields the first formula in \eqref{E:Gstar}.
The equality
$$
\mathcal{U}=\text{gr }{U^{[*]}}=\sk\{{\left\{
\begin{pmatrix}
  g\\
  u_2
\end{pmatrix},
\begin{pmatrix}
  T^{[*]}g+G^{[*]}u_2\\
  F^{[*]}g+H^{[*]}u_2
\end{pmatrix}\right\}:\,\begin{array}{c}
  g\in\sH\\
  u_2\in\sL_2
\end{array}}\}.
$$
leads to the second representation of $\G$ in \eqref{E:Gstar}.

(iii)\&(iv) The formulas \eqref{eq:VV}, \eqref{E:V_star}, and \eqref{eq:mulV*} are all implied by \eqref{E:Gstar}.

(v) The formulas \eqref{E:V_12} for $V_1=\ker\G_1$ and $V_2=\ker\G_2$ are again obtained from \eqref{E:Gstar}.
In particular, $V_2$ is the graph of the bounded operator $T^{[*]}$. It is closed and $\cD_e=\r(V_2)=\r(T^{[*]})\neq \emptyset$.
Since $T^{[*]}$ is a contractive operator in the Pontryagin space $(\sH, J_\sH)$, its spectrum in $\mathbb{D}_e$ consists of at most $\k$ eigenvalues, so
\[
\card(\dD_e\backslash\cD_e)=\card(\dD_e\backslash\r(T^{[*]}))\le\k,
\]
see e.g. \cite[p.91 Lemma 11.8]{IKL}.
Similarly, $V_1^{-1}$ is the graph of the contractive operator $T$ in the Pontryagin space $(\sH, J_\sH)$. Therefore,
\[
\cD=\rho(V_1)\cap \dD=\rho(T^{-1})\cap \dD=(\rho(T^{[*]})\cap \dD_e)^\circ=\cD_e^\circ.
\]
This completes the proof.
\end{proof}

By Proposition~\ref{cor:ker_G} the closure of $V_*$ is $V^{-[*]}$.
Hence, by taking closures in \eqref{E:V_star} one arrives at the following representations for $V^{-[*]}$:
\begin{equation}\label{eq:Vminstar}
V^{-[*]}=\sk\{{\begin{pmatrix}
    Th+f\\ h
\end{pmatrix}:\,\begin{array}{c}
  h\in\sH\\
  f\in\sH[-]\ker F^{[*]}
\end{array}}\},
\end{equation}
and
\begin{equation}\label{eq:Vminstar2}
V^{-[*]}=\sk\{{\begin{pmatrix}
  g\\ T^{[*]}g+f
\end{pmatrix}:\,\begin{array}{c}
  g\in\sH\\
  f\in\sH[-]\ker G
\end{array}}\}.
\end{equation}

\begin{theorem}\label{T:UV2}
Let  $\Delta=({\sH},{\sL}_1,{\sL}_2, U)$  be a unitary colligation of the form~\eqref{eq:U}, let $\cU=\text{gr }U^{[*]}$, let $\Gamma=\cJ^{-1}U$ and let  the corresponding  Weyl function $\Theta(\lambda)$ of $V$ and the $\gamma$-fields  $\gamma_1(\lambda)$ and $\gamma_2(\lambda)$ be given by~\eqref{D:M1} and~\eqref{D:G1}, \eqref{D:G2}.
Then the following statements hold:
\begin{enumerate}
\def\labelenumi{\rm (\roman{enumi})}

\item The defect subspace $\sN_\l(V_*)$ takes the form
 \begin{equation}\label{eq:sNl}
\sN_\l(V_*)=\{(I-\l T)^{-1}Fu_1:\,u_1 \in {{\sL}_1}\},\quad \l\in\cD;
 \end{equation}
 \begin{equation}\label{eq:sNle}
\sN_\l(V_*)=\{(\l I- T^{[*]})^{-1}G^{[*]}u_2:\, u_2 \in {{\sL}_2}\}, \quad \lambda\in \cD_e.
 \end{equation}

\item
$\Gamma_1(\wh\sN_\lambda(V_*))={\sL}_1$, $\l\in\cD$, and $\Gamma_2(\wh\sN_\lambda(V_*))={\sL}_2$, $\l\in\cD_e$,
and the $\gamma$-fields $\gamma_1(\lambda)$ and $\gamma_2(\lambda)$ take the form
$$
\gamma_1(\lambda)=(I-\l T)^{-1}F,\quad \l\in\cD;
$$
$$
\gamma_2(\lambda)=(\l I-T^{[*]})^{-1}G^{[*]}, \quad \lambda\in \cD_e.
$$

\item The Weyl function $\Theta(\lambda)$ of $V$ corresponding to the unitary boundary pair $({\sL},\Gamma)$
coincides with the characteristic function of the colligation $\Delta$, i.e.,
$$
{\Theta}(\l)=\Theta_\Delta(\lambda)=H+\l G(I-\l T)^{-1}F,\quad \l\in\cD.
%%\qquad
%%{\Theta}(\l)=\Theta_\Delta^{\#}(\lambda),\quad \l\in\cD_e.
$$
\item The unitary colligation  $\Delta$ is closely connected if and only if the operator $V$ is simple.
\end{enumerate}
\end{theorem}

\begin{proof}
(i) Recall that $\wh\sN_\l(V_*)$ consists of vectors
$\begin{pmatrix}f_\l\\ \l f_\l\end{pmatrix}\in V_*$.
Therefore, the vector
 $\begin{pmatrix}
   Th_\l+Fu_1\\
      h_\l
 \end{pmatrix}$
$(h_\l\in\sH,\, u_1\in\sL_1)$ belongs to $\wh\sN_\l(V_*)$ precisely when
\[
 h_\l=\l(Th_\l+Fu_1).
\]
Hence, for $\l\in\cD=\r(T^{-1})\cap\dD$ one obtains
\[
 h_\l=\l(I-\l T)^{-1}Fu_1, \quad\l\in \cD.
\]
Similarly, the vector
 $\begin{pmatrix}
   g_\l\\
   T^{[*]}g_\l+G^{[*]}u_2
    \end{pmatrix}$
$g_\l\in\sH,\, u_2\in\sL_2$ belongs to $\wh\sN_\l(V_*)$ if and only if $$T^{[*]}g_\l+G^{[*]}u_2=\l g_\l.$$
Hence, for $\l\in\cD_e=\r(T^{[*]})\cap\dD_e$ one obtains
\[
 g_\l=(\l I- T^{[*]})^{-1}G^{[*]}u_2, \quad\l\in \cD_e.
\]

(ii) By the first formula in \eqref{E:Gstar} one  gets
\begin{equation}\label{eq:3.23A}
 \left\{\begin{pmatrix} f_\l \\ \l f_\l\end{pmatrix},
 \begin{pmatrix} u_1\\ Hu_1+\l G(I-\l T)^{-1}Fu_1
\end{pmatrix}\right\}\in \G, \quad\l\in \cD,
 \end{equation}
 and in view of \eqref{E:g12} and \eqref{E:g122} the $\g$-field $\g_1(\l)$ takes the form
 $$
 \g_1(\l) u_1= f_\l=(I-\l T)^{-1}Fu_1, \quad\l\in \cD.
 $$

Similarly, by the second formula in \eqref{E:Gstar}
\begin{equation}\label{eq:3.23B}
 \left\{\begin{pmatrix}  g_\l \\ \l g_\l \end{pmatrix},
 \begin{pmatrix} H^*u_2+F^{[*]}(\l I-T^{[*]})^{-1}G^{[*]}u_2 \\ u_2\end{pmatrix}\right\}\in \G, \quad\l\in \cD_e,
\end{equation}
and hence
$$
 \g_2(\l) u_2= g_\l=(\l I-T^{[*]})^{-1}G^{[*]}u_2, \quad\l\in \cD_e.
$$

(iii) It follows also from \eqref{eq:3.23A} that
$$
{\Theta}(\l)u_1=\sk({H+\l G(I-\l T)^{-1}F})u_1={\Theta}_\Delta(\l)u_1,\quad \l\in\cD,\, u_1\in\sL_1.
$$

(iv) Notice that in view of \eqref{eq:Vminstar} and \eqref{eq:Vminstar2} the defect subspaces $\sN_\lambda$ take the form
\begin{equation}\label{eq:sN}
 \sN_\l=\{(I-\l T)^{-1}f:\,f \in {{\sH}[-]\ker F^{[*]}}\},\quad \l\in\cD;
\end{equation}
\begin{equation}\label{eq:sNe}
 \sN_\l=\{(\l I- T^{[*]})^{-1}g:\, g \in {\sH}[-]\ker G\}, \quad \lambda\in \cD_e.
\end{equation}
Comparison of \eqref{eq:sNl}, \eqref{eq:sNle} with \eqref{eq:sN}, \eqref{eq:sNe} shows that the subspaces $\sN_\lambda(V_*)$ are dense in $\sN_\lambda$ for all $\lambda\in\cD\cup\cD_e$. Therefore, the set
\[
\overline{\textup{span}}\,\{\sN_\lambda:\,\lambda\in\cD\}\quad(\textup{resp.}\quad
\overline{\textup{span}}\,\{\sN_\lambda:\,\lambda\in\cD_e\})
\]
coincides with the set
\[
\overline{\textup{span}}\,\{\sN_\lambda(V_*):\,\lambda\in\cD\}\quad(\textup{resp.}\quad
\overline{\textup{span}}\,\{\sN_\lambda(V_*):\,\lambda\in\cD_e\}).
\]
In view of the formulas \eqref{eq:sNl}, \eqref{eq:sNle} one obtains the equalities
 \begin{equation}\label{eq:span_D}
\overline{\textup{span}}\,\{\sN_\lambda:\,\lambda\in\cD\}=
\overline{\textup{span}}\,\{\ran (T^mF):\,m\ge 0\},
 \end{equation}
 \begin{equation}\label{eq:span_De}
\overline{\textup{span}}\,\{\sN_\lambda:\,\lambda\in\cD_e\}=
\overline{\textup{span}}\,\{\ran (T^{[*]})^nG^{[*]}):\,n\ge 0\}.
 \end{equation}
By Definition \ref{Def:simple} and the definition in \eqref{eq:Closely_C} this proves the statement (iv).
\end{proof}

\begin{remark}\label{g_inf}
With the assumptions of Theorem~\ref{T:UV2} there exist uniform limits
\begin{equation}\label{eq:g_inf}
  (\gamma_2(\infty):=)\lim_{\lambda\to\infty}\gamma_2(\lambda)=0,\quad
  \lim_{\lambda\to\infty}\lambda\gamma_2(\lambda)=G^{[*]}.
\end{equation}
\end{remark}

\subsection{General case}
\label{subsec3.4}

To prove the desired statements for a general unitary boundary pair some
preparatory lemmas will be used.

Let $\alpha\in \dC$ such that $|\alpha|\neq 1$ and define the
transform $M^{(\alpha)}\in \mathbf{B}(\sH^2)$ in the space
$({\sH}^2, J_{{\sH}^2})$ by
\begin{equation}\label{Malpha}
 M^{(\alpha)}=\frac{1}{\beta}\begin{pmatrix}I & -\alpha I \\ -\bar\alpha I & I \end{pmatrix},\quad \beta:=\sqrt{1-|\alpha|^2}.
\end{equation}
It is easy to check that $M^{(\alpha)}$ is a standard unitary
operator in the Kre\u{\i}n space $({\sH}^2, J_{{\sH}^2})$. Clearly,
$M^{(0)}=I$ and
\begin{equation}\label{Minv}
 (M^{(\alpha)})^{-1}=M^{(-\alpha)}, \quad |\alpha|\neq 1.
\end{equation}
Associated with $M^{(\alpha)}$ define a transform of the extended
complex plane $\dC\cup\{\infty\}$ by the formula
\[
 \mu^{(\alpha)}(\lambda):=\frac{\lambda-\bar\alpha}{1-\alpha \lambda}.
\]

\begin{lemma}\label{lem:Malpha}
The transform $M^{(\alpha)}$ in \eqref{Malpha} maps closed subspaces
of the Kre\u{\i}n space $({\sH}^2, J_{{\sH}^2})$ back to closed subspaces
and it satisfies
\begin{equation}\label{Malpha1}
    \left(M^{(\alpha)}(\cS^{-1})\right)=\left(M^{(\bar\alpha)}\left(\cS\right)\right)^{-1}
\end{equation}
and
\begin{equation}\label{Malpha2}
    M^{(\alpha)}\left(\cS^{-[*]}\right)=\left(M^{(\alpha)}(\cS)\right)^{-[*]}.
\end{equation}
In particular, $M^{(\alpha)}$ maps isometric (unitary, contractive,
expanding) relations $\cS$ in $\sH$ back to isometric
(unitary, contractive, expanding) relations $M^{(\alpha)}(\cS)$ in
$\sH$ and, moreover, for any closed linear relation $\cS$ in
$\sH$ the following statements hold with $1-\alpha
\lambda\neq 0$:
\begin{enumerate}
\def\labelenumi{\rm (\roman{enumi})}
\item $\ker(\cS-\lambda)=\ker(M^{(\alpha)}(\cS)-\mu^{(\alpha)}(\lambda)I)$, $\mul\cS = \ker(\alpha M^{(\alpha)}(\cS)+I)$;
\item $\ran(\cS-\lambda)=\ran(M^{(\alpha)}(\cS)-\mu^{(\alpha)}(\lambda)I)$, $\dom\cS = \ran(\alpha M^{(\alpha)}(\cS)+I)$;
\item $\lambda\in\sigma_j(\cS)$ $\Longleftrightarrow$
  $\mu^{(\alpha)}(\lambda)\in\sigma_j(M^{(\alpha)}(\cS))$ for $j=p,c,r$;
\item $\lambda\in\rho(\cS)$ $\Longleftrightarrow$
  $\mu^{(\alpha)}(\lambda)\in\rho(M^{(\alpha)}(\cS))$.
\end{enumerate}
\end{lemma}
\begin{proof}
The mapping $M^{(\alpha)}$ is unitary in the Kre\u{\i}n space
$({\sH}^2, J_{{\sH}^2})$ and, in fact, also unitary and selfadjoint as a
linear operator on the Hilbert space $\sH^2$. Therefore, it maps
closed subsets to closed subsets in $\sH^2$. The formula
\eqref{Malpha1} is checked with a straightforward calculation.
Moreover, for all $\wh f=\{f,f'\},\wh g=\{g,g'\}$
\begin{equation}\label{Malpha4}
 \left(J_{\sH^2}\wh f,\wh g\right)=\left(J_{\sH^2}M^{(\alpha)}\wh f,M^{(\alpha)}\wh
 g\right),
\end{equation}
which shows that if $\cS$ is isometric (contractive, expanding)
relation  in $\sH$ so is its image
$M^{(\alpha)}(\cS)$, since the expression in \eqref{Malpha4} with
$\wh f=\wh g\in \cS$ takes the value $=0$ (resp. $\geq 0$ and $\leq
0$). Moreover, \eqref{Malpha4} implies the property \eqref{Malpha2}.

The relation $\cS$ in the Pontryagin space $\sH$ is unitary
precisely when $S=S^{-[*]}$ and hence \eqref{Malpha4} implies that
then also $M^{(\alpha)}(\cS)=(M^{(\alpha)}(\cS))^{-[*]}$, i.e.,
$M^{(\alpha)}(\cS)$ is a unitary relation in $\sH$.

To prove the remaining assertions let $\{f,f'\}\in \cS$. Then
$\{f,f'-\lambda f\}\in (\cS-\lambda I)$ and this is equivalent to
\[
  \binom{\frac{1}{\beta}(f-\alpha f')}
  {\frac{\beta}{1-\alpha\lambda}(f'-\lambda f)}
  \in M^{\alpha}(\cS)-\mu^{(\alpha)}(\lambda) I,
\]
where $\alpha\in\dD$,  $1-\alpha\lambda\neq 0$.
This formula with $\lambda\in\dC$ gives the equalities for
$\ker(\cS-\lambda)$ and $\ran(\cS-\lambda)$. Analogously the choice
$\lambda=\infty$ corresponds to $\mu^{(\alpha)}(\infty)=-1/{\alpha}$
and this yields the formulas for $\mul \cS$ and $\dom \cS$ in
(i) and (ii).

The statements (iii) and (iv) follow from (i) and (ii) when applying
the definitions of the resolvent set $\rho(\cS)$ and the spectral
components $\sigma_j(\cS)$, $j=s,c,r$.
\end{proof}

In the next lemma the transform $M^{(\alpha)}\in \mathbf{B}(\sH^2)$ is
composed with a unitary boundary pair.

\begin{lemma}\label{lem:Galpha}
Let $V$ be a closed isometric operator in a Pontryagin space
${\sH}$, let $({\sL},\G)$ be a unitary boundary pair for $V$, and let
$\alpha\in\dC$, $|\alpha|\neq 1$. Then
\[
 V^{(\alpha)}:=M^{(\alpha)}(V)
\]
is also a closed isometric relation in $\sH$.
Moreover, $\mul V^{(\alpha)}=\{0\}$
precisely when $\alpha^{-1}\not\in \sigma_p(V)$,
and in this case:
\begin{enumerate}
\def\labelenumi{\rm (\roman{enumi})}
\item The composition
\begin{equation}\label{eq:Gamma_alpha}
     \G^{(\alpha)}:=\G\circ M^{(-\alpha)}
\end{equation}
defines a unitary boundary pair $({\sL},\G^{(\alpha)})$ for the
isometric operator $V^{(\alpha)}$.

\item The Weyl function $\Theta^{(\alpha)}$ and the $\gamma$-fields $\gamma_1^{(\alpha)}$, $\gamma_2^{(\alpha)}$ of the  unitary boundary pair $({\sL},\G^{(\alpha)})$ are connected to
  the Weyl function and the $\gamma$-fields  of the  unitary boundary pair $({\sL},\G)$ by
\begin{equation}\label{eq:theta_a}
\Theta^{(\alpha)}
\left(\mu(\alpha)\right)=  \Theta(\lambda),\quad
\quad\lambda\in\cD,
\end{equation}
\begin{equation}\label{eq:gamma_1a}
\gamma_1^{(\alpha)}
\left(\mu(\alpha)\right)=\frac{1-\lambda\alpha}{\beta}\gamma_1(\lambda),\quad\lambda\in\cD,
\end{equation}
\begin{equation}\label{eq:gamma_2a}
\gamma_2^{(\alpha)}
\left(\mu(\alpha)\right)=\frac{1-\lambda\alpha}{\beta}\gamma_2(\lambda),\quad\lambda\in\cD_e.
\end{equation}
\end{enumerate}
\end{lemma}

\begin{proof}
The statements concerning the linear relation $ V^{(\alpha)}$ are implied by Lemma~\ref{lem:Malpha}.

(i) By definition $M^{(-\alpha)}$ is a standard unitary operator in the
Kre\u{\i}n space $({\sH}^2, J_{{\sH}^2})$. Therefore, the composition
$\G^{(\alpha)}=\G\circ M^{(-\alpha)}$ is a unitary relation from the
Kre\u{\i}n space $({\sH}^2, J_{{\sH}^2})$ to the
Kre\u{\i}n space $({\sL}, J_{{\sL}})$.
It follows from the equivalence
\begin{equation}\label{equiv:G}
\{\wh f,\wh u\}\in\Gamma\Longleftrightarrow \{M^{(\alpha)}\wh f,\wh u\}\in \G^{(\alpha)}
\end{equation}
that $\ker
\G^{(\alpha)}=M^{(\alpha)}(V)$ and $\dom
\G^{(\alpha)}=M^{(\alpha)}(V_*)$; see \eqref{Minv}.

(ii) By Lemma~\ref{lem:Malpha}
\[%begin{equation}\label{eq:sigma_p}
\lambda\in\sigma_p(V_*)
\Leftrightarrow \mu^{(\alpha)}(\lambda)\in\sigma_p(V_*^{(\alpha)})
\]
and hence the defect subspaces $\sN_\omega(V_*^{(\alpha)}):=\ker( V_*^{(\alpha)} - \omega I)$ are connected with the defect subspaces $\sN_\lambda(V_*)$ by
\begin{equation}\label{equiv:sN}
\sN_\lambda(V_*)=\sN_{\mu^{({\alpha})}(\lambda)}(V_*^{(\alpha)}).
\end{equation}
Rewriting the  equivalence~\eqref{equiv:G} for vectors $f\in\sN_\lambda(V_*)$ we obtain
\begin{equation}\label{equiv:G2}
\left\{\binom{ f}{\lambda f},\binom{u_1}{u_2}\right\}\in\Gamma\Longleftrightarrow
 \left\{\frac{1-\lambda\alpha}{\beta}\binom{ f}{\mu^{({\alpha})}(\lambda)f},\binom{u_1}{u_2}\right\}\in \G^{(\alpha)}.
\end{equation}
In view of Definition~\ref{def:gamma}  this implies
\[
\gamma_1^{(\alpha)}(\mu^{({\alpha})}(\lambda))u_1=\frac{1-\lambda\alpha}{\beta}\gamma_1(\lambda)u_1,\quad \lambda\in\cD;
\]
\[
\gamma_2^{(\alpha)}(\mu^{({\alpha})}(\lambda))u_2=\frac{1-\lambda\alpha}{\beta}\gamma_2(\lambda)u_2,\quad \lambda\in\cD_e.
\]
By virtue of~\eqref{eq:WF2} this proves (ii).
\end{proof}

The next two theorems contain a full characterization of the class of Weyl functions $\Theta(\lambda)$ of boundary pairs.
In the first theorem it is shown that $\Theta(\lambda)$ belongs to the generalized Schur class $\cS_\kappa(\sL_1,\sL_2)$.

\begin{theorem}\label{thm:General}
Let $(\sL,\Gamma)$ be a unitary boundary pair for an isometric operator $V$ in a Pontryagin space $\sH$ and let $V_1=\ker\Gamma_1$, $V_2=\ker\Gamma_2$.
Then:
\begin{enumerate}
\def\labelenumi{\rm (\roman{enumi})}

\item  $V_1$ and $V_2$ are closed linear relations which are connected by $V_1=V_2^{-[*]}$ and, moreover, the sets
\begin{equation}\label{eq:D_D_e}
     \cD:=\rho(V_1)\cap \dD\quad\textup{and}\quad\cD_e:=\rho(V_2)\cap \dD_e
\end{equation}
coincide with the sets in \eqref{eq:cD} and they are connected by~\eqref{eq:DD_e}. In particular, they are nonempty and the sets
$\dD\setminus\cD$ and $\dD_e\setminus\cD_e$ contain at most $\kappa$ points.

\item  The $\gamma$-field $\gamma_1(\lambda)$ is holomorphic on $\cD$ with values in $\mathbf{B}(\sL_1,\sH)$.

\item  The $\gamma$-field $\gamma_2(\lambda)$ is holomorphic on $\cD_e$ with values in $\mathbf{B}(\sL_2,\sH)$. If $0\in\cD$ and $\gamma_2^\#$ is defined by
 \begin{equation}\label{eq:gam_sharp}
   \gamma_2^\#(\lambda):= \gamma_2(\bar{\lambda}^{-1})^*,\quad \lambda\in\cD,
 \end{equation}
then the following uniform limits exist
\begin{equation}\label{eq:g_inf2}
  (  \gamma_2^\#(0):=)\lim_{\lambda\to 0}\gamma_2^\#(\lambda)=0,\quad
  ((\gamma_2^\#)'(0):=)\lim_{\lambda\to 0}\frac{1}{\lambda}\gamma_2^\#(\lambda)(\in\mathbf{B}(\sH,\sL_2)).
\end{equation}
\item The Weyl function $\Theta(\lambda)$ is holomorphic on $\cD$,   takes values in $\mathbf{B}(\sL_1,\sL_2)$ for $\lambda\in\cD$ and belongs to the class $\cS_\kappa(\sL_1,\sL_2)$.
\item
For all $\lambda\in\cD$ the following relation holds
 \begin{equation}\label{eq:WF2}
 \left\{\begin{pmatrix}
      \g_1(\l)u_1 \\
     \l \g_1(\l)u_1
    \end{pmatrix},\begin{pmatrix}
      u_1 \\
     \Theta(\lambda) u_1
    \end{pmatrix}\right\}\in \Gamma,  \quad u_1 \in \sL_1,\,\, \l\in\cD,
  \end{equation}
and for all $\lambda\in\cD_e$ the following relation holds
 \begin{equation}\label{eq:WF1}
 \left\{\begin{pmatrix}
      \g_2(\l)u_2 \\
     \l\g_2(\l)u_2
    \end{pmatrix},\begin{pmatrix}
     \Theta^\#(\lambda)u_2 \\
       u_2
    \end{pmatrix}\right\}\in \Gamma,  \quad u_2 \in\sL_2,\,\, \l\in\cD_e,
  \end{equation}	
  where
  \[
   \Theta^\#(\l):=\Theta(1/\ov \l)^* \quad (\l\in \cD_e).
  \]
\item
Moreover, the following equalities are satisfied
\begin{equation}\label{E:1.8}
\g_1(\l)=\g_1(\m)+(\l-\m)(V_1-\l I)^{-1}\g_1(\m),\quad \l,\m\in\cD;
\end{equation}
\begin{equation}\label{E:1.8.2}
\g_2(\l)=\g_2(\m)+(\l-\m)(V_2-\l I
)^{-1}\g_2(\m),\quad \l,\m\in\cD_e.
\end{equation}
\end{enumerate}
\end{theorem}
\begin{proof}
(i) By Proposition~\ref{P:kappagamma} there is $\alpha\in\dD$, such that
\begin{equation}\label{eq:ker_V}
    \ker(I-\alpha V_2)=\{0\}.
\end{equation}
Then the linear relation $V^{(\alpha)}=M^{(\alpha)}V$ is isometric and single-valued. By Lemma~\ref{lem:Galpha} the pair
$(\sL,\Gamma^{(\alpha)})$ with $\Gamma^{(\alpha)}=\Gamma \circ M^{(-\alpha)}$
is a unitary boundary pair for the  isometric operator $V^{(\alpha)}$.

The linear relations $V_1^{(\alpha)}:=\ker\Gamma_1^{(\alpha)}$, $V_2^{(\alpha)}:=\ker\Gamma_2^{(\alpha)}$ are related with the linear relations $V_1$ and $V_2$ by the equalities
\begin{equation}\label{eq:V1}
V_1^{(\alpha)}=M^{(\alpha)}(V_1)=\left\{\begin{pmatrix}
                                   f-{\alpha}f' \\  -\overline{\alpha} f + f'
                                 \end{pmatrix}: \{f,f'\}\in V_1 \right\},
\end{equation}
\begin{equation}\label{eq:V2}
V_2^{(\alpha)}=M^{(\alpha)}(V_2)=\left\{\begin{pmatrix}
                                   f-{\alpha}f' \\  -\overline{\alpha} f + f'
                                 \end{pmatrix}: \{f,f'\}\in V_2 \right\}.
\end{equation}
By the choice of $\alpha$ we have $\mul V_2^{(\alpha)}=\{0\}$. Let $\cU^{(\alpha)}=\cJ\Gamma^{(\alpha)}$.
By Proposition~\ref{P:G1G2} $\mul \cU^{(\alpha)}=\{0\}$, which by Lemma~\ref{PontryIsom2}~(v) implies that
$\cU^{(\alpha)}$ is (the graph of) a unitary colligation $U^{(\alpha)}$.
Now an application of Theorem~\ref{T:UV} shows that $V_2^{(\alpha)}$ and $(V_1^{(\alpha)})^{-1}$ are graphs of closed bounded
operators defined everywhere on $\sH$ and $V_1^{(\alpha)}=(V_2^{(\alpha)})^{-[*]}$; see \eqref{E:V_12}.
In view of~\eqref{eq:V1} and~\eqref{eq:V2} the linear relations $V_1$ and $V_2$ are closed and \eqref{Malpha2} implies that $V_1=V_2^{-[*]}$; cf. Lemma~\ref{lem:Malpha}.
Moreover, it follows from Lemma~\ref{lem:Malpha} that with $j=1,2$,
\begin{equation}\label{eq:rho}
\lambda\in\rho(V_j)\Leftrightarrow  \mu^{(\alpha)}(\lambda)\in\rho(V_j^{(\alpha)}).
\end{equation}

Hence the set $\cD:=\rho(V_1)\cap \dD$ is nonempty  and  $\card(\dD\setminus\cD)\le\kappa$ since the same properties hold for
the set $\cD^{(\alpha)}:=\rho(V_1^{(\alpha)})\cap \dD$; see Theorem~\ref{T:UV}~(v).

Similarly, by Theorem~\ref{T:UV} (v) the set $\cD_e^{(\alpha)}:=\rho(V_2^{(\alpha)})\cap \dD_e$ is nonempty,  $\card(\dD_e\setminus\cD_e^{(\alpha)})\le\kappa$ and by~\eqref{eq:rho} this implies the corresponding statement for $\dD_e\setminus\cD_e$.

(ii) \& (iii) By Theorem~\ref{T:UV2} the $\gamma$-field $\gamma_1^{(\alpha)}(\lambda)$ is holomorphic on $\cD^{(\alpha)}$ with values in $\mathbf{B}(\sL_1,\sH)$ and
the $\gamma$-field $\gamma_2^{(\alpha)}(\lambda)$ is holomorphic on $\cD_e^{(\alpha)}$ with values in $\mathbf{B}(\sL_2,\sH)$. The desired statement for the $\gamma$-fields $\gamma_1(\lambda)$ and $\gamma_2(\lambda)$ follows from~\eqref{eq:gamma_1a} and \eqref{eq:gamma_2a}.

The existence of the limits in~\eqref{eq:g_inf2} is obtained from~\eqref{eq:gamma_1a} and Remark \ref{g_inf}.

(iv) This statement is implied by Theorem~\ref{T:UV2} (iii) and ~\eqref{eq:theta_a}.

(v) The relation~\eqref{eq:WF2} is implied by~\eqref{D:G2},~\eqref{D:M1} and items (ii), (iv).

To prove~\eqref{eq:WF1}, let us choose for $\lambda\in\cD_e$, $u_2\in\sL_2$, and a unique vector $v\in\sL_1$, such that
\begin{equation}\label{eq:gam_Theta}
\left\{\begin{pmatrix}
      \g_2(\l)u_2 \\
      \l\g_2(\l)u_2
    \end{pmatrix},\begin{pmatrix}
     v \\
       u_2
    \end{pmatrix}\right\}\in \Gamma.
\end{equation}
By applying \eqref{E:1.1} or \eqref{Eqgamma1} to the elements \eqref{eq:gam_Theta} with $\lambda\in\cD_e$ and \eqref{eq:WF2}
with $\lambda$ replaced by $\bar\l^{-1}\in\cD$ it is seen that for all $u_1\in \sL_1$
\[
\begin{split}
    0&=[\g_2(\l)u_2,\g_1(\bar\l^{-1})u_1]_{\sH} - [\lambda\g_2(\l)u_2,\bar\l^{-1}\g_1(\bar\l^{-1})u_1]_{\sH}\\
    &=(v,u_1)_{\sL_1} - (u_2,\Theta(\bar\l^{-1})u_1)_{\sL_2}
\end{split}
\]
and hence $v =\Theta(\bar\l^{-1})^*u_2=\Theta^\#(\lambda)u_2$.

(vi)
To prove the identity \eqref{E:1.8} consider the vector
$f_\m=\g_1(\m)u_1\in\sN_\m(V_*)$, where $u_1\in\sL_1$, $\m\in \cD$.
Then there exists $u_2\in\sL_2$ such that
\begin{equation}\label{eq:f_mu}
    \left\{{\begin{pmatrix}
      f_\m \\
      \m f_\m
    \end{pmatrix}},\begin{pmatrix}
      u_1 \\
      u_2
    \end{pmatrix}\right\}\in \Gamma,\quad \m\in \cD.
    \end{equation}
For $\l \in \cD$ consider the vector
\[%begin{equation}\label{01}
\wh f_\l=\begin{pmatrix} f_\m \\ \m f_\m
\end{pmatrix}+\wh g, \quad\textup{where}\quad
\wh g=(\l-\m)\begin{pmatrix} (V_1-\l I)^{-1}f_\m
\\I + \l(V_1-\l I)^{-1}f_\m\end{pmatrix}
\in V_1\subset V_*.
\]%end{equation}
Direct calculations show that $\wh f_\l\in\wh\sN_\l(V_*)$. Since $\wh g=\begin{pmatrix}
      g \\
      g'
    \end{pmatrix}\in V_1$ there exists $v_2\in\sL_2$ such that
\begin{equation}\label{eq:g_V1}
    \left\{{\begin{pmatrix}
      g \\
      g'
    \end{pmatrix}},\begin{pmatrix}
      0 \\
      v_2
    \end{pmatrix}\right\}\in \Gamma.
    \end{equation}
It follows from~\eqref{eq:f_mu} and~\eqref{eq:g_V1} that
\[
    \left\{{\begin{pmatrix}
        f_\l \\
       \l f_\l
    \end{pmatrix}},\begin{pmatrix}
      u_1 \\
      u_2+v_2
    \end{pmatrix}\right\}=
    \left\{{\begin{pmatrix}
        f_\m+g \\
       \m f_\m+g'
    \end{pmatrix}},\begin{pmatrix}
      u_1 \\
      u_2+v_2
    \end{pmatrix}\right\}\in \Gamma
\]
and hence $\gamma_1(\lambda)u_1=f_\lambda$. This proves \eqref{E:1.8}.

The equality \eqref{E:1.8.2} is proved similarly.
\end{proof}

There is an analog for the notion of transposed boundary triple (see \cite{DHMS06}) for boundary pairs of isometric operators.
In the present case this notion contains the second boundary triple associated
with a dual pair $\{V,V^{-1}\}$ as defined in \cite{MM03} in the case of ordinary boundary triples for Hilbert space isometries.
For this purpose the notion of \textit{transposed boundary pair} $(\sL_2\oplus\sL_1,\Gamma^\top)$ is introduced for boundary pairs of
isometric operators and its basic properties are established in the next proposition.

\begin{proposition}\label{prop:trans}
Let $(\sL,\Gamma)$ be a unitary boundary pair for an isometric operator $V$ in a Pontryagin space $\sH$, let $V_1=\ker\Gamma_1$, $V_2=\ker\Gamma_2$, let $\gamma_1(\lambda)$ and $\gamma_2(\lambda)$ be the $\gamma$-fields of $(\sL,\Gamma)$, let $\Theta(\lambda)$ be the Weyl function of $(\sL,\Gamma)$, and define
\begin{equation}\label{eq:G_trans}
        \Gamma^\top=\left\{\left\{\binom{f'}{f},\binom{u_2}{u_1}\right\}:\,
        \left\{\binom{f}{f'},\binom{u_1}{u_2}\right\}\in\Gamma\right\}.
\end{equation}
Then:
\begin{enumerate}
\def\labelenumi{\rm (\roman{enumi})}

\item $(\sL_2\oplus\sL_1,\Gamma^\top)$ is a unitary boundary pair for $V^{-1}$.
\item  $V_1^\top:=\ker\Gamma_1^\top=V_2^{-1}$ and $V_2^\top:=\ker\Gamma_2^\top=V_1^{-1}$.
\item  The Weyl function $\Theta^\top(\lambda)$ and the $\gamma$-fields $\gamma_1^\top(\lambda)$ and $\gamma_2^\top(\lambda)$, corresponding to the pair $(\sL_2\oplus\sL_1,\Gamma^\top)$ are connected with $\Theta(\lambda)$, $\gamma_1(\lambda)$ and $\gamma_2(\lambda)$ by
\begin{equation}\label{eq:theta_top}
  \Theta^\top(\lambda)=\Theta(\bar\lambda)^*,\quad \l\in\overline{\cD}:=\{\ov\mu:\mu \in \cD\},
\end{equation}
\begin{equation}\label{eq:gamma_top}
  \gamma_1^\top(\lambda)=\frac{1}{\lambda}\gamma_2\left(\frac{1}{\lambda}\right),\quad\l\in\overline{\cD};\qquad
 \gamma_2^\top(\lambda)=\frac{1}{\lambda}\gamma_1\left(\frac{1}{\lambda}\right),\quad\l\in\overline{\cD}_e.
\end{equation}
\end{enumerate}
\end{proposition}
\begin{proof}
(i)  This statement is implied by the equality $V=\ker \G$ and the following identity
\begin{equation}\label{eq:may9a}
[f',g']_{{\sH}}-[f,g]_{{\sH}}=(u_2,v_2)_{{\sL}_2}-(u_1,v_1)_{{\sL}_1},
\end{equation}
which is valid for all $\{\wh f, \wh u\},\{\wh g, \wh v\}\in \G$; see Definition~\ref{D:1}.

(ii) The statement is clear from the definition of $(\sL_2\oplus\sL_1,\Gamma^\top)$.

(iii) If $\l\in\overline{\cD}$ then $1/\l\in{\cD_e}$ and one obtains from~\eqref{eq:WF1}
\[%begin{equation}\label{eq:may9b}
    \left\{{\begin{pmatrix}
      \gamma_2(1/\lambda)u_2 \\
      1/\lambda\gamma_2(1/\lambda)u_2
    \end{pmatrix}},\begin{pmatrix}
      \Theta^\#(1/\lambda)u_2 \\
      u_2
    \end{pmatrix}\right\}\in \Gamma,\quad u_2\in\sL_2,\quad\l\in \overline{\cD}.
\]%end{equation}
In view of~\eqref{eq:G_trans} this implies
\[%begin{equation}\label{eq:may9b}
    \left\{{\begin{pmatrix}
      1/\lambda\gamma_2(1/\lambda)u_2 \\
      \gamma_2(1/\lambda)u_2
    \end{pmatrix}},\begin{pmatrix}
      u_2\\
      \Theta^\#(1/\lambda)u_2
      \end{pmatrix}\right\}\in \Gamma^\top,\quad u_2\in\sL_2,\quad\l\in \overline{\cD}
\]%end{equation}
and hence
\[
 \gamma_1^\top(\lambda)=\frac{1}{\lambda}\gamma_2\left(\frac{1}{\lambda}\right)
\quad\mbox{ and }
\quad
  \Theta^\top(\lambda)=\Theta^\#(1/\lambda)=\Theta(\bar\lambda)^*,\quad \l\in\overline{\cD}.
\]
The second equality in~\eqref{eq:gamma_top} is proved similarly.
\end{proof}

\subsection{Realization theorem}
The converse statement to Theorem~\ref{thm:General} contains the main realization result for the generalized Schur class $\cS_\kappa(\sL_1,\sL_2)$:
every function $\Theta(\lambda)$ from the class $\cS_\kappa(\sL_1,\sL_2)$ can be realized as the Weyl function of
a boundary pair for some isometric operator $V$ in a Pontryagin space.

\begin{theorem}\label{thm:Inv}
Let $s(\cdot)\in \cS_\k({\sL}_1,{\sL}_2)$ with the domain of holomorphy $\gh_s(\subset\dD)$. Then there exists a simple isometric operator $V$ in a Pontryagin space ${\sH}$
and a unitary boundary pair $({\sL}_1\times{\sL}_2, \G)$ such that the corresponding Weyl function $\Theta(\l)$ coincides with $s(\l)$ on $\gh_s$.
\end{theorem}

\begin{proof}
1) First assume that $s(\cdot)$ is holomorphic at $0$. Then by
\cite[Theorem~2.3.1]{ADRS97} there exists a closely connected
unitary colligation $U=\begin{pmatrix}
  T&F\\
  G&H
\end{pmatrix}$
such that its characteristic function coincides with $s(\l)$ for all $\lambda\in\gh_s$.

Let $V$ be defined by the formula $V=T^{[*]}\upharpoonright \ker F^{[*]}$ and let the
linear relation $\G$ be defined by  \eqref{E:Gstar}. Then by Theorem ~\ref{T:UV}
the pair $({\sL},\G)$ is a unitary boundary pair for $V_*$ and the corresponding Weyl
function $\Theta(\l)$ coincides with $s(\l)$  on $\gh_s$.

2) Assume that $s(\l)$ is holomorphic at $\ov\alpha  \in\gh_s$. Consider a new operator function
\begin{equation}\label{szero}
s^{(\alpha)}(\zeta)=s\sk({\frac{\zeta+\ov\alpha}{1+ \alpha\zeta}}).
\end{equation}
Since $s^{(\alpha)}(\cdot)\in \cS_\k({\sL}_1,{\sL}_2)$ and also is holomorphic at $0$ there exist
a Pontryagin space ${\sH}$, a simple isometric operator $V^{(\alpha)}$ in ${\sH}$, and a unitary boundary pair
$({\sL}, \G^{(\alpha)})$ such that the corresponding Weyl function $\Theta^{(\alpha)}(\zeta)$
coincides with $s^{(\alpha)}(\zeta)$.

By Lemma~\ref{lem:Galpha} the pair
$(\sL,\Gamma)$ with $\Gamma:=\Gamma^{(\alpha)} \circ M^{(\alpha)}$
is a unitary boundary pair for the simple isometric operator
\[
V=\ker\Gamma=M^{(-\alpha)}V^{(\alpha)}=
\left\{\begin{pmatrix}
                                   f+{\alpha}f' \\  \overline{\alpha} f + f'
\end{pmatrix}: \{f,f'\}\in V^{(\alpha)} \right\}.
\]
The domains $V_*:=\dom\Gamma$  and $V_*^{(\alpha)}:=\dom\Gamma^{(\alpha)}$  of $\Gamma$ and $\Gamma^{(\alpha)}$ are connected by
\[
V_*:=\dom\Gamma=M^{(-\alpha)}V_*^{(\alpha)}=
\left\{\begin{pmatrix}
                                   f+{\alpha}f' \\  \overline{\alpha} f + f'
\end{pmatrix}: \{f,f'\}\in V_*^{(\alpha)} \right\},
\]
The defect subspace $\sN_\lambda(V_*):=\ker(V_*-\lambda I)$ and
$\sN_\omega(V_*^{(\alpha)}):=\ker(V_*^{(\alpha)}-\omega I)$ are connected by~\eqref{equiv:sN} and the equivalence~\eqref{equiv:G2} holds.
It follows from~\eqref{equiv:G2} and Definition~\ref{char} that the Weyl functions of the  unitary boundary pairs $({\sL},\G)$ and $({\sL},\G^{(\alpha)})$ are connected by
\begin{equation}\label{eq:theta1}
  \Theta(\lambda)
=\Theta^{(\alpha)}
\left(\frac{\lambda-\ov{\alpha}}{1-{\alpha}\lambda}\right),\quad\lambda\in\gh_s.
\end{equation}
Setting $\zeta=\frac{\lambda-\ov{\alpha}}{1-{\alpha}\lambda}$ one obtains  $\lambda={\frac{\zeta+\ov\alpha}{1+ \alpha\zeta}}$ and hence by~\eqref{eq:theta1} and~\eqref{szero}
\[
  \Theta(\lambda)=\Theta^{(\alpha)}(\zeta)=s^{(\alpha)}(\zeta)=s\sk({\frac{\zeta+\ov\alpha}{1+ \alpha\zeta}})=s(\lambda),\quad\lambda\in\gh_s.
\]
This completes the proof.
\end{proof}

\begin{remark}
  Notice that a simple isometric operator $V$ in Theorem~\ref{thm:Inv}
and a   unitary boundary pair $({\sL}_1\times{\sL}_2, \G)$  are determined by the Weyl function $s(\cdot)$ uniquely, up to a unitary equivalence.
The latter means that if there is another simple isometric operator $V'$ in a Pontryagin space ${\sH}'$ and a    unitary boundary pair $({\sL}_1\times{\sL}_2, \G')$, such that its Weyl function coincides with $s(\cdot)$, then there exists a unitary operator $W$ from
 ${\sH}'$ to  ${\sH}'$, such that
 \[
 V'=WVW^{-1},\quad \Gamma'=\Gamma W^{-1}.
 \]
\end{remark}

Next Theorem~\ref{thm:General} is specialized to the case where $\sH$ is a Hilbert space and $(\sL,\Gamma)$ is a unitary boundary pair for
an isometric operator $V$ in $\sH$. When combined with Theorem~\ref{thm:Inv} we get a general realization result for operator valued
Schur functions $s(\cdot)\in \cS({\sL}_1,{\sL}_2)$ as a Weyl function of a unitary boundary pairs for an isometric operator in a Hilbert space.

\begin{theorem}\label{thm:Hilbert}
Let $\sH$ and $\sL=\sL_1\times\sL_2$ be Hilbert spaces, let $V$ be an isometric operator in $\sH$, let
$(\sL,\Gamma)$ be a unitary boundary pair for $V$ and let $V_1$, $V_2$ be defined by~\eqref{eq:A12}.
Then:
\begin{enumerate}
\def\labelenumi{\rm (\roman{enumi})}

\item $V_1^{-1}$ and $V_2$ are contractive operators in $\mathbf{B}(\sH)$ and they are connected by
\begin{equation}\label{V1V2}
    V_2=V_1^{-*}.
\end{equation}
\item The $\gamma$-field $\gamma_1(\lambda)$ is holomorphic on $\dD$ with values in $\mathbf{B}(\sL_1,\sH)$
and satisfies the identity \eqref{E:1.8} for all $ \l,\m\in\dD$.
\item The $\gamma$-field $\gamma_2(\lambda)$ is holomorphic on $\dD_e$ with values in $\mathbf{B}(\sL_2,\sH)$
and satisfies the identity \eqref{E:1.8.2} for all $ \l,\m\in\dD_e$.

\item The Weyl function $\Theta(\lambda)$ of $V$  corresponding to the boundary pair $(\sL, \G)$ belongs to the Schur class $\cS(\sL_1,\sL_2)$.
\end{enumerate}
Conversely, for every function $s(\cdot)$  from the Schur class $\cS({\sL}_1,{\sL}_2)$ there exists a simple
isometric operator $V$ in a Hilbert space ${\sH}$ and a unitary boundary pair $(\sL, \G)$
for $V$ such that the corresponding Weyl function $\Theta(\cdot)$ coincides with $s(\cdot)$.
\end{theorem}
\begin{proof}
(i) According to Corollary \ref{Cor:G1G2} $V_1^{-1}$ and $V_2$ are contractive operators.
From item (i) in Theorem~\ref{thm:General} one concludes that
\[
\dD\subseteq  \rho(V_2^{-1}) \quad\textup{and}\quad \dD_e\subseteq \rho(V_1^{-1}).
\]
In particular, $0\in \rho(V_2^{-1})$ so that $V_2\in \mathbf{B}(\sH)$.
On the other hand, by Proposition~\ref{P:G1G2} one has the inclusion $V_2\subseteq  V_1^{-*}$.
Since $\dD_e\subseteq \rho(V_1^{-*})$, the equality \eqref{V1V2} must prevail and, hence,
$V_1^{-*}\in \mathbf{B}(\sH)$ $\Leftrightarrow$  $V_1^{-1}\in \mathbf{B}(\sH)$.

The assertions (ii)--(iv) are now obtained directly from Theorem~\ref{thm:General}.

The last statement follows from Theorem~\ref{thm:Inv} by taking $\kappa=0$.
\end{proof}

\subsection{Classification of unitary boundary pairs}
\label{subsec3.5}

We start by collecting some main properties of unitary boundary pairs in the next proposition.

\begin{proposition}\label{P:GeneralBpair}
Let $V$ be an isometric operator in the Pontryagin space ${\sH}$ and
let $({\sL},\G)$ be a unitary boundary pair for $V$. Then $\G$ and its
components $\Gamma_1$ and $\Gamma_2$ defined by~\eqref{eq:G_1}, \eqref{eq:G_2} admit the following properties:
\begin{enumerate}
\def\labelenumi{\rm (\roman{enumi})}

\item $\Gamma_1$ and $\Gamma_2$ are closed linear relations with
$\ran \Gamma_1=\sL_1$, $\ran \Gamma_2=\sL_2$ and, moreover, $\mul
\Gamma_1=\{0\}$ $\Leftrightarrow$ $\mul \Gamma_2=\{0\}$
$\Leftrightarrow$ $\mul \Gamma=\{0\}$;

\item $V_1=\ker\Gamma_1$ and $V_2=\ker\Gamma_2$ have nonempty resolvent sets and $V_1=V_2^{-[*]}$;

\item $V_*=\dom \Gamma$ admits the decompositions
\[
  V_*=V_1 \hplus \wh{\sN}_\l(V_*), \quad  \lambda\in \cD=\rho(V_1)\cap
  \dD,
\]
\[
  V_*=V_2 \hplus \wh{\sN}_\l(V_*), \quad \lambda\in
  \cD_e=\rho(V_2)\cap\dD_e .
\]
\end{enumerate}
\end{proposition}
\begin{proof}
(i) These properties were proven in Proposition~\ref{P:G1G2}.

(ii) The fact that $\rho(V_1)$ and $\rho(V_2)$ are nonempty and the equality $V_2=V_1^{-*}$ were proven in Theorem \ref{thm:General}.

(iii) This is a direct consequence of (ii); see e.g. \cite[Lemma~4.1]{HSSW07}.
\end{proof}

\begin{remark}
Proposition \ref{P:GeneralBpair} shows that in a Pontryagin space
${\sH}$ every unitary boundary pair $({\sL},\G)$ (as well as its
transposed boundary pair) of an isometric operator $V$ can be seen
as an analog of so-called ($B$-)generalized boundary triple, since
the component mappings $\G_1$ and $\G_2$ are surjective and the
corresponding kernels $V_1=\ker \G_1$ and $V_2=\ker \G_2$ are closed
extensions of $V$ with nonempty resolvent sets; see \cite{DM95,DHMS06}
and \cite{DHM17,DHM2020a,DHM2020b} for some further developments.
\end{remark}

\begin{proposition}\label{P:1.2}
The following relations hold:
\begin{equation}\label{M2}
\sfS^\Theta_\m(\l):=\frac{I-\Theta(\m)^* \Theta(\l)}{1-\l\ov\m}=\g_1(\m)^\zx
\g_1(\l),\quad\l,\m\in\cD;
\end{equation}
\begin{equation}\label{M1}
\sfS^\Theta_\m(\l):=\frac{I-\Theta^\#(\m)^*
\Theta^\#(\l)}{1-\l\ov\m}=-\g_2(\m)^\zx \g_2(\l),\quad\l,\m\in\cD_e;
\end{equation}
\begin{equation}\label{M3}
\sfS^\Theta_\m(\l):=\frac{\Theta(\l)-\Theta^\#(\m)^* }{1-\l\ov\m}=-\g_2(\m)^\zx
\g_1(\l),\quad\l\in\cD,\text{ }\m\in\cD_e;
\end{equation}
\begin{equation}\label{M4}
\sfS^\Theta_\m(\l):=\frac{\Theta^\#(\l)-\Theta(\m)^*}{1-\l\ov\m}=\g_1(\m)^\zx
\g_2(\l),\quad\l\in\cD_e, \text{ }\m\in\cD.
\end{equation}
\end{proposition}
\begin{proof}
Let $\l,\m\in\cD$ and $u_1,v_1\in\sL_1$. Then by (\ref{eq:WF2})
$$
    \left\{{\begin{pmatrix}
      \gamma_1(\lambda)u_1 \\
      \l \gamma_1(\lambda)u_1
    \end{pmatrix}},
    \begin{pmatrix}
      u_1 \\
      \Theta(\lambda)u_1
    \end{pmatrix}\right\}\in \Gamma,\quad
        \left\{{\begin{pmatrix}
      \gamma_1(\mu)v_1 \\
      \mu\gamma_1(\mu)v_1
    \end{pmatrix}},
    \begin{pmatrix}
      v_1 \\
      \Theta(\mu)v_1
    \end{pmatrix}\right\}\in \Gamma
$$
and the identity \eqref{Eqgamma1} applied to these vectors yields
$$
(1-\l\ov\m)[\g_1(\l)u_1,\g_1(\m)v_1]_\sH=[u_1,v_1]_{\sL_1} - [\Theta(\l)u_1,\Theta(\m)v_1]_{\sL_2}.
$$
This proves the equality (\ref{M2}).

Let now $\l\in\cD$, $\m\in\cD_e$ and let $u_1\in\sL_1$ and
$v_2\in\sL_2$. Then by (\ref{eq:WF2})
$$
    \left\{{\begin{pmatrix}
      \gamma_1(\lambda)u_1 \\
      \lambda\gamma_1(\lambda)u_1
    \end{pmatrix}},
    \begin{pmatrix}
      u_1 \\
      \Theta(\lambda)u_1
    \end{pmatrix}\right\}\in \Gamma,\quad
        \left\{{\begin{pmatrix}
      \gamma_2(\mu)v_2 \\
      \mu\gamma_2(\mu)v_2
    \end{pmatrix}},
    \begin{pmatrix}
     \Theta^\#(\mu)v_2 \\
       v_2
    \end{pmatrix}\right\}\in \Gamma.
$$
By applying \eqref{Eqgamma1} to these vectors one arrives at
$$
(1-\l\ov\m)[\g_1(\l)u_1,\g_2(\m)v_2]_\sH=[u_1,\Theta^\#(\m)v_2]_{\sL_1} - [\Theta(\l)u_1,v_2]_{\sL_2}.
$$
This yields (\ref{M3}). The proof of \eqref{M1} and \eqref{M4}  is
analogous.
\end{proof}

\begin{proposition}\label{P:SpecialBpair}
Let $({\sL},\G)$ be a unitary boundary pair for an isometric operator $V$
in the Pontryagin space ${\sH}$. Then
\[
 \mul \Gamma_1=\ker \g_1(\l) \; (\l\in\cD) \quad\text{and}\quad \mul \Gamma_2=\ker
\g_2(\l) \;  (\l\in\cD_e).
\]
Moreover, the following statements are equivalent:
\begin{enumerate}
\def\labelenumi{\rm (\roman{enumi})}

\item $\Gamma$ is single valued (i.e. $\mul\Gamma=\{0\}$);

\item $\ran \Gamma$ is dense in $\sL_1\times\sL_2$;

\item  $\ker \g_1(\l)=\{0\}$ for some (equivalently for all) $\l\in\cD$;

\item $\ker \g_2(\l)=\{0\}$ for some (equivalently for all) $\l\in\cD_e$.
\end{enumerate}
If, in addition, the operator $V$ is simple then the conditions (i)--(iv)\! are\! equivalent\! to
\begin{enumerate}
  \item[(v)] $\bigcap\limits_{\mu\in\cD\cup\cD_e}\ker\sfS^\Theta_\m(\l)=\{0\}$ for some (equivalently for
all) $\l\in\cD\cup \cD_e$.
\end{enumerate}
\end{proposition}
\begin{proof}
(i) $\Leftrightarrow$ (ii) This is item (ii) in Proposition \ref{cor:ker_G}.

(i) $\Leftrightarrow$ (iii), (iv) The formula \eqref{E:g12} and Definition \ref{def:gamma} of the $\g$-fields
shows that $\ker \g_1(\l)=\mul\G_1$ for $\l\in\cD$ and $\ker \g_2(\l)=\mul\G_2$ for $\l\in\cD_e$.
Now the statement follows from item (i) in Proposition \ref{P:GeneralBpair}.

To prove the statement in (v) first observe that with $\lambda\in\cD$ the inclusions
\[%begin{equation}\label{eq:may9e}
  \ker\gamma_1(\lambda)\subseteq \ker\sfS^\Theta_\m(\l)\;\textup{ for all }\; \mu\in\cD\cup \cD_e
\]%end{equation}
are clear from \eqref{M2} and \eqref{M3} in Proposition \ref{P:1.2}. Thus
\begin{equation}\label{eq:may9e}
  \ker\gamma_1(\lambda)\subseteq \bigcap\limits_{\mu\in\cD\cup\cD_e}\ker\sfS^\Theta_\m(\l), \quad \l\in\cD.
\end{equation}
As to the reverse inclusion apply Proposition \ref{P:1.2} again to see that for all $\l,\m\in\cD$ and $u_1,v_1\in\sL_1$,
\begin{equation}\label{eq:may9f}
  (\sfS^\Theta_\m(\l)u_1,v_1)_{\sL_1}=[\gamma_1(\lambda)u_1,\gamma_1(\m)v_1]_\sH
\end{equation}
and for all $\l\in\cD$, $\nu\in\cD_e$ and $u_1\in\sL_1$,  $v_2\in\sL_2$,
\begin{equation}\label{eq:may9g}
  (\sfS^\Theta_\m(\l)u_1,v_2)_{\sL_2}=[\gamma_1(\lambda)u_1,\gamma_2(\mu)v_2]_\sH.
\end{equation}
Now assume that the operator $V$ is simple.
Then
\[
\ov\spn\{\gamma_1(\m)\sL_1,\, \gamma_2(\nu)\sL_2:\,\mu\in\cD,\,\nu\in\cD_e\}=\sH
\]
and by virtue of~\eqref{eq:may9f} and~\eqref{eq:may9g} the reverse inclusion in~\eqref{eq:may9e} follows.
This proves the equivalence
(iv) $\Leftrightarrow$ (v) for $\lambda\in\cD$ when $V$ is simple. Similarly one proves the equivalence (iii) $\Leftrightarrow$ (v)
for $\lambda\in\cD_e$.
\end{proof}

\begin{proposition}\label{P:OrdinaryBpair}
Let $({\sL},\G)$ be a unitary boundary pair for an isometric operator $V$
in the Pontryagin space ${\sH}$. Then the following statements are
equivalent:
\begin{enumerate}
\def\labelenumi{\rm (\roman{enumi})}

\item $\ran \Gamma=\sL_1\times\sL_2$;

\item $\dom \Gamma=V^{-[*]}$ and $\mul\G=\{0\}$;

\item $\ker \g_1(\l)=\{0\}$ and $\ran \g_1(\l)$ is closed for some (equivalently for all) $\l\in\cD$;

\item $\ker \g_2(\l)=\{0\}$ and $\ran \g_2(\l)$ is closed for some (equivalently for all) $\l\in\cD_e$.
\end{enumerate}
If $0\in \rho(\sfS^\Theta_\l(\l))$ for some  $\l\in\cD\cup \cD_e$, then  the conditions (i)--(iv) hold.

If one of the condition (i)--(iv) is satisfied then $({\sL},\G)$ is an ordinary boundary pair for $V$.
\end{proposition}
\begin{proof}
(i) $\Leftrightarrow$ (ii) This is item (iii) in Proposition \ref{cor:ker_G}.

(ii) $\Leftrightarrow$ (iii), (iv) Definition \ref{def:gamma} of the $\g$-fields
shows that $\ran \g_1(\l)=\sN_\lambda(V_*)$ for $\l\in\cD$ and $\ran \g_2(\l)=\sN_\lambda(V_*)$ for $\l\in\cD_e$.
Now the decompositions of $V_*=\dom \G$ in item (iii) of Proposition \ref{P:GeneralBpair} imply that
$\sN_\lambda(V_*)$ is closed for some $\l\in\cD\cup\cD_e$ if and only if $V_*=V^{-[*]}$.
This combined with Proposition \ref{P:SpecialBpair} gives the stated equivalences.

The last implication follows from Proposition \ref{P:1.2}.
Indeed, if $0\in \rho(\sfS^\Theta_\l(\l))$ for some  $\l\in\cD$,
then by Proposition~\ref{P:SpecialBpair}  $\ker \g_1(\l)=\{0\}$ and by \eqref{M2} $\ran \g_1(\l)$ is closed.
Similarly, if $0\in \rho(\sfS^\Theta_\l(\l))$ for some  $\l\in\cD_e$, then again by Proposition~\ref{P:SpecialBpair}
$\ker \g_2(\l)=\{0\}$ and by \eqref{M1} $\ran \g_2(\l)$ is closed.

Finally, the fact that $({\sL},\G)$ is an ordinary boundary pair for $V$ is clear from the properties in (ii).
\end{proof}

The next example shows that the condition (v) in Proposition~\ref{P:SpecialBpair} cannot be replaced by a single condition $\ker\sfS^\Theta_\l(\l)=\{0\}$.

\begin{example}
  Let $\sH=\dC^4$ with the skew-diagonal fundamental symmetry $J=(\delta_{j,5-k})_{j,k=1}^4$, where $\delta_{j,k}$ is the Kronecker delta and let us set
  $e_j=(\delta_{j,k})_{k=1}^4$, $(j=1,2,3,4)$.
Let $V$ be an isometry in $\sH$ which maps $e_1$ into $e_2$.
Then the defect subspaces of $V$
\begin{equation}\label{eq:may11a}
\sN_\lambda=\left\{(f_1,f_2, f_3,\lambda f_3)^\top:\, f_1,f_2,f_3\in\dC\right\}
 \end{equation}
are degenerate for all $\lambda\in\dC$ and thus, the operator $V$ is not standard.

The linear relation $V^{-[*]}$ consists of vectors
\[
\wh f=\begin{pmatrix}
f \\
f'
\end{pmatrix},\quad
\mbox{ where }\quad f=\sum_{j=1}^{4}f_j e_j,\quad f'=\sum_{j=1}^{4}f'_j e_j,\quad f_3'=f_4.
\]
Therefore, the left part of the identity~\eqref{E:1.1} for $\wh f=\wh g$ takes the form
\[
(f_1-f_2')\overline{f_4}
+f_2\overline{f_3}+f_3\overline{f_2}+f_4(\overline{f_1}-\overline{f_2'})-f_1'\overline{f_4'}-f_4'\overline{f_1'}
\]
and can be rewritten in the diagonal form
\[
\frac12\hspace{-2pt}\left\{|f_1'-\!f_4'|^2\hspace{-3pt}-\!|f_1'+f_4'|^2\hspace{-3pt} + |f_1-\! f_2'+f_4|^2\hspace{-3pt}
 -\! |f_1-f_2'-\!f_4|^2\hspace{-3pt}+|f_2+f_3|^2\hspace{-3pt} -\! |f_2-\! f_3|^2\right\}\hspace{-2pt}.
\]
 Hence a single-valued boundary triple $(\sL_1\oplus\sL_2, \G_1,\G_2)$ can be chosen as follows
 \begin{equation}\label{eq:may11b}
   \sL_1=\sL_2=\dC^3,\quad
   \G_1 \wh f=\frac{1}{\sqrt{2}}\begin{pmatrix}
                            f'_1-f'_4\\
                           f_1-f'_2+f_4 \\
                           f_2+f_3
                         \end{pmatrix},\quad
   \G_2 \wh f=\frac{1}{\sqrt{2}}\begin{pmatrix}
                            f'_1+f'_4\\
                            f_1-f'_2-f_4\\
                           f_2-f_3
                         \end{pmatrix}.
 \end{equation}
 Then for $|\l|<1$ one obtains from~\eqref{eq:may11a} and \eqref{eq:may11b}
 $$
 \Theta(\l)=\frac{1}{3\l^2}\begin{pmatrix}
                         \l^2&2\l^3&2\l^4 \\
                         2\l&\l^2&-2\l^3 \\
                         2&-2\l&\l^2
                       \end{pmatrix}
 $$
 and hence
 $$
    \sfS^\Theta_\omega(\l)=\frac{-2}{9\l^2\ov\omega^2}\Omega^*
                       \begin{pmatrix}
                         2(1+2\l\ov\omega)&-(2+\l\ov\omega)&1-\l\ov\omega \\
                         -(2+\l\ov\omega)&2(1-\l\ov\omega)&-(1+2\l\ov\omega) \\
                         1-\l\ov\omega&-(1+2\l\ov\omega)&-2(2+\l\ov\omega)
                       \end{pmatrix}\Lambda,
       $$
where $\Lambda=\mbox{diag }(1,\lambda,\lambda^2)$, $\Omega=\mbox{diag }(1,\omega,\omega^2)$.
Notice that in this example $\det \sfS^\Theta_\omega(\l)\equiv 0$ for all $\l,\omega\in\dD\backslash\{0\}$, while
$\bigcap\limits_{\omega\in\dD\backslash\{0\}}\ker  \sfS^\Theta_{\omega}(\l)=\{0\}$.
In fact, for every pair $\omega_1,\omega_2 \in\dD\backslash\{0\}$ one gets
$\ker  \sfS^\Theta_{\omega_1}(\l)\cap\ker  \sfS^\Theta_{\omega_2}(\l)=\{0\}$.
\end{example}

\begin{remark}
Let $A$ be a closed symmetric operator in a Pontryagin space $\sH$ with equal defect numbers, and let $\pm i\notin\s_p(A)$.
Then its Cayley transform $V=(A-iI)(A+iI)^{-1}$ is an isometric operator in $\sH$.

Let $(\sL_1\oplus\sL_2,\G)$ be a boundary pair for $V$ with $V_*=\dom \G$ such that $\sL_1=\sL_2=:\cH$.
Define the Kre\u{\i}n spaces $(\sH^2,[\![\cdot,\cdot]\!]_{\sH^2})$ and $(\cH^2,[\![\cdot,\cdot]\!]_{\cH^2})$ with the inner products
\[
    [\![\wh f,\wh f]\!]_{\sH^2}= -i\left([f',f]_\sH-[f,f']_\sH\right),\quad \wh f=
    \begin{pmatrix}
      f\\f'
    \end{pmatrix}\in \sH^2,
\]
\[
    [\![\wh u,\wh u]\!]_{\cH^2}= -i\left((u',u)_\cH-(u,u')_\cH\right),\quad \wh u=
    \begin{pmatrix}
      u\\u'
    \end{pmatrix}\in \cH^2.
\]
Then the Cayley transform determines the unitary operator from the Kre\u{\i}n space $(\sH^2,[\![\cdot,\cdot]\!]_{\sH^2})$ to the
Kre\u{\i}n space $(\sH^2,J_{\sH^2})$ defined in Section \ref{subsec3.1},
\[
    C=\frac{1}{\sqrt{2}}\begin{pmatrix}
                          i&1 \\
                          -i&1
                        \end{pmatrix}:(\sH^2,[\![\cdot,\cdot]\!]_{\sH^2})\to (\sH^2,J_{\sH^2})
\]
and similarly with $\sL=\sL_1\times\sL_2$ it determines a unitary mapping from the Kre\u{\i}n space $(\cH^2,[\![\cdot,\cdot]\!]_{\cH^2})$ to the
Kre\u{\i}n space $(\sL,J_{\sL})$,
\[
    C=\frac{1}{\sqrt{2}}\begin{pmatrix}
                          i&1 \\
                          -i&1
                        \end{pmatrix}:(\cH^2,[\![\cdot,\cdot]\!]_{\cH^2})\to (\sL,J_{\sL}).
\]
It follows that the linear relation
\[
\wt \G=C^{-1}\circ\G\circ C
=\left\{\left\{\wh f,
    \wh u\right\}:=\left\{
    \begin{pmatrix}
      ig'-ig\\g'+g
    \end{pmatrix},
        \begin{pmatrix}
      iv'-iv\\v'+v
    \end{pmatrix}\right\}: \left\{
    \begin{pmatrix}
      g\\g'
    \end{pmatrix},
    \begin{pmatrix}
      v\\v'
    \end{pmatrix}\right\}\in\Gamma\right\}
 \]
is unitary from $(\sH^2,[\![\cdot,\cdot]\!]_{\sH^2})$ to $(\cH^2,[\![\cdot,\cdot]\!]_{\cH^2})$ with the kernel $\ker \wt \G=A$ and the domain
\[
\dom \wt \G=C^{-1}V_*=\left\{\wh f=
    \begin{pmatrix}
      ig'-ig\\g'+g
    \end{pmatrix}:
    \begin{pmatrix}
      g\\g'
    \end{pmatrix}\in V_*\right\}.
 \]
Since the mapping $\wt\Gamma$ is isometric from $(\sH^2,[\![\cdot,\cdot]\!]_{\sH^2})$ to $(\cH^2,[\![\cdot,\cdot]\!]_{\cH^2})$ the following (Green's) identity
\[
[f',f]_\sH-[f,f']_\sH=[u',u]_\cH-[u,u']_\cH
\]
holds for all $\{\wh f,\wh u\}\in\wt\Gamma$ and
due to \cite{DHMS06,DHM17} the unitarity of $\wt\Gamma$ means that $(\cH^2,\wt \G)$ is a unitary boundary pair for the symmetric operator $A$.
This boundary pair becomes ordinary when the mapping $\wt\Gamma$ is surjective or, equivalently, when $(\cH^2, \G)$ is an ordinary  boundary pair for $V$.
\end{remark}

Finally the main results in this subsection are specialized to unitary boundary pairs of Hilbert space isometries.
In a Hilbert space setting the properties of the $\gamma$-fields can be connected more directly to the properties
of the Weyl function.

\begin{proposition}\label{SpecialHilbert}
Let $({\sL},\G)$ be a unitary boundary pair for an isometric operator $V$
in the Pontryagin space ${\sH}$. Then
\[
 \mul \Gamma_1=\ker \g_1(\l), \; \l\in\dD, \quad\text{and}\quad \mul \Gamma_2=\ker
\g_2(\l), \;  \l\in\dD_e,
\]
and $({\sL},\G)$ admits the following further properties.
\begin{itemize}
     \item[(a)] The following statements are equivalent:
\begin{enumerate}
\def\labelenumi{\rm (\roman{enumi})}

\item $\Gamma$ is single valued (i.e. $\mul\Gamma=\{0\}$)

\item $\ran \Gamma$ is dense in $\sL_1\times\sL_2$;

\item  $\ker \g_1(\l)=\{0\}$ for some (equivalently for all) $\l\in\dD$

\item $\ker \g_2(\l)=\{0\}$ for some (equivalently for all) $\l\in\dD_e$;
\item $\ker\sfS^\Theta_\l(\l)=\{0\}$
for some (equivalently for all) $\l\in\dD\cup \dD_e$.
\end{enumerate}

    \item[(b)] Moreover, the following statements are equivalent:
\begin{enumerate}
\def\labelenumi{\rm (\roman{enumi})}
\item $({\sL},\G)$ reduces to an ordinary boundary triple $({\sL},\G_1,\G_2)$ for $V$;

\item $\dom \Gamma=V^{-[*]}$ and $\mul\G=\{0\}$;

\item $\ran \Gamma=\sL_1\times\sL_2$;

\item $\ker \g_1(\l)=\{0\}$ and $\ran \g_1(\l)$ is closed for some (equivalently for all) $\l\in\dD$;

\item $\ker \g_2(\l)=\{0\}$ and $\ran \g_2(\l)$ is closed for some (equivalently for all) $\l\in\dD_e$;

\item $0\in \rho(\sfS^\Theta_\l(\l))$, i.e., $\|\Theta(\lambda)\|<1$ for some (equivalently for
all) $\l\in\dD\cup \dD_e$.
\end{enumerate}
\end{itemize}
\end{proposition}

\begin{proof}
(a) The equivalences (i)--(iv) follow from Proposition \ref{P:SpecialBpair}; see also Theorem \ref{thm:Hilbert}.
To see the equivalence with item (v) apply Proposition \ref{P:1.2} with $\lambda=\mu$:
\begin{equation}\label{M01}
\sfS^\Theta_\l(\l)=\frac{I-\Theta(\l)^* \Theta(\l)}{1-|\l|^2}=\g_1(\l)^*\g_1(\l)\geq 0, \quad\l\in\dD;
\end{equation}
\begin{equation}\label{M02}
\sfS^\Theta_\l(\l)=\frac{I-\Theta^\#(\l)^* \Theta^\#(\l)}{1-|\l|^2}=-\g_2(\l)^*\g_2(\l)\leq 0, \quad\l\in\dD_e.
\end{equation}
In the present Hilbert space case these identities lead to
\[
 \ker \sfS^\Theta_\l(\l)=\ker \g_1(\l), \; \l\in\dD; \quad
 \ker \sfS^\Theta_\l(\l)=\ker \g_2(\l),\; \l\in\dD_e.
\]
This implies the equivalence of (iii), (iv) and (v) in part (a).

(b) Here the equivalence of (i) and (ii) holds just by the definition of an ordinary boundary triple (see \cite{MM03}).
The equivalences (ii)--(v) are obtained from Proposition \ref{P:OrdinaryBpair} (cf. also Theorem \ref{thm:Hilbert}).
To see the equivalence with item (vi) apply the identities \eqref{M01}, \eqref{M02}:
\[
\begin{array}{ll}
 0\in\rho(I-\Theta(\l)^* \Theta(\l))        &\Longleftrightarrow\quad 0\in\rho(\g_1(\l)^*\g_1(\l)), \quad \l\in\dD; \\
 0\in\rho(I-\Theta^\#(\l)^* \Theta^\#(\l))  &\Longleftrightarrow\quad 0\in\rho(\g_2(\l)^*\g_2(\l)),\quad \l\in\dD_e.
\end{array}
\]
Thus $I-\Theta(\l)^* \Theta(\l)$ and $I-\Theta^\#(\l)^* \Theta^\#(\l)$ are uniformly positive or, equivalently,
$\|\Theta(\lambda)\|<1$ and $\|\Theta^\#(\l))\|<1$. This completes the proof.
\end{proof}

Notice that part (b) of Proposition \ref{SpecialHilbert} contains the properties and generality that can be attained
when applying (ordinary) boundary triples for isometric operators which have been introduced and studied in \cite{MM03,MM04}.
\begin{remark}
  An analog of boundary triple  in scattering form~\eqref{E:1.1} is encountered in~\cite{DD19},
  where extension theory of multiplication operators in indefinite de Branges spaces was developed.
  The role of the Weyl function in that work is played by de~Branges matrix.
\end{remark}

\section{Extension theory and generalized coresolvents}\label{sec4}

An extension $\wt V$ of the isometric operator $V$ is called \textit{proper}, if $V\subsetneq \wt V\subsetneq V^{-[*]}$.
The set  of proper extensions of $V$ % completed by $V$ and $V^{-[*]}$
was parametrized in~\cite{B13} via an ordinary boundary triple.
In the present section we consider extensions $\wt V$ of the isometric operator $V$, which are proper with respect to a given unitary boundary pair $(\sL,\Gamma)$, i.e.
\begin{equation}
  V\subsetneq \wt V\subsetneq  V_*=\dom\Gamma.
\end{equation}
For such extensions we prove sufficient conditions for regularity of a point $\lambda$, find formulas for their coresolvents and then apply them for a description of generalized coresolvents of the isometric operator $V$.

\subsection{A preparatory lemma}

\begin{lemma}\label{L:Ggf}
Let $V:{\sH}\to{\sH}$ be an isometric operator and let $\Pi=({\sL},\Gamma)$ be a unitary boundary pair
for $V$. Then:
\begin{enumerate}
\def\labelenumi{\rm (\roman{enumi})}
\item For every $\begin{pmatrix}
  f_1\\f_1'
\end{pmatrix}\in V_1$ and $\l\in\cD$, as defined in \eqref{eq:gam_sharp} one has
\begin{equation}\label{gamma1}
\sk\{{\begin{pmatrix} f_1\\f_1'
\end{pmatrix},\frac{1}{\l}\g_2^\#(\l)(f_1'-\l f_1)}\}\in \G_2.
\end{equation}
If $0\in\cD$ then the formula \eqref{gamma1} for $\lambda=0$ takes the form
\begin{equation}\label{gamma10}
\sk\{{\begin{pmatrix} f_1\\f_1'
\end{pmatrix},(\gamma_2^\#)'(0)f_1}\}\in \G_2.
\end{equation}
\item
For every $\begin{pmatrix}
  f_2\\f_2'
\end{pmatrix}\in V_2$ and $\l\in\cD_e$ one has
\begin{equation}\label{gamma2}
\sk\{{\begin{pmatrix} f_2\\f_2'
\end{pmatrix},-\frac{1}{\l}\g_1^\#(\l)(f_1'-\l f_1)}\}\in \G_1.
\end{equation}
\end{enumerate}
\end{lemma}

\begin{proof}
(i) Using Definitions~\ref{def:gamma},~\ref{char} it is seen that
\begin{equation}\label{eq:4.1A}
  \sk\{{
  \begin{pmatrix}
    \g_2\sk({\frac{1}{\ov \l}})u_2 \\
    \frac{1}{\ov \l}\g_2\sk({\frac{1}{\ov \l}})u_2
  \end{pmatrix},
  \begin{pmatrix}
    \Theta\sk({\frac{1}{\ov \l}})^*u_2 \\
    u_2
  \end{pmatrix}
  }\},\quad%\mbox{and}\quad
  \sk\{{
  \begin{pmatrix}
    f_1 \\
    f_1'
  \end{pmatrix},
  \begin{pmatrix}
   0 \\
    v_2
  \end{pmatrix}
  }\}\in \G,
\end{equation}
for all $u_2\in \sL_2$, and some $v_2\in \sL_2$, $\l\in \cD\setminus\{0\}$.
By applying the identity \eqref{E:1.1} to these elements one obtains
$$
\sk[{\g_2\sk({\frac{1}{\ov \l}})u_2,f_1}]_\sH-\frac{1}{\ov \l}\sk[{\g_2\sk({\frac{1}{\ov \l}})u_2,f_1'}]_\sH=-(u_2, v_2)_{\sL_2},
$$
or, equivalently,
$$
 \sk({u_2,\g_2\sk({\frac{1}{\ov \l}})^*\left(\frac{1}{\l}f_1'-f_1\right)})_{\sL_2}=(u_2, v_2)_{\sL_2}.
$$
Since $u_2\in\sL_2$ is arbitrary this implies the equality
$$
v_2=\frac{1}{\l}\g_2^\#\sk({\l})\sk({f_1'-\l f_1})
$$
which in combination with \eqref{eq:4.1A} yields \eqref{gamma1}.

The equality~\eqref{gamma10} is implied by \eqref{gamma1} and \eqref{eq:g_inf2}.

(ii) Similarly, applying \eqref{E:1.1} to the vectors
\begin{equation}\label{eq:4.2A}
  \sk\{{
  \begin{pmatrix}
    \g_1\sk({\frac{1}{\ov \l}})u_1 \\
    \frac{1}{\ov \l}\g_1\sk({\frac{1}{\ov \l}})u_1
  \end{pmatrix},
  \begin{pmatrix}
    u_1 \\
    \Theta\sk({\frac{1}{\ov \l}})u_1
  \end{pmatrix}
  }\},\,
  \sk\{{
  \begin{pmatrix}
    f_2 \\
    f_2'
  \end{pmatrix},
  \begin{pmatrix}
   v_1\\
    0
  \end{pmatrix}
  }\}\in \G,
\end{equation}
where $u_1, v_1 \in \sL_1$, $f_2,f_2'\in\sH$, $\l\in \cD_e$, one obtains
$$
\sk[{\g_1\sk({\frac{1}{\ov \l}})u_1,f_2}]_\sH-\frac{1}{\ov \l}\sk[{\g_1\sk({\frac{1}{\ov \l}})u_1,f_2'}]_\sH=(u_1, v_1)_{\sL_1}
$$
and
$$
\sk({u_1,\g_1^\#\sk({\l})\sk({f_2-\frac{1}{\l}f_2'})})_{\sL_1}=(u_1, v_1)_{\sL_1}.
$$
This implies
$$
v_1=-\frac{1}{\l}\g_1^\#\sk({\l})\sk({f_2'-\l f_2}),
$$
which together with \eqref{eq:4.2A} yields \eqref{gamma2}.
\end{proof}

\subsection{Weyl function and spectrum of proper extensions of $V$}
A unitary boundary pair $(\sL,\Gamma)$ is a tool which allows to determine those extensions $\wt V$ of $V$ that satisfy $V\subsetneq  \wt V\subsetneq  V_*$ in the following way.
Let $\Phi$ be a  linear relation from $\sL_1$ to $\sL_2$ represented in the form
\begin{equation}\label{eq:Phi}
  \Phi=\left\{\binom{\Phi_1h}{\Phi_2h}:\,h\in\cH\right\},
\end{equation}
where $\cH$ is an auxiliary Hilbert space and $\Phi_j$ are bounded linear operators $\Phi_j:\cH\to\sL_j$ $(j=1,2)$, such that
\begin{equation}\label{eq:Invert0}
\ker(\Phi_1^*\Phi_1+\Phi_2^*\Phi_2)=\{0\}.
\end{equation}
Then $\Phi$ is closed if and only if
\begin{equation}\label{eq:Invert}
 0\in\rho(\Phi_1^*\Phi_1+\Phi_2^*\Phi_2).
\end{equation}
Associate with $\Phi$ an extension $V_\Phi$ of $V$ by
\begin{equation}\label{eq:VPhi}
  V_\Phi= \left\{
  \begin{pmatrix}
    f \\
    f'
  \end{pmatrix}
  \in V_*:\left\{
  \begin{pmatrix}
    f \\
    f'
  \end{pmatrix},
  \begin{pmatrix}
    \Phi_1h \\
    \Phi_2 h
  \end{pmatrix}\right\}
  \in \G \text{ for some } h\in \cH
  \right\}.
\end{equation}
The following theorem gives a description of the spectrum of $V_\Phi$ and contains
a Kre\u{\i}n type resolvent formula.

\begin{theorem}\label{res1}
Let $V$ be a closed isometric operator in $\sH$,
let $\Pi=({\sL},\Gamma)$ be a
unitary boundary pair for $V$, let $\Phi\in \mathbf{B}(\cH,\sL)$ and let~\eqref{eq:Invert0} hold. If $\l\in\cD$ then:
\begin{enumerate}
\def\labelenumi{\rm (\roman{enumi})}
\item
$\l\in\s_p(V_\Phi) \Longrightarrow \ker(\Phi_2-\Theta(\l)\Phi_1)\ne\{0\}$;
\item
$\Phi_2-\Theta(\l)\Phi_1:\cH\to\sL_2$ has a bounded inverse
$\Longrightarrow \l\in\r(V_\Phi)$. \end{enumerate}
When (ii) is satisfied the resolvent of $V_\Phi$ takes the form
\begin{equation}\label{E:res0}
  (V_\Phi-\l I_{\sH})^{-1}=(V_1 - \l I_{\sH})^{-1}+\frac{1}{\l}\g_1(\l)\Phi_1(\Phi_2-\Theta(\l)\Phi_1)^{-1}\g_2^\#(\l)\quad (\l\in\cD).
\end{equation}
If $\l\in\cD_e$ then:
\begin{enumerate}
\item[(iii)]
$\l\in\s_p(V_\Phi) \Longrightarrow \ker(\Phi_1-\Theta^\#(\l)\Phi_2)\ne\{0\}$;
\item[(iv)]
$\Phi_1-\Theta^\#(\l)\Phi_2:\cH\to\sL_1$ has a bounded inverse
$ \Longrightarrow \l\in\r(V_\Phi)$; \end{enumerate}
When (iv) is satisfied the resolvent of $V_\Phi$ takes the form
\begin{equation}\label{E:res1}
  (V_\Phi-\l I_{\sH})^{-1}=(V_2-\l I_{\sH})^{-1}-\frac{1}{\l}\g_2(\l)\Phi_2(\Phi_1-\Theta^\#(\l)\Phi_2)^{-1}\g_1^\#(\l)\quad (\l\in\cD_e).
\end{equation}
If, in addition,  $\Pi=({\sL},\Gamma)$ is an ordinary
boundary triple for $V$ then the implications (i)--(iv) become equivalences.
\end{theorem}
\begin{proof}
The proof is divided into steps.
\bigskip

{\bf 1.} {\it Verification of  (i).}
If $\l\in\cD\cap\s_p(V_\Phi)$, $\cD=\rho(V_1)$, then there is $f\in\sH\setminus\{0\}$ and $ h\in \cH$ such that
\begin{equation}\label{eq:4.3B}
  \left\{\begin{pmatrix}
    f \\
   \l f
  \end{pmatrix},
  \begin{pmatrix}
    \Phi_1h \\
    \Phi_2h
  \end{pmatrix}\right\}
  \in \Gamma.
\end{equation}
Since  $\lambda\in\cD$, one has $\lambda\not\in\sigma_p(V)$ and hence $h\ne 0$. By Definition~\ref{char} this means that
  \begin{equation}\label{eq:4.3D}
    \Theta(\l)\Phi_1h=\Phi_2h, \quad \lambda\in\cD.
  \end{equation}
Hence, $h\in \ker(\Phi_2-\Theta(\l)\Phi_1)$.

{\bf 2.} {\it Verification of (ii).}
First assume that $(\Phi_2-\Theta(\l)\Phi_1)$ has a bounded inverse with $\l\in\cD\setminus\{0\}$.
Let us find a solution
$\wh f=
\begin{pmatrix}
  f \\
  f'
\end{pmatrix}
\in V_\Phi$ of the equation
\begin{equation}\label{eq:4.4}
  f'-\l f=g
\end{equation}
for arbitrary $g\in \sH$. Since $\l\in\cD=\r(V_1)$, there are $f_1,f_1'\in\sH$, such that
\[%begin{equation}\label{eq:4.4A}
  f_1'-\l f_1=g\quad\textup{and}\quad \begin{pmatrix}
  f_1 \\
  f'_1
\end{pmatrix}\in V_1.
\]%end{equation}
Hence
\begin{equation}\label{eq:4.5B}
f_1=(V_1-\l I_{\sH})^{-1}g.
\end{equation}
By Lemma \ref{L:Ggf}
\begin{equation}\label{eq:4.6}
  \left\{
  \begin{pmatrix}
    f_1 \\
    f_1'
  \end{pmatrix},
   \begin{pmatrix}
   0\\
   u_2
   \end{pmatrix}\right\}\in\G,\;\textup{where}\quad
   u_2= \frac{1}{\l}\g^\#_2(\l)g=\frac{1}{\l}\g^\#_2(\l)(f_1'-\l f_1).
\end{equation}
Now choose $h=(\Phi_2-\Theta(\l)\Phi_1)^{-1}u_2$ and apply \eqref{eq:WF2} to get
\begin{equation}\label{eq:N_lambda}
  \left\{
  \begin{pmatrix}
 \g_2(\l)\Phi_1h \\
  \l\g_2(\l)\Phi_1h
\end{pmatrix},
   \begin{pmatrix}
 \Phi_1h \\
  \Theta(\lambda)\Phi_1h
\end{pmatrix}
\right\}\in\Gamma.
\end{equation}
Combining \eqref{eq:4.6} and \eqref{eq:N_lambda} one obtains
\begin{equation}\label{eq:4.6B}
  \left\{
  \begin{pmatrix}
    f_1 \\
    f_1'
  \end{pmatrix}
 +\begin{pmatrix}
 \g_1(\l)\Phi_1h \\
 \l \g_1(\l)\Phi_1h
\end{pmatrix} ,
    \begin{pmatrix}
    \Phi_1h \\
     u_2 +\Theta(\lambda)\Phi_1h
  \end{pmatrix}\right\}
  \in\G.
\end{equation}
Setting
\begin{equation}\label{eq:4.5}
\begin{pmatrix}
  f \\
  f'
\end{pmatrix}=
\begin{pmatrix}
  f_1 \\
  f'_1
\end{pmatrix}+
\begin{pmatrix}
 \g_1(\l))\Phi_1h \\
  \l\g_1(\l))\Phi_1h
\end{pmatrix}
\end{equation}
and using the equality
\[
    u_2+\Theta(\lambda)\Phi_1h=(I+\Theta(\lambda)\Phi_1(\Phi_2-\Theta(\l)\Phi_1)^{-1})u_2=\Phi_2h
\]
one obtains from~\eqref{eq:4.6B}
\begin{equation}\label{eq:4.3E}
  \left\{\begin{pmatrix}
   f \\
   f'
  \end{pmatrix},
  \begin{pmatrix}
    \Phi_1h \\
    \Phi_2h
 \end{pmatrix}\right\}
  \in \Gamma.
\end{equation}
Therefore, the equation \eqref{eq:4.4} has a solution
$\wh f=
\begin{pmatrix}
  f \\
  f'
\end{pmatrix}
\in V_\Phi$
and
\[%begin{equation}\label{eq:4.9}
\begin{split}
   f & = f_1+ \g_1(\l)\Phi_1h \\
     & = (V_1-\l I_{\sH})^{-1}g + \frac{1}{\l}\g_1(\l)\Phi_1(\Phi_2-\Theta(\l)\Phi_1)^{-1} \g_2^\#(\l)g.
\end{split}
\]%end{equation}

Next assume that $0\in\cD$ and $0\in\r(\Phi_2-\Theta(0)\Phi_1)$.
Then by Lemma \ref{L:Ggf}
the equality~\eqref{eq:4.6} holds with $u_2=(\gamma_2^\#)'(0)f_1$ and now combining this analog of \eqref{eq:4.6} with \eqref{eq:N_lambda} yields~\eqref{eq:4.3E}.
The formula \eqref{E:res0} for $\lambda=0$ takes the form
%is trivial due to the equality $\gamma_1^\#(0)=0$.
\[
V_\Phi^{-1}= V_1^{-1} + \g_1(0)\Phi_1(\Phi_2-\Theta(0)\Phi_1)^{-1} (\g_2^\#)'(0)
\]
This completes the proof of the implication in (ii) and the formula \eqref{E:res0}.

{\bf 3.} {\it Verification of  (iii).} If $\l\in\s_p(V_\Phi)\cap\cD_e$ then there is $f\in\sH\backslash\{0\}$ and $h\in\cH$ such that~\eqref{eq:4.3B} holds.
Again, since $\lambda\in\cD$, one has $\lambda\not\in\sigma_p(V)$ and hence $h\ne 0$. By Definition~\ref{char} one gets
\[%begin{equation}\label{eq:4.12}
    \Phi_1h-\Theta^\#(\l)\Phi_2h=0.
\]%end{equation}
Hence  $h\in\ker(\Phi_1-\Theta^\#(\l)\Phi_2)$.

{\bf 4.} {\it Verification of  (iv).} Assume that  $\Phi_1-\Theta^\#(\l)\Phi_2$ has a bounded inverse and  $\l\in\cD_e$.
Since $\l\in\cD_e=\r(V_2)$, there are $f_2,f_2'\in\sH$, such that
\[%begin{equation}\label{eq:4.4A}
  f_2'-\l f_2=g\quad\textup{and}\quad \begin{pmatrix}
  f_2 \\
  f'_2
\end{pmatrix}\in V_2.
\]%end{equation}
Hence
$%\begin{equation}\label{eq:4.5B}
f_2=(V_2-\l I_{\sH})^{-1}g$.
%\end{equation}
By Lemma \ref{L:Ggf}
\begin{equation}\label{eq:4.6D}
  \left\{
  \begin{pmatrix}
    f_2 \\
    f_2'
  \end{pmatrix},
   \begin{pmatrix}
   u_1\\
   0
   \end{pmatrix}\right\}\in\G\quad\textup{with}\quad
   u_1= -\frac{1}{\l}\g^\#_1(\l)g.
\end{equation}
Now choose $h=(\Phi_1-\Theta^\#(\l)\Phi_2)^{-1}u_1$ and apply \eqref{eq:WF1} to get
\begin{equation}\label{eq:N_lambdaB}
  \left\{
  \begin{pmatrix}
 \g_2(\l)\Phi_2h \\
 \l \g_2(\l)\Phi_2h
\end{pmatrix},
   \begin{pmatrix}
   \Theta^\#(\lambda)\Phi_2h \\
   \Phi_2h
\end{pmatrix}
\right\}\in\Gamma.
\end{equation}
Combining \eqref{eq:4.6D} and \eqref{eq:N_lambdaB} one obtains
\begin{equation}\label{eq:4.6C}
  \left\{
  \begin{pmatrix}
    f \\
    f'
  \end{pmatrix},
    \begin{pmatrix}
    \Phi_1h \\
    \Phi_2h
  \end{pmatrix}\right\}
  \in\G, \;\textup{where}\quad
  \begin{pmatrix}
    f \\
    f'
  \end{pmatrix}=\begin{pmatrix}
  f_2 \\
  f'_2
\end{pmatrix}+
\begin{pmatrix}
 \g_2(\l)\Phi_2h \\
 \l \g_2(\l)\Phi_2h
\end{pmatrix}.
\end{equation}
Making use of \eqref{eq:4.6C}, \eqref{eq:4.6D}, and the above formulas for $h$ and $f_2$ one obtains
\[%begin{equation}\label{eq:4.20}
\begin{split}
   f & = f_2+ \g_2(\l)\Phi_2h \\
     & = (V_2-\l I_{\sH})^{-1}g - \frac{1}{\l}\g_2(\l)\Phi_2(\Phi_1- \Theta^\#(\l)\Phi_2)^{-1} \g_1^\#(\l)g
\end{split}
\]%end{equation}
This proves \eqref{E:res1} and the implication in (iv).

{\bf 5.} {\it Verification of the reverse implication in (i) for the case of an ordinary
boundary triple $({\sL},\Gamma)$.}
Let $(\Phi_2-\Theta(\l)\Phi_1)h=0$  for some $h\in\cH\backslash\{0\}$. Then it follows from
\begin{equation}\label{eq:4.3C}
  \sk\{{
  \begin{pmatrix}
     \g_1(\l)\Phi_1h \\
    \l\g_1(\l)\Phi_1h
  \end{pmatrix},
  \begin{pmatrix}
    \Phi_1h \\
    \Theta(\l)\Phi_1h
  \end{pmatrix}
  }\}  \in \Gamma
\end{equation}
and \eqref{eq:4.3D} that  \eqref{eq:4.3B} holds with $f=\g_1(\l)\Phi_1h$.
Notice that $\Phi_1h\ne 0$ since otherwise  $\Phi_2h= 0$ by \eqref{eq:4.3D}, which contradicts to~\eqref{eq:Invert0}.
Therefore, $f\ne 0$ since $\ker \g_1(\l)=\{0\}$ for the ordinary
boundary triple $\Pi=({\sL},\Gamma)$; see Proposition \ref{P:OrdinaryBpair}. Thus $\l\in\s_p(V_\Phi)$.

{\bf 6.} {\it Verification of the reverse implication in (ii) for the case of an ordinary
boundary triple $({\sL},\Gamma)$.}
Let $\l\in\r(V_\Phi)$. By virtue of item {\bf 5} to prove the boundedness of the inverse $(\Phi_2-\Theta(\l)\Phi_1)^{-1}:\sL_2\to \cH$ it is enough to show that
\begin{equation}\label{ran}
  \ran(\Phi_2-\Theta(\l)\Phi_1)=\sL_2.
\end{equation}
By assumption $\l\in\r(V_\Phi)\cup\cD$ and hence for arbitrary $g\in\sH$ one can find vectors
$\begin{pmatrix}
  f\\
  f'
\end{pmatrix}\in V_\Phi$ and
$\begin{pmatrix}
  f_1\\
  f_1'
\end{pmatrix}\in V_1$
such that
\[%begin{equation}\label{eq:4.52}
  f_1'-\l f_1=f-\l f'=g\quad (\l\in\r(V_\Phi)\cup\cD).
\]%end{equation}
Then \eqref{eq:4.6}--\eqref{eq:4.6B} hold for some $h\in\cH$ and, in particular,
$$
(\Phi_2-\Theta(\l)\Phi_1)h=u_2=\frac{1}{\l}\g_2^\#(\l)g.
$$
Since $g\in\sH$ is arbitrary and for an ordinary boundary triple $\ran \g_2^\#(\l)=\sL_2$, the claim \eqref{ran} is proved.
By Open Mapping Theorem the operator $T:=\Phi_2-\Theta(\l)\Phi_1:\cH\to \sL_2$ has a bounded inverse, since $T\in{\mathbf B}(\cH,\sL_2)$, $\ker T=\{0\}$ and $\ran T=\sL_2$.
\end{proof}

The operator function $ (I_{\sH}-z V_\Phi)^{-1}$ is called the \textit{coresolvent} of $V_\Phi$.
Setting $\l = 1/z$ in Theorem~\ref{res1} one obtains  the following statement for coresolvents of $V_\Phi$.
\begin{corollary}\label{cores1}
Let $V$ be a  closed isometric operator in $\sH$,
let $\Pi=({\sL},\Gamma)$ be a
unitary boundary pair for $V$, let $\Phi\in \mathbf{B}(\cH,\sL)$ and let~\eqref{eq:Invert0} hold. If $z\in\overline{\cD}$ then:
\begin{enumerate}
\def\labelenumi{\rm (\roman{enumi})}
\item
$z\in\s_p(V_\Phi^{-1}) \Longrightarrow \ker\left(\Phi_1-\Theta\left(\bar{z}\right)^*\Phi_2\right)\ne\{0\}$;
\item[(ii)]
$\Phi_1-\Theta(\bar{z})^*\Phi_2:\cH\to\sL_1$ has a bounded inverse
$ \Longrightarrow z\in\r(V_\Phi^{-1})$; \end{enumerate}
When (ii) is satisfied then the coresolvent of $V_\Phi$ takes the form
\begin{equation}\label{E:cores1}
  (I_{\sH}-z V_\Phi)^{-1}=(I_{\sH}-z V_2)^{-1}+\g_2\left(\frac{1}{z}\right)\Phi_2\left(\Phi_1-\Theta(\bar{z})^*\Phi_2\right)^{-1}\g_1(\bar{z})^*. %\quad % (z\in\cD).
\end{equation}
If $z\in\overline{\cD_e}$ then:
\begin{enumerate}
\item[(iii)]
$z\in\s_p(V_\Phi^{-1}) \Longrightarrow \ker\left(\Phi_2-\Theta\left({1}/{z}\right)\Phi_1\right)\ne\{0\}$;
\item
$\Phi_2-\Theta\left({1}/{z}\right)\Phi_1:\cH\to\sL_2$ has a bounded inverse
$\Longrightarrow z\in\r(V_\Phi^{-1})$. \end{enumerate}
When (iv) is satisfied then the coresolvent of $V_\Phi$ takes the form
\begin{equation}\label{E:cores2}
  (I_{\sH}-z V_\Phi)^{-1}=(I_{\sH} - z V_1)^{-1}-\g_1\left({1}/{z}\right)\Phi_1\left(\Phi_2-\Theta\left({1}/{z}\right)\Phi_1\right)^{-1}\g_2(\bar{z})^*.
  %\quad (\l\in\cD_e).
\end{equation}
If, in addition,  $\Pi=({\sL},\Gamma)$ is an ordinary
boundary triple for $V$ then the implications (i)--(iv) become equivalences.
\end{corollary}

\begin{remark}\label{rem:4.3}
If  $(\sL_1\oplus\sL_2,\G_1,\G_2)$ is an ordinary boundary triple for $V$
then every closed proper extension of $V$ can be represented in the form~\eqref{eq:VPhi} with $\Phi_j\in\mathbf{B}(\cH,\sL_j)$ $(j=1,2)$ such that~\eqref{eq:Invert} holds due to \cite[Theorem 2.1]{B13}, \cite[Proposition~6.12]{DM17}.
Moreover, if  $\Phi$ is defined by \eqref{eq:Phi}, then the following equivalences hold:
\begin{enumerate}
  \item $V_\Phi$ is an isometric relation in $\cH$ $\Longleftrightarrow$ $\Phi$ is a graph of an isometric operator;
  \item $V_\Phi$ is a unitary relation in $\cH$ $\Longleftrightarrow$  $\Phi$ is  a graph of a unitary operator;
  \item $V_\Phi$ is a contractive relation in $\cH$ $\Longleftrightarrow$  $\Phi$ is  a graph of a contraction.
\end{enumerate}
The fact that the implications (i)--(iv) of Theorem~\ref{res1} become  equivalences for  an ordinary
boundary triple $\Pi=({\sL},\Gamma)=({\sL},\Gamma_1,\Gamma_2)$  was proved in~\cite{MM03} in the case when $\k=0$,
and in~\cite{B13}  in the case $\k\ne 0$.
\end{remark}

\subsection{Description of generalized coresolvents.}

\begin{definition}[see \cite{L71,KL72}]
An operator-valued function $\mathbf{K}_\l$ holomorphic
in a domain $\cO\subseteq  \dD$  with values in $\mathbf{B}(\sH)$ is called the
\textit{generalized coresolvent} of an isometric operator
$V:\sH\rightarrow\sH$, if there exist a Pontryagin
space $\wt{{\sH}}\supset{\sH}$ with negative index $\wt\kappa=\kappa_-(\wt\sH)$ and a unitary extension
$\wt V:\wt{\sH}\rightarrow \wt{\sH}$ of the operator
$V$ such that $\cO\subseteq \r(\wt{V}^{-1})$, and
\begin{equation}\label{minA}
\mathbf{K}_{z}=P_{{\sH}}\sk({I_{\wt\sH}-{z}\wt{V}})^{-1}\upharpoonright{\sH},
\quad{z}\in\cO,
\end{equation}
where $P_{{\sH}}$ is the orthogonal projection from $\wt{{\sH}}$ onto ${\sH}$.
Notice that in~\cite{KL72} the operator function $\mathbf{K}_{z}$ in~\eqref{minA} is called generalized resolvent of $V$.

\noindent
The representation~\eqref{minA} of the generalized coresolvent of %the isometric operator
$V$ is called \textit{minimal}, if
\[%begin{equation}\label{eq:MinGres}
    \wt{{\sH}}=\overline{\mbox{span}}\left\{\sH+(I_{\wt\sH}-{z}\wt{V})^{-1}\sH:\,\lambda\in\cO\right\}.
\]%end{equation}
The generalized coresolvent $\mathbf{K}_{z}$ is said to be $k$-\textit{regular}, if $k=\kappa_-(\wt\sH[-]\sH)$ for a minimal representation~\eqref{minA}.
\end{definition}
Every generalized coresolvent $\mathbf{K}_{z}$ of the isometric operator
$V$ admits a minimal representation~\eqref{minA} and every two minimal representations of $\mathbf{K}_{z}$ are unitarily equivalent, see~\cite[Proposition 4.1]{Der99} for the case of a symmetric operator.

\begin{lemma}\label{L:Ker}
  Let $\bK_{z}$ be a $(\wt\k-\k)$-regular generalized coresolvent. Then the kernel
  \begin{equation}\label{eq:may4a}
    {\sfR}_{w}({z}):=\frac{\bK_{z}+\bK_{w}^{[*]}-I}{1-{z} \ov{w}}-\bK_{w}^{[*]}\bK_{z}
  \end{equation}
  has $\wt\k-\k$ negative squares on $\r(\wt V^{-1})\cap \dD$.
\end{lemma}
\begin{proof}
  Let $\{{z}_j\}_{j=1}^n$ be a set of points in $\r(\wt V^{-1})\cap\dD$ and let $g_j \in\sH$, $j=1,2,\dots,n$. Denote
  \begin{equation}\label{eq:may4b}
    \wt f_j:= (I_{\wt\sH}-{z}_j\wt V)^{-1}g_j,\quad
    f_j:= P_\sH \wt  f_j=\bK_{{z}_j}g_j,\quad(j=1,2,\dots,n).
  \end{equation}
  Then it follows form the first equality in \eqref{eq:may4b} that
  $$
  \begin{pmatrix}
    \wt f_j\\
    g_j
  \end{pmatrix}\in I_{\wt\sH}-{z}_j\wt V \quad\Longleftrightarrow\quad
  \begin{pmatrix}
    {z}_j\wt f_j\\
    \wt f_j-g_j
  \end{pmatrix}\in \wt V.
  $$
  Since $\wt V$ is a unitary relation in $\wt \sH$ one obtains
  $
    {{z}_j\ov{z}_k}[\wt f_j,\wt f_k]_{\wt\sH}=[\wt f_j-g_j,\wt f_k-g_k]_{\wt\sH}
  $
  or, equivalently,
  \begin{equation}\label{eq:may4d}
    (1-{z}_j\ov{z}_k)[\wt f_j,\wt f_k]_{\wt\sH}=[\wt f_j,g_k]_{\wt\sH}+[g_j,\wt f_k]_{\wt\sH}-[g_j,g_k]_\sH.
  \end{equation}
  It follows from \eqref{eq:may4a}, \eqref{eq:may4b}, and \eqref{eq:may4d} that
  \begin{equation}\label{eq:may4e}
    \begin{split}
      \sum_{j,k=1}^{n}[\sfR_{{z}_k}({z}_j)g_j&,g_k]_\sH\xi_j\ov\xi_k=
      -\sum_{j,k=1}^n[\bK_{{z}_j}g_j,\bK_{{z}_k}g_k]_\sH\xi_j\ov\xi_k
      \\
      +&\sum_{j,k=1}^n\frac{[\bK_{{z}_j}g_j,g_k]_\sH+[g_j,\bK_{{z}_k}g_k]_\sH
      -[g_j,g_k]_\sH}{1-{z}_j\ov{z}_k}{\xi_j\ov\xi_k}\\
      \overset{\eqref{eq:may4b}}{=}&
      \sum_{j,k=1}^n\left\{\frac{[\wt f_j,g_k]_{\wt\sH}+[g_j,\wt f_k]_{\wt\sH}-[g_j,g_k]_\sH}{1-{z}_j\ov{z}_k}
      -[f_j,f_k]_\sH\right\}\xi_j\ov\xi_k\\
      \overset{\eqref{eq:may4d}}{=}&\sum_{j,k=1}^n\left\{[\wt f_j,\wt f_k]_{\wt\sH}-[f_j,f_k]_{\sH}\right\}\xi_j\ov\xi_k\\
      =&\sum_{j,k=1}^n\left[(I-P_\sH)\wt f_j,(I-P_\sH)\wt f_k\right]_{\wt\sH}\xi_j\ov\xi_k.
    \end{split}
  \end{equation}
  This form has at most $\wt \k-\k$ negative squares, since $\ind(\wt\sH[-]\sH)=\wt \k-\k$. Because the representation \eqref{minA} is $(\wt \k-\k)$-regular the set
  $$
    \left\{(I_{\wt\sH}-P_\sH)(I_{\wt\sH}-{z}\wt V)^{-1}\sH:{z}\in\r(\wt V^{-1})\right\}
  $$
  is dense in $\wt \sH[-]\sH$ and hence it contains a $(\wt \k-\k)$-dimensional negative subspace. Therefore, the form \eqref{eq:may4e} has exactly $\wt \k-\k$ negative squares for an appropriate choice of ${z}_j$, $g_j$ $(j=1,2,\dots,n)$.
  \end{proof}

\begin{remark}
In the case of a standard isometric operator the statement of Lemma~\ref{L:Ker} was proved in \cite{DLS90}.
\end{remark}

\begin{theorem}\label{invprob}
Let $V:{\sH}\to {\sH}$ be an isometric operator, let $({\sL}_1\oplus{\sL}_2,\G_1,\G_2)$ be an ordinary boundary triple for $V$,
and let $\Theta(\cdot)$, $\g_1(\cdot)$, $\g_2(\cdot)$ be the corresponding Weyl function and the $\g$-fields.

Then for $ z\in \overline{\cD}\cap \r(\wt V^{-1})$ the formula
\begin{equation}\label{E:GKR}
   \mathbf{K}_z=(I_{\wt\sH}-z V_2)^{-1}+\g_2(1/z)\e(z)\sk({I_{\sL_1}-\Theta(\bar{z})^*\e(z)})^{-1}\g_1(\bar z)^*,
\end{equation}
establishes a one-to-one correspondence between the set of  $\wt \kappa-\kappa$-regular generalized coresolvents of $V$ and the set of all operator-valued functions $\e(\cdot)\in \cS_{\widetilde{\k}-\k}({\sL}_1,{\sL}_2)$, such that
\begin{equation}\label{E:regularityGKR}
    0\in\rho ({I_{\sL_1}-\Theta(\bar{z})^*\e(z)}).
\end{equation}

For $z\in\overline{\cD}_e\cap\r(\wt V^{-1})$ the formula~\eqref{E:GKR} takes the form
\begin{equation}\label{E:GKRe}
   \mathbf{K}_z=(I_{\wt\sH}-z V_1)^{-1}-\g_1(1/z) \varepsilon^T(1/z)\sk({I_{\sL_2}-\Theta(1/z)\e^T(1/z)})^{-1}\g_2(\bar z)^*.
  \end{equation}
\end{theorem}

\begin{proof}
The proof is divided into steps.
\bigskip

{\bf 1.} {\it Verification  that for every $\e(\cdot)\in \mathcal{S}_{\widetilde{\kappa}-\k}({\sL}_1,{\sL}_2)$ satisfying~\eqref{E:regularityGKR} the formula~\eqref{E:GKR} determines a  $(\widetilde\k-\k)$-regular generalized coresolvent of $V$.}

By Theorem~\ref{thm:Inv} there exists a simple isometric operator $V^-$ in a Pontryagin space $\sH^-$ with negative index $\widetilde\k-\k$ and a unitary boundary pair $(\sL,\G^-)$ such that the corresponding Weyl function $\Theta^-({z})$ coincides with $\e({z})$ for ${z}\in\gh_\varepsilon=\r(V_1^-)$, where %$\sL^-=\sL_1^-\times\sL_2^-$ with $\sL_1^-=\sL_2$, $\sL_2^-=\sL_1$, and
$V_1^-=\ker\Gamma_1^-$.

Next we construct a new unitary boundary pair $(\widetilde\sL,\widetilde\G)$ as the direct sum of the ordinary boundary triple $(\sL, \G^+):=(\sL,\G)$ and the unitary boundary pair $(\sL, \G^-)$ by the formulas
$$
\widetilde\sL=\widetilde\sL_1\oplus\widetilde\sL_2, \text{ where }\widetilde\sL_1=\widetilde\sL_2=\begin{pmatrix}
  \sL_1\\
  \sL_2
\end{pmatrix},
$$
\begin{equation}\label{tildeBR}
\widetilde \G=\sk\{{\sk\{{\begin{pmatrix}
  f_+\\
  f_-\\
  f_+'\\
  f'_-
\end{pmatrix},\begin{pmatrix}
  u_1^+\\
  u_2^-\\
  u_1^-\\
  u_2^+
\end{pmatrix}}\}:\left\{\begin{pmatrix}
  f_\pm\\
  f_\pm'
\end{pmatrix},\begin{pmatrix}
  u_1^\pm\\
  u_2^\pm
\end{pmatrix}\right\}\in \G^\pm}\}.
  \end{equation}
Let $\varepsilon({z})$, $\g_1^-({z})$, $\g_2^-({z})$ be the Weyl function and the $\g$-fields of the unitary boundary pair $(\sL,\G^-)$,
let $\varepsilon^T({z})$, $\g_1^{-,T}({z})$, $\g_2^{-,T}({z})$ be the Weyl function and the $\g$-fields of the transposed boundary pair $(\sL,(\G^-)^T)$ and let $\wt V_1=\ker\wt\Gamma_1$, $\wt V_2=\ker\wt\Gamma_2$.
Then the Weyl function $\widetilde\Theta({z})$ and the $\g$-fields $\widetilde\g_1({z})$, $\widetilde\g_2({z})$ of the unitary boundary pair $(\widetilde\sL,\widetilde\G)$ are given by
\begin{equation}\label{eq:5.2A}
   \widetilde \Theta({z})=\begin{pmatrix}
     0&\varepsilon^T({z})\\
     \Theta({z})&0
   \end{pmatrix}:
   \begin{pmatrix}
     {\sL}_1\\
     {\sL}_2
   \end{pmatrix}\to
   \begin{pmatrix}
     {\sL}_1\\
     {\sL}_2
   \end{pmatrix},\quad{z}\in\widetilde\cD=\r(\widetilde V_1)\cap\mathbb{D},
  \end{equation}

\begin{equation}\label{bgfields2}
   \widetilde \g_2({z})=\begin{pmatrix}
     0&\g_2^+({z})\\
     \g_2^{-,T}({z})&0
   \end{pmatrix}
   :\begin{pmatrix}
     {\sL}_1\\
     {\sL}_2
   \end{pmatrix}\to\begin{pmatrix}
     \sH\\
     {\sH}^-
   \end{pmatrix},\quad{z}\in\widetilde\cD=\r(\widetilde V_1)\cap\mathbb{D},
  \end{equation}
\begin{equation}\label{bgfields1}
   \widetilde \g_1({z})=\begin{pmatrix}
     \g_1^+({z})&0\\
      0&\g_1^{-,T}({z})
   \end{pmatrix}
    :\begin{pmatrix}
     {\sL}_1\\
     {\sL}_2
   \end{pmatrix}\to\begin{pmatrix}
     \sH\\
     {\sH}^-
   \end{pmatrix},\quad{z}\in\widetilde\cD_e:=\r(\widetilde V_2)\cap\mathbb{D}_e.
\end{equation}

Consider the extension $\widetilde V_{\Phi}$ of the operator
  $\widetilde V=\begin{pmatrix}
    V&0\\
    0&V^-
  \end{pmatrix}$
in the space $\widetilde\sH:=\sH^+\oplus\sH^-$, $\sH^+:=\sH$, corresponding to the linear relation $\Phi$ of the form~\eqref{eq:Phi}, where
\[
% \Phi\,(=\Phi_2\Phi_1^{-1}) \equiv {I}_{\sL_1\oplus\sL_2}, \quad
\Phi_1=\Phi_2={I}_{\sL_1\oplus\sL_2}.
\]
 In view of~\eqref{tildeBR} the extension $\widetilde V_{\Phi}$ takes the form
\[%begin{equation}\label{eq:may3a}
\widetilde V_{\Phi}=\left\{\left\{\begin{pmatrix}
  f_+\\
  f_-\\
\end{pmatrix},\begin{pmatrix}
  f_+'\\
  f'_-
\end{pmatrix}\right\}:\left\{\begin{pmatrix}
  f_-\\
  f_-'
\end{pmatrix},\begin{pmatrix}
  \Gamma_2\wh f_+\\
  \Gamma_1\wh f_+
\end{pmatrix}\right\}\in \G^-,\,\wh f_+\in V^{-[*]}\right\}.
\]%end{equation}
Then $I_{\widetilde\sL_1}-\widetilde\Theta(\bar{z})^*=\begin{pmatrix}
      I_{\sL_1}&    -\Theta(\bar{z} )^*\\
      -\e(z)   &   I_{\sL_2}
    \end{pmatrix}$
and the assumption~\eqref{E:regularityGKR} yields
\begin{equation}\label{eq:may03}
  0\in\r(I_{\widetilde\sL_1}-\widetilde\Theta(\bar{z})^*)
\end{equation}
for $\bar z\in\widetilde\cD$.
By~\eqref{eq:may03} and Corollary~\ref{cores1} one obtains $z\in\r(\widetilde V_{\Phi}^{-1})$ and
\begin{equation}\label{eq:5.4A}
    (I_{\widetilde\sH}-{z} \widetilde V_{\Phi})^{-1}=(I_{\widetilde\sH}-{z}  \widetilde V_2)^{-1}+\widetilde \g_2(1/z)(I_{\widetilde\sL_1}-\widetilde\Theta(\bar z)^*)^{-1}\widetilde\g_1(\bar{z})^*.
\end{equation}
By using \eqref{bgfields1},  \eqref{bgfields2}, \eqref{eq:5.4A} and the formula (see \cite{Bernstein}[Prop. 2.8.7, p.108])
\begin{equation}\label{eq:5.4B}
   (I_{\widetilde\sL_1}-\widetilde\Theta(\bar{z})^*)^{-1}=
    \begin{pmatrix}
      I_{\sL_1}&    -\Theta(\bar{z} )^*\\
      -\e(z)   &   I_{\sL_2}
    \end{pmatrix}^{-1}=
    \begin{pmatrix}
      *%(I_{\sL_1}-\Theta({z})\e({z}))^{-1}
          & * \\%\Theta({z})(I_{\sL_2}-\e({z})\Theta({z}))^{-1}\\
      \e({z})(I_{\sL_1}-\Theta(\bar{z} )^*\e(z))^{-1}& *%(I_{\sL_2}-\e({z})\Theta({z}))^{-1}
    \end{pmatrix}
\end{equation}
one obtains the equality
\begin{equation}\label{eq:5.5A}
  \begin{split}
    &(I_{\widetilde\sH}-z \widetilde V_{\Phi})^{-1}=(I_{\widetilde\sH}-z\widetilde V_2)^{-1}\\
    &+
    \begin{pmatrix}
      0&\g_2(z)\\
      \g^{-,T}_2(z)&0
    \end{pmatrix}
    \begin{pmatrix}
      *  & * \\%(I_{\sL_1}-\Theta\e)^{-1}     &\Theta(I_{\sL_2}-\e\Theta)^{-1}\\
      \e({z})(I_{\sL_1}-\Theta(\bar{z} )^*\e(z))^{-1}   & * %(I_{\sL_2}-\e\Theta)^{-1}
    \end{pmatrix}
    \begin{pmatrix}
      \g_1(\bar z)^*&0\\
      0&\g^{-,T}_1(\bar z)^*
    \end{pmatrix},
    \end{split}
\end{equation}
where  $*$ denotes blocks which are not used in further calculations.
%where, for simplicity, the argument ${z}$ is dropped in the last line.
Considering the compression of the formula~\eqref{eq:5.5A} to the subspace $\sH^+=\sH$ one arrives at~\eqref{E:GKR}.

Similarly, for $\bar z\in\widetilde\cD_e $
one obtains by~\eqref{eq:may03} and Corollary~\ref{cores1}
that  $z\in\r(\widetilde V_{\Phi}^{-1})$ and
\[%begin{equation}\label{eq:5.6A}
    (I_{\widetilde\sH}-z \widetilde V_{\Phi})^{-1}=(I_{\widetilde\sH}-{z} \widetilde V_1)^{-1}-\widetilde \g_1(1/{z})(I_{\widetilde\sL_1}-\widetilde\Theta(1/{z}))^{-1}\widetilde\g_2(\bar{z})^*.
\]%end{equation}
By \eqref{eq:5.4B}
\begin{equation}\label{eq:5.4C}
   \left(I_{\widetilde\sL_1}-\widetilde\Theta\left(\frac{1}{z}\right)\right)^{-1}=
%    \begin{pmatrix}
%      I_{\sL_1}&    -\Theta(\bar{z} )^*\\
%      -\e(z)   &   I_{\sL_2}
%    \end{pmatrix}^{-1}=
    \begin{pmatrix}
      * & \omega_{12}(z)\\
      * & *
    \end{pmatrix},
    \quad %\text{
\end{equation}
where    $\omega_{12}(z)=\varepsilon^T(1/z)(I_{\sL_2}-\Theta(1/z)\varepsilon^T(1/z))^{-1}$. Using \eqref{bgfields1},  \eqref{bgfields2} and \eqref{eq:5.4C} one gets
\begin{equation}\label{eq:5.8A}
  \begin{split}
    (I_{\widetilde\sH}-&{z} \widetilde V_{\Phi})^{-1}=(I_{\widetilde\sH}-{z}\widetilde V_1)^{-1}\\
    &-
    \begin{pmatrix}
      \g_1(1/z)&0\\
      0&\g^{-,T}_1(1/z)
    \end{pmatrix}
    \begin{pmatrix}
      * & \omega_{12}(z)\\
      * & *
      %           &\e^\#(I_{\sL_2}-\Theta^\#\e^\#)^{-1}\\
      %\Theta^\#(I_{\sL_1}-\e^\#\Theta^\#)^{-1}  &(I_{\sL_2}-\Theta^\#\e^\#)^{-1}
    \end{pmatrix}
    \begin{pmatrix}
      0&\g^{-,T}_2(\bar z)^*\\
      \g_2(\bar z)^*&0
    \end{pmatrix}.
    \end{split}
\end{equation}
The compression of the formula~\eqref{eq:5.8A} to the subspace $\sH^+$ gives the equality~\eqref{E:GKRe}.
%{\CB which reduces to~\eqref{E:GKRe} in view of the relation
%\[
%  \Theta({z})\sk({I-\e({z})\Theta({z})})^{-1}=\sk({I-\Theta({z})\e({z})})^{-1}\Theta({z}).
%\]
%}

{\bf 2.}  {\it Verification  that  every $(\widetilde\k-\k)$-regular generalized coresolvent of $V$
admits the representation~\eqref{E:GKR}, where  $\e(\cdot)\in \cS_{\widetilde\k-\k}({\sL}_1,{\sL}_2)$ and~\eqref{E:regularityGKR} holds.}

First observe that for ${z}\in \r(\widetilde V^{-1})$ and $g\in\sH$ the following relations hold:
  \begin{equation}\label{eq:may3b}
  \begin{pmatrix}
    {z}(I_{\wt\sH}-{z} \widetilde V)^{-1} g\\
    (I_{\wt\sH}-{z} \widetilde V)^{-1}g-g
  \end{pmatrix}\in \widetilde V,\quad
    \widehat{ \mathbf{K}}_{z} g:=\begin{pmatrix}
    {z}\mathbf{K}_{z} g\\
    \mathbf{K}_{z} g-g
  \end{pmatrix}\in V^{-[*]}.
  \end{equation}
Indeed, the first relation in~\eqref{eq:may3b} is self-evident and
hence for every $h\in\dom V$ one also has the equality
\begin{equation}\label{eq:may3c}
      [{z}(I_{\wt\sH}-{z} \widetilde V)^{-1} g,h]_\sH=[ -g+(I_{\wt\sH}-{z} \widetilde V)^{-1} g,Vh]_\sH.
\end{equation}
On the other hand,
\[%begin{equation*}
    \begin{split}
      [{z}\bK_{z} g,h]_\sH&+[(I_{\sH}-\bK_{z})g,V h]_\sH=
      [\bK_{z} g,(\ov{z} I_{\sH}-V)h]_\sH+[g,Vh]_\sH\\
      &=
     \left[(I_{\wt\sH}-{z} \widetilde V)^{-1}g,(\ov{z} I_{\sH}-V)h\right]_\sH+[g,Vh]_\sH\\
   &=[{z}(I_{\wt\sH}-{z} \widetilde V)^{-1}g,h]_\sH
     -[(I_{\wt\sH}-{z} \widetilde V)^{-1}g,Vh]_\sH+[g,Vh]_\sH=0,
    \end{split}
\]%end{equation*}
and here the last equality follows from \eqref{eq:may3c}.
This proves the second relation in~\eqref{eq:may3b}.

Next consider the linear relation
  \begin{equation}\label{eq:may3d}
V_{z}:=\left\{\widehat{ \mathbf{K}}_{z} g: \,g\in\sH\right\},\quad {z}\in\rho(\widetilde V^{-1})
  \end{equation}
in $\sH$. The linear relation $V_{z}$  is closed, since the assumptions
\[
{z} \mathbf{K}_{z} g_n\to f,\quad \mathbf{K}_{z} g_n-g_n\to f'\quad(n\to\infty)
\]
imply $g_n\to g:=\frac{1}{{z} }(f-{z} f')$ $({z} \ne 0)$ and hence
\[
\begin{pmatrix}
    f\\
    f'
  \end{pmatrix}=
    \begin{pmatrix}
    {z}\mathbf{K}_{z} g\\
    \mathbf{K}_{z} g-g
  \end{pmatrix}\in V_{z}.
\]

In order to construct a parametric representation of the proper extension $V_{z}$ let us introduce a closed subspace $\cN$ of $V_2$ such that
\[
V_2=V\dotplus \cN.
\]
Since $V^{-[*]}=V_1+V_2$, this implies
  \begin{equation}\label{eq:may3e}
V^{-[*]}=V_1  \dotplus\cN.
\end{equation}
Let $P_1$ and $P_2$ be projections onto the 1-st and the 2-nd components in $\sH^2$ and let $\Phi_j({z})$  be operator functions with values in $\cB(\cN,\sL_j)$ defined by
\[
\Phi_j({z})h=\Gamma_j\widehat{ \mathbf{K}}_{z} (P_1 h-{z} P_2h),\quad h\in\cN,\quad {z}\in\rho(\widetilde V^{-1}),\quad j=1,2.
\]
The values of $\Phi_j({z})$  belong to $\cB(\cN,\sL_j)$ due to formula \eqref{eq:may3b}.
Moreover, $\Phi_1({z})$ and $\Phi_2({z})$ satisfy~\eqref{eq:Invert0}, since the assumption
$\Phi_1({z})h=\Phi_2({z})h=0$ implies $h\in V\cap\cN=\{0\}$.

Now introduce the linear relation
\begin{equation}\label{eq:may3f}
\varepsilon({z}):=\left\{\binom{\Phi_1({z})h}{\Phi_2({z})h}:\,h\in\cN\right\}
\quad {z}\in\rho(\widetilde V^{-1}).
\end{equation}
Since $\ran(I_{\sH}-{z} V_2)=\sH$ for all ${z}\in\cD$ it follows from~\eqref{eq:may3d} and~\eqref{eq:may3f} that the linear relations $ V_{z}$ and $\varepsilon({z})$ are connected via~\eqref{eq:VPhi} and hence
\begin{equation}\label{eq:may9}
 V_{z}=V_{\varepsilon({z})}, \quad {z}\in \r(\widetilde V^{-1}).
\end{equation}
Since $(\sL,\G)$ is an ordinary boundary triple one concludes that $\varepsilon({z})$ is closed in $\sL_1\oplus\sL_2$,
and hence by Remark~\ref{rem:4.3} $\Phi_1({z})$ and $\Phi_2({z})$ satisfy~\eqref{eq:Invert};
cf.~\cite[Theorem 2.1]{B13}. Using \eqref{eq:may3b} and \eqref{eq:may9} one obtains
  \begin{equation}\label{eq:may3g}
\mathbf{K}_{z} g=(I_{\sH}-{z} V_{\varepsilon({z})})^{-1} g,\quad g\in\sH.
\end{equation}
Therefore ${z}\in\rho(V_{\varepsilon({z})}^{-1})$ for all
${z}\in\rho(\widetilde V^{-1})$ and since $(\sL,\G)$ is an ordinary boundary triple
Corollary~\ref{cores1} shows that $0\in\rho(I_{\sL_1}-\Theta({z})\e({z}))$ and that the following formula holds for ${z}\in\rho(\widetilde V^{-1})\cap \overline{\cD}$:
\begin{equation}\label{E:res0B}
  (I_{\sH}-{z} V_{\varepsilon({z})})^{-1}=(I_{\sH}-{z} V_2)^{-1}+\g_2(1/{z})\Phi_2({z})(\Phi_1({z})-\Theta(\bar{z})^*\Phi_2({z}))^{-1}\g_1(\bar{z})^*.
\end{equation}

It remains to show that $\e(\cdot)\in \cS_{\widetilde\k-\k}({\sL}_1,{\sL}_2)$.
For a choice of ${z}_j\in\rho(\widetilde V^{-1})\cap\dD$ and $g_j\in\sH$ denote
  \begin{equation}\label{eq:may3h}
\wh f_{j}=\binom{f_j}{f_j'}:=\widehat{ \mathbf{K}}_{{z}_j} g_j,\quad j=1,\dots,n,
\end{equation}
and let
  \begin{equation}\label{eq:may3i}
\Gamma\wh f_{j}=\binom{\Phi_1({z}_j)h_j}{\Phi_2({z}_j)h_j},\quad h_j\in\cN,\quad j=1,\dots,n.
\end{equation}
Then it follows from~\eqref{eq:may3h}, \eqref{eq:may3i} and~\eqref{E:1.1} that with $\xi_j\in\dC$,
  \begin{equation}\label{eq:may3j}
    \begin{split}
    \sum_{j,k=1}^n a_{j,k}\xi_j\overline{\xi_k}
     &:=\sum_{j,k=1}^n\frac{(\Phi_1({z}_j)h_j,\Phi_1({z}_k)h_k)_{\sL_1}-
     (\Phi_2({z}_j)h_j,\Phi_2({z}_k)h_k)_{\sL_2}}{1-{z}_j\overline{{z}}_k}
     {\xi_j\overline{\xi_k}}\\
      &=\sum_{j,k=1}^n\frac{(\Gamma_1\wh f_{j},\Gamma_1\wh f_{k})_{\sL_1}-
     (\Gamma_2\wh f_{j},\Gamma_2\wh f_{k})_{\sL_2}}{1-{z}_j\overline{{z}}_k}
     {\xi_j\overline{\xi_k}}\\
           &=\sum_{j,k=1}^n\left\{[ f_{j}, f_{k}]_{\sH}-
    [ f_{j}', f_{k}']_{\sH}\right\}
     \frac{\xi_j\overline{\xi_k}}{1-{z}_j\overline{{z}}_k}.
    \end{split}
  \end{equation}
Since
\[%begin{equation}\label{eq:may4f}
    \begin{split}
[ f_{j}, f_{k}]_{\sH}&-
    [ f_{j}', f_{k}']_{\sH}
    ={{z}_j\ov{z}_k}[\bK_{{z}_j}g_j,\bK_{{z}_k}g_k]_\sH-[\bK_{{z}_j}g_j-g_j,\bK_{{z}_k}g_k-g_k]_\sH\\
    &=({z}_j\ov{z}_k-1)[\bK_{{z}_j}g_j,\bK_{{z}_k}g_k]_\sH
    +[\bK_{{z}_j}g_j,g_k]_\sH+[g_j,\bK_{{z}_k}g_k]_\sH-[g_j,g_k]_\sH\\
    &=({1-{z}_j\ov{z}_k})[\sfR_{{z}_k}({z}_j)g_j,g_k]_\sH
    \end{split}
\]%end{equation}
the form in \eqref{eq:may3j} is reduced to
$$
    \sum_{j,k=1}^n a_{j,k}\xi_j\overline{\xi_k}=\sum_{j,k=1}^n [\sfR_{{z}_k}({z}_j)g_j,g_k]_\sH{\xi_j\overline{\xi_k}}.
$$
By Lemma~\ref{L:Ker} the form \eqref{eq:may3j} has at most $\widetilde\k-\k$ and for some choice of ${z}_j$, $g_j$, $j=1,\dots,n$,
exactly $\widetilde\k-\k$ negative squares.

In view of Lemma~3.2 in~\cite{DD09}
this implies that the operator $\Phi_1({z})$ is invertible for all ${z}\in\rho(\widetilde V^{-1})\cap\dD$ except $\widetilde\k-\k$ points and
\begin{equation}\label{eq:may4g}
    \e({z})=\Phi_2({z})\Phi_1({z})^{-1}\in \cS_{\widetilde
    \k-\k}(\sL_1,\sL_2).
\end{equation}
In view of \eqref{eq:may4g} the formula \eqref{E:res0B} can be rewritten as \eqref{E:GKR}.
\end{proof}
\begin{remark}
%\begin{itemize}
%  \item [(1)]
1) For a standard isometric operator in a Pontryagin (resp. Kre\u{\i}n)
  space similar formulas for generalized coresolvents  were found in \cite{L71,KL72,LS74} (resp.~\cite{DLS90}). For the case of a nonstandard isometric operator in a Pontryagin space see~%\cite{Sor85},
  \cite{N002}.
  An elegant proof of the formula for generalized resolvents of a nonstandard Pontryagin space symmetric operator with deficiency index  (1,1) given by H. de Snoo was presented in~\cite{KW98}.
  In \cite{B13} a description of regular generalized resolvents of a nonstandard Pontryagin space isometric operator was given by the method of boundary triples.
For a Hilbert space isometric operator this method was developed earlier in \cite{MM03} and applied to the proof of Kre\u{\i}n type resolvent formulas \eqref{E:GKR}, \eqref{E:GKRe}.

%  \item [(2)]
2) The extension $V_{\varepsilon({z})}$ appearing in~\eqref{eq:may9} is an analog of Shtraus extension, which was introduced in~\cite{Str} for the case of a symmetric operator. In view of \eqref{eq:may3g} the vector function $f_{z}={\bf K}_{z} g$ can be treated as a solution of the following ''abstract boundary value problem'' with ${z}$-dependent boundary conditions
      \[
      \wh f_{z}:=\binom{{z} f_{z}}{f_{z}-g}\in V^{-[*]},\quad \Gamma_2 \wh f_{z}=\varepsilon({z})\Gamma_1 \wh f_{z}.
      \]
%  \item [(3)]

3) In abstract interpolation problem considered in~\cite{KKhYu87} the crucial role was played by the Arov-Grossman formula for scattering matrices of unitary extensions of   isometric operators, \cite{ArGr83}. In \cite{B13b} the formula for generalized coresolvents was applied to the description of scattering matrices of unitary extensions of   Pontryagin space isometric operators which, in turn, was used in \cite{B14} for parametrization of solutions of an indefinite abstract interpolation problem, see also~\cite{D2001}. The present version of formula \eqref{E:GKR} will allow to consider $k$-regular  indefinite interpolation problems with the growth of index $\kappa$.
%\end{itemize}
\end{remark}

%\noindent \textbf{Acknowledgements.}


\begin{thebibliography}{99} %\small


\bibitem{ADRS97}
        D.~Alpay, A.~Dijksma, J.~Rovnyak, and H.S.V.~de~Snoo, \textit{
        Schur functions, operator colligations, and reproducing kernel
        Pontryagin spaces}, Oper. Theory: Adv. Appl., {\bf 96},
        Birkh\"auser Verlag, Basel, 1997.% -- 229 p.

\bibitem{Arens}
        R.~Arens, Operational calculus of linear relations, Pacific J.
        Math., {\bf 11} (1961), 9--23.

\bibitem{ArGr83}
        D.Z.~Arov and L.Z.~Grossman, The scattering matrices in the theory of the extensions of isometric operators, DAN USSR, {\bf 33} (1983), no. 4, 17--20.

\bibitem{AI86}
        T.Ya.~Azizov and I.S.~Iokhvidov, \textit{Linear operators in spaces
        with indefinite metric}, John Wiley and Sons, New York, 1989.

\bibitem{B13}
        D.V.~Baidiuk, On boundary triplets and generalized resolvents of an isometric operators in a Pontryagin space,
        Ukrainian Mathematical Bulletin, {\bf 10} (2013), 176--200.

\bibitem{B13b}
        D.V.~Baidiuk, Description of scattering matrices of unitary extensions of isometric operators in Pontryagin space, Math. Notes, {\bf 94} (2013), no. 6, 940--943.

\bibitem{B14}
        D.~Baidiuk, Abstract interpolation problem in generalized Schur classes, 	arXiv:1403.4038 [math.FA] (2014).

\bibitem{BDHS11}
        J.~Behrndt, V.~Derkach, S.~Hassi, and H.S.V.~de~Snoo,  A realization theorem for generalized Nevanlinna families. Oper. Matrices, {\bf 5} (2011), no. 4, 679--706.

\bibitem{Ben72}
        Ch.~Bennewitz, Symmetric relations on a Hilbert space, Lect. Notes Math., {\bf 280} (1972), 212--218.


\bibitem{Bernstein}
				D. ~Bernstein,  \textit{Matrix Mathematics}, Princeton University Press, 2009.

\bibitem{Bognar}
        J.~Bogn\'ar, \textit{Indefinite Inner Product Space},
        Springer-Verlag, Berlin, 1974.

\bibitem{Br}
         M.Sm.~Brodskii, Unitary operator colligations and their characteristic functions, Uspekhi Mat. Nauk, {\bf 33} (1978), no. 4, 141--168.

\bibitem{Cal39}
        J.W.~Calkin, Abstract symmetric boundary conditions, Trans.
        Amer. Math. Soc., {\bf 45} (1939), 369--442.

\bibitem{Der99}
        V.~Derkach, On  generalized resolvents of  Hermitian relations in
        Kre\u{\i}n spaces, J.of Math.Sci., {\bf 97} (1999), no. 5, 4420--4460.

\bibitem{D2001}
        V.A.~Derkach,
        On indefinite abstract interpolation problem,
        Methods of Funct. Analysis and Topology, {\bf 7} (2001), no. 4, 87--100.


\bibitem{DD09}
        V.~Derkach, H.~Dym,
        On linear fractional transformations associated with
        generalized $J$-inner matrix functions,
        Integ. Eq. Oper. Th., {\bf 65} (2009), 1--50.

\bibitem{DD19} V. Derkach and  H. Dym,
       Rigged de Branges-Pontryagin spaces and their application to  extensions and  embedding,
        J. Funct. Anal., {\bf 277} (2019), no. 1, 31--110.

\bibitem{DHMS1}
        V.A.~Derkach, S.~Hassi, M.M.~Malamud, and H.S.V.~de~Snoo,
        Generalized resolvents of symmetric operators and admissibility,
        Methods of Functional Analysis and Topology, {\bf 6} (2000), no. 3, 24--55.

\bibitem{DHMS06}
        V.A.~Derkach, S.~Hassi, M.M.~Malamud, and H.S.V.~de~Snoo, Boundary
        relations and Weyl families, Trans. Amer. Math. Soc., {\bf 358} (2006),
        5351--5400.

\bibitem{DHMS09}
        V.A.~Derkach, S.~Hassi, M.M.~Malamud, and H.S.V.~de~Snoo, Boundary
        relations and generalized resolvents of symmetric operators.  Russ.
        J. Math. Phys.  {\bf 16} (2009), no. 1, 17--60.

\bibitem{DHMS12}
				V.A.~Derkach, S.~Hassi, M.M.~Malamud, and H.S.V.~de~Snoo, Boundary triples and Weyl 	functions. Recent developments, London Mathematical Society Lecture Notes, {\bf 404} (2012), 161--220.

\bibitem{DHM17}
        V.A.~Derkach, S.~Hassi, and M.M.~Malamud,  Generalized boundary triples, Weyl functions and inverse problems, arXiv:1706.07948 (2017), 104 p.


\bibitem{DHM2020a}
        V.A.~Derkach, S.~Hassi, and M.M.~Malamud,
Generalized boundary triples, I. Some classes of isometric and unitary boundary pairs
and realization problems for subclasses of Nevanlinna functions, Math. Nachr., {\bf 293} (2020), no. 7, 1278--1327.

\bibitem{DHM2020b}
        V.A.~Derkach, S.~Hassi, and M.M.~Malamud, Generalized boundary triples, II.
Some applications of generalized boundary triples and form domain invariant Nevanlinna functions,
submitted for publication.

%\bibitem{DM91}
%        V.A.~Derkach and M.M.~Malamud, Generalized resolvents and the
%        boundary value problems for hermitian operators with gaps, J.
%        Funct. Anal., 95 (1991), 1--95.

\bibitem{DM87}
        V.A.~Derkach and M.M.~Malamud,
        On the Weyl function and Hermite operators with lacunae, Dokl. Akad. Nauk SSSR, {\bf 293} (1987), no. 5, 1041--1046.

\bibitem{DM95}
        V.A.~Derkach and M.M.~Malamud, The extension theory of hermitian
        operators and the moment problem, J. Math. Sciences, {\bf 73} (1995),
        141--242.

\bibitem{DM17}
        V.A.~Derkach and M.M.~Malamud, Extension theory of symmetric
        operators and boundary value problems, Proceeding of Institute of Math. of NAS of Ukraine, Kiev, {\bf 104} (2017), 141--242.

\bibitem{DLS90}
        A.~Dijksma, H.~Langer, and H.S.V.~de~Snoo, Generalized
        coresolvents of standard isometric relations and
        generalized resolvent of standard symmetric relations in Kre\u{\i}n
        space, Operator theory: Advances and Applicatios, Birkh\"auser
        Verlag, Basel, {\bf 48} (1990), 261--274.

\bibitem{HSSW07}
        S.~Hassi, A.~Sandovici, H.S.V.~de~Snoo, and H.~Winkler,
        A general factorization approach to the extension theory of nonnegative operators and relations,
        J. Operator Theory, {\bf 58} (2007), no. 2, 351--386.

\bibitem{IK65}
        I.S.~Iokhvidov and M.G.~Kre\u{\i}n, The spectral theory of operators in spaces with indefinite metric, I, Trudy Mosk. Mat. Obshch., {\bf 5} (1956), 367--432.

\bibitem{IKL}
        I.S.~Iohvidov, M.G.~Krein, and H.~Langer,  \textit{Introduction to the spectral theory of operators
        in spaces with an indefinite metric}, Mathematical Research, {\bf 9}, Akademie-Verlag, Berlin, 1982. %120 pp.

\bibitem{GK}
        I.~Gohberg and M.G.~Kre\u{\i}n, The basic propositions on defect
        numbers, root numbers and indices of linear operators, Uspekhi
        Mat. Nauk., {\bf 12} (1957), 43--118 (Russian) [English translation:
        Transl. Amer. Math. Soc., {\bf 13} (1960), no. 2, 185--264].

\bibitem{GG84}
        V.I.~Gorbachuk,  M.L.~Gorbachuk, \textit{Boundary problems for
        differential operator equations}, Naukova Dumka, Kiev, 1984;
        (English transl., Mathematics and its Applications, {\bf 48}. Kluwer Academic Publishers Group, Dordrecht, 1991.
\bibitem{KW98}
        M. Kaltenb\"ack, H. Woracek, Generalized resolvent matrices and spaces of analytic functions. Integral Equations Operator Theory, {\bf 32} (1998), no. 3, 282--318.

\bibitem{Koc75}
        A.N.~Kochubei,  On extensions of symmetric operators and symmetric
        binary relations, Matem. Zametki, {\bf 17} (1975), no. 1, 41--48.

\bibitem{KKhYu87}
        V.E.~Katsnel'son, A.Ya.~Kheifets, and P.M.~Yuditskii,
        An abstract interpolation problem and the theory of the extensions of isometric operators,
        Operators in Functional Spaces and Problems in Function Theory (Russian), Naukove Dumka, Kiev, (1987), 83--96.

\bibitem{KL71}
        M.G.~Kre\u{\i}n and H.~Langer, On defect subspaces and generalized
        resolvents of Hermitian operator in Pontryagin space, Funkts.
        Anal. i Prilozhen., {\bf 5} (1971), no. 2, 59--71; ibid. {\bf 5} (1971), no. 3,
        54--69 (Russian) [English translation: Funct. Anal. Appl., {\bf 5} (1971),
        136--146; ibid. {\bf 5} (1971), 217--228].
%\bibitem{KL73}
%        M.G.~Kre\u{\i}n and H.~Langer, \"{U}ber die $Q$-Funktion eines $\pi$-hermiteschen Operators im Raume $\Pi_\kappa$. (German) Acta Sci. Math. (Szeged), 34 (1973), 191--230.

\bibitem{KL72}
         M.G.~Kre\u{\i}n and H.~Langer,
        \"Uber die verallgemeinerten Resolventen und die characteristische
        Function eines isometrischen Operators im Raume $\Pi_\kappa$,
        Hilbert space Operators and Operator Algebras,
        Proc.Intern.Conf.,Tihany, (1970);
        Colloq.Math.Soc.Janos Bolyai, North--Holland, Amsterdam, {\bf 5} (1972), 353--399.

\bibitem{L71}
         H.~Langer, The generalized coresolvents of a $\pi$-isometric operator
         with unequal defect numbers, Funkcional. Anal. i Prilozen.,
         {\bf 5} (1971), no. 4, 73--75 (Russian).

\bibitem{LS74}
        H.~Langer, P. Sorjonen, Verallgemeinerte Resolventen hermitescher und isometrischer Operatoren im Pontrjaginraum. Ann. Acad. Sci. Fenn. Ser. A. I., (1974), no. 561, 1--45.

\bibitem{MM03}
        M.M.~Malamud and V.I.~Mogilevskii, The generalized resolvents of an isometric operator, Mat. Zametki, {\bf 73} (2003), no. 3, 460--466.

\bibitem{MM04}
        M.M.~Malamud and V.I.~Mogilevskii, The resolvent matrices and the spectral functions of an isometric operator, Docl. Akad. Nauk, {\bf 395} (2004), no. 1, 11--17.

%\bibitem{Mog06}
%         V.I.~Mogilevskii, Boundary triplets and Krein type resolvent formula for symmetric operators with unequal defect numbers. Methods Funct. Anal. Topology 12, no. 3 (2006), 258--280.

\bibitem{N001}
        O.~Nitz, Generalized resolvents of isometric linear relations
        in Pontryagin space, 1: Foundations, Operator theory:
        Advances and Applications, {\bf 118} (2000), 303--319.

\bibitem{N002}
        O.~Nitz, Generalized resolvents of isometric linear relations
        in Pontryagin space, 2: Krein--Langer formula, Methods of Functional Analysis and
        Topology, {\bf 6} (2000), no. 3, 72--96.

\bibitem{Sh76}
        Yu.L.~Shmul'yan, Theory of linear relations and spaces with indefinite metric, Funkts. Anal. Pril., {\bf 10} (1976), no. 1, 67--72.

\bibitem{Sor85}
        P.~Sorjonen, Generalized resolvents of an isometric operator in a Pontrjagin space. Z. Anal. Anwendungen, {\bf 4} (1985), no.6, 543--555.

\bibitem{Str}
        A.V.~\v{S}traus, Extensions and generalized resolvents of a
        symmetric operator which is not densely defined, Izv. Akad. Nauk.
        SSSR, Ser. Mat., {\bf 34} (1970), 175--202 (Russian) [English translation:
        Math. USSR-Izvestija, {\bf 4} (1970), 179--208].

\end{thebibliography}
\end{document}